\documentclass[12pt]{amsart}       %was 12pt
\usepackage{txfonts}
\usepackage{amssymb}
\usepackage{eucal}
\usepackage{bbm}
\usepackage{graphicx}
\usepackage{amsmath}
\usepackage{amscd}
\usepackage[all]{xy}           %xypic macro for latex2.09
\usepackage{amsfonts,latexsym}
\usepackage{xspace}
\usepackage{epsfig}
\usepackage{float}
\usepackage{color}
\usepackage{fancybox}
\usepackage{colordvi}
\usepackage{multicol}
\usepackage{colordvi}
\usepackage[active]{srcltx} %SRC Specials for DVI Searching
\usepackage{tikz}
\usepackage{graphicx}
\usepackage{shuffle}
\allowdisplaybreaks[4]
\usepackage{ifpdf}
\ifpdf
  \usepackage[colorlinks,final,backref=page,hyperindex]{hyperref}
\else
  \usepackage[colorlinks,final,backref=page,hyperindex,hypertex]{hyperref}
\fi
\usepackage{shuffle}

\ifpdf
  \hypersetup{pdftitle={Coloured shuffle compatibility, Hadamard products, and ask zeta functions},
    pdfauthor={Angela Carnevale, Vassilis Dionyssis Moustakas, Tobias Rossmann}
  }
\fi

\newcommand{\tddeux}{\begin{picture}(5,5)(0,-1)
\put(3,0){\circle*{3}}
\thicklines
\put(3,0){\line(0,1){5}}
\put(3,5){\circle*{3}}
\end{picture}}

\newcommand{\tddeuxa}[2]{\begin{picture}(5,5)(0,-1)
\put(3,1){\circle*{3}}
\thicklines
\put(3,1){\line(0,1){5}}
\put(3,6){\circle*{3}}
\put(6,-3){\tiny #1}
\put(6,3){\tiny #2}
\end{picture}}

\newcommand{\tdddeuxa}[3]{\begin{picture}(5,5)(0,-1)
\put(3,0){\circle*{3}}
\thicklines
\put(3,0){\line(0,1){5}}
\put(3,5){\circle*{3}}
\put(3,5){\line(0,1){5}}
\put(3,10){\circle*{3}}
\put(6,-4){\tiny #1}
\put(6,1){\tiny #2}
\put(6,7){\tiny #3}
\end{picture}}

\newcommand{\tdtroisuna}[3]{\begin{picture}(12,12)(-5,-1)
\put(3,0){\circle*{3}}
\put(-0.65,0){$\vee$}
\put(6,7){\circle*{3}}
\put(0,7){\circle*{3}}
\put(5,-2){\tiny #1}
\put(8,5){\tiny #2}
\put(-6,5){\tiny #3}
\end{picture}}

\newtheorem{theorem}{Theorem}[section]
\newtheorem{proposition}[theorem]{Proposition}
\newtheorem{lemma}[theorem]{Lemma}

\newtheorem{prop-def}{Proposition-Definition}[section]
\newtheorem{coro-def}{Corollary-Definition}[section]

\theoremstyle{definition}
\newtheorem{definition}[theorem]{Definition}
\newtheorem{remark}[theorem]{Remark}

\newcommand{\nc}{\newcommand}
\nc{\tred}[1]{\textcolor{red}{#1}}
\nc{\tblue}[1]{\textcolor{blue}{#1}}
\nc{\tgreen}[1]{\textcolor{green}{#1}}
\nc{\tpurple}[1]{\textcolor{purple}{#1}}
\nc{\btred}[1]{\textcolor{red}{\bf #1}}
\nc{\btblue}[1]{\textcolor{blue}{\bf #1}}
\nc{\btgreen}[1]{\textcolor{green}{\bf #1}}
\nc{\btpurple}[1]{\textcolor{purple}{\bf #1}}
\nc{\NN}{{\mathbb N}}
\nc{\ncsha}{{\mbox{\cyr X}^{\mathrm NC}}} \nc{\ncshao}{{\mbox{\cyr
X}^{\mathrm NC}_0}}

\newcommand{\delete}[1]{}

\nc{\mlabel}[1]{\label{#1}}
\nc{\mcite}[1]{\cite{#1}}
\nc{\mref}[1]{\ref{#1}}
\nc{\meqref}[1]{\eqref{#1}}
\nc{\mbibitem}[1]{\bibitem{#1}}
\delete{% Use the next lines to show names
\nc{\mlabel}[1]{\label{#1}{\hfill \hspace{1cm}{\bf{{\ }\hfill(#1)}}}}
\nc{\mcite}[1]{\cite{#1}{{\bf{{\ }(#1)}}}}
\nc{\mref}[1]{\ref{#1}{{\bf{{\ }(#1)}}}}
\nc{\meqref}[1]{\eqref{#1}{{\bf{{\ }(#1)}}}}
\nc{\mbibitem}[1]{\bibitem[\bf #1]{#1}}
}
\nc{\sha}{{\mbox{\cyr X}}}  %used to be \cyr
\newfont{\scyr}{wncyr10 scaled 550}
\nc{\ssha}{\mbox{\bf \scyr X}}
\nc{\shap}{{\mbox{\cyrs X}}} %sha as product
\nc{\shpr}{\diamond}    %Shuffle product
\nc{\shp}{\ast} \nc{\shplus}{\shpr^+}
\nc{\shprc}{\shpr_c}    %Cartier's product
\nc{\dep}{\mrm{dep}} \nc{\lc}{\lfloor} \nc{\rc}{\rfloor}
\nc{\db}{\leq_{\rm db}}
\nc{\cala}{{\mathcal A}} \nc{\calb}{{\mathcal B}}
\nc{\calc}{{\mathcal C}}
\nc{\cald}{{\mathcal D}} \nc{\cale}{{\mathcal E}}
\nc{\calf}{{\mathcal F}} \nc{\calg}{{\mathcal G}}
\nc{\calh}{{\mathcal H}} \nc{\cali}{{\mathcal I}}
\nc{\call}{{\mathcal L}} \nc{\calm}{{\mathcal M}}
\nc{\caln}{{\mathcal N}} \nc{\calo}{{\mathcal O}}
\nc{\calp}{{\mathcal P}} \nc{\calr}{{\mathcal R}}
\nc{\cals}{{\mathcal S}} \nc{\calt}{{\mathcal T}}
\nc{\calu}{{\mathcal U}} \nc{\calw}{{\mathcal W}} \nc{\calk}{{\mathcal K}}
\nc{\calx}{{\mathcal X}} \nc{\CA}{\mathcal{A}}

\nc{\fraka}{{\mathfrak a}} \nc{\frakA}{{\mathfrak A}}
\nc{\frakb}{{\mathfrak b}} \nc{\frakB}{{\mathfrak B}}
\nc{\frakc}{{\mathfrak c}}
\nc{\frakD}{{\mathfrak D}} \nc{\frakF}{\mathfrak{F}}
\nc{\frakf}{{\mathfrak f}} \nc{\frakg}{{\mathfrak g}}
\nc{\frakH}{{\mathfrak H}} \nc{\frakL}{{\mathfrak L}}
\nc{\frakM}{{\mathfrak M}} \nc{\bfrakM}{\overline{\frakM}}
\nc{\frakm}{{\mathfrak m}} \nc{\frakP}{{\mathfrak P}}
\nc{\frakN}{{\mathfrak N}} \nc{\frakp}{{\mathfrak p}}
\nc{\frakS}{{\mathfrak S}} \nc{\frakT}{\mathfrak{T}}
\nc{\frakX}{{\mathfrak X}}

\nc{\calT}{\mathcal{T}} \nc{\calF}{\mathcal{F}} \nc{\f}{\calF  }
\nc{\fn}{\calF^{N}} \nc{\fsn}{\calF^{\leq N}} \nc{\fsN}{\calF^{\leq N}}
\nc{\hck}{\mathcal{H}_{\rm CK}}
\nc{\hgl}{\mathcal{H}_{\rm GL}}
\nc{\RR}{\mathbb{R}} \nc{\ZZ}{\mathbb{Z}}  \nc{\etree}{\mathbbm{1}} \nc{\conc}{\cdot}
\nc{\hckn}{\hck^n} \nc{\hgln}{\hgl^n}
\nc{\hcksn}{\hck^{\leq n}} \nc{\hglsn}{\hgl^{\leq n}}
\nc{\hcksN}{\hck^{\leq N}} \nc{\hglsN}{\hgl^{\leq N}}
\nc{\X}{{\bf X}} \nc{\brp}{{\bf PBRP}_\alpha^{N}} \nc{\brpt}{{\bf PBRP}_\alpha^{3}}
\nc{\cbrpx}{\mathcal{D}^{N}_{\X;\alpha }}
\nc{\cbrpxd}{\mathcal{D}^{3}_{\X;\alpha }}
\nc{\cbrptx}{\mathcal{D}^{N}_{\tilde{X};\alpha }}
\nc{\cbrpxt}{\mathcal{D}^{3}_{\X;\alpha }}
\nc{\cbrptxt}{\mathcal{D}^{3}_{\tilde{X};\alpha }}
\nc{\Y}{{\bf Y}} \nc{\Z}{{\bf Z}} \nc{\W}{{\mathcal{Y}}}
\nc{\h}{{\bf H}}  \nc{\U}{{\bf U}} \nc{\bfone}{{\bf 1}}
\nc{\con}[2]{\ell_{#1}(#2)} \nc{\cont}[3]{\ell_{#1}^{#3}(#2)}
\nc{\fan}[1]{|||#1|||}
\nc{\hcksth}{\hck^{\leq 3}} \nc{\hglsth}{\hgl^{\leq 3}}
\nc{\hcksnd}{\hck^{\leq 2}} \nc{\hglsnd}{\hgl^{\leq 2}}
\nc{\ckr}{\Delta_{{\rm CK}}^{{\rm red}}}
\nc{\R}{\resizebox{0.8\width}{0.8\height}}

\nc{\hm}{\mathcal{H}_{\text {MKW}}}
\nc{\he}{\mathcal{H}_{\text {MKW}}^{*}}
\nc{\hmn}{\hm^N} \nc{\hen}{\mathcal{H}_{\text {MKW}}^{*N}}
\nc{\hmsn}{\hm^{\leq N}} \nc{\hesn}{\mathcal{H}_{\text {MKW}}^{*\leq N}} \nc{\hmsN}{\hm^{\leq N-1}} \nc{\hesN}{\mathcal{H}_{\text {MKW}}^{*\leq N}}
\nc{\hmsNd}{\hm^{\leq 2}} \nc{\hesnt}{\mathcal{H}_{\text {MKW}}^{*\leq 3}}
\nc{\pbrp}{{{\bf PBRP} }_\alpha^{N}}
\nc{\pbrps}{{{\bf PBRP}}_\alpha^{2}}
\nc{\pbrpt}{{{\bf PBRP}}_\alpha^{3}}
\nc{\hmsnd}{\hm^{\leq 2}} \nc{\hesnd}{\mathcal{H}_{\text {MKW}}^{*\leq 2}}
\nc{\hmsth}{\hm^{\leq 3}} \nc{\hesth}{\mathcal{H}_{\text {MKW}}^{*\leq 3}}
\nc{\EX}{\hat{\bf X}}
\nc{\hmr}{\Delta_{{\rm \text {MKW}}}^{{\rm red}}}
\nc{\hmnn}{\hm^n} \nc{\henn}{\mathcal{H}_{\text {MKW}}^{*n}}

\newcommand{\Int}{\ensuremath{{\textstyle\int}}}

\font\cyr=wncyr10 \font\cyrs=wncyr7
\nc{\xing}[1]{\textcolor{blue}{Xing:#1}}
\nc{\Dominique}[1]{\textcolor{blue}{Dominique: #1}}
\nc{\blue}[1]{\textcolor{blue}{#1}}
\nc{\revise}[1]{\textcolor{red}{#1}}
\nc{\red}[1]{\textcolor{red}{#1}}
\numberwithin{equation}{section}
\newcommand{\tdtroisunab}[1]{\begin{picture}(5,5)(0,-1)
\put(3,0){\circle*{3}}
\put(5,-2){\tiny #1}
\end{picture}}

\newcommand{\tdtroisunac}[3]{\begin{picture}(5,5)(0,-1)
\put(3,-4){\circle*{3}}
\thicklines
\put(3,-4){\line(0,1){6}}
\put(3,2){\circle*{3}}
\put(3,2){\line(0,1){6}}
\put(3,8){\circle*{3}}
\put(5,-6){\tiny #1}
\put(5,0){\tiny #2}
\put(5,6){\tiny #3}
\end{picture}}

\newcommand{\tdtroisunad}[3]{\begin{picture}(12,12)(-5,-1)
\put(3,-2){\circle*{3}}
\put(-0.65,-2){$\vee$}
\put(6,5){\circle*{3}}
\put(0,5){\circle*{3}}
\put(5,-4){\tiny #1}
\put(8,3){\tiny #2}
\put(-6,3){\tiny #3}
\end{picture}}

\newcommand{\tdtroisunae}[2]{\begin{picture}(5,5)(0,-1)
\put(3,-2){\circle*{3}}
\thicklines
\put(3,-2){\line(0,1){7}}
\put(3,5){\circle*{3}}
\put(5,-4){\tiny #1}
\put(5,3){\tiny #2}
\end{picture}}

\newcommand{\tdtroisunaf}[3]{\begin{picture}(5,5)(0,-1)
\put(8,-2){\circle*{3}}
\thicklines
\put(8,-2){\line(0,1){7}}
\put(8,5){\circle*{3}}
\put(-1,1.5){\circle*{3}}
\put(10,-4){\tiny #1}
\put(10,3){\tiny #2}
\put(1,-0.5){\tiny #3}
\end{picture}}

\newcommand{\tdtroisunaa}[4]{\begin{picture}(12,12)(-5,-1)
\thicklines
\put(3,-4){\circle*{3}}
\put(-0.85,-3){$\vee$}
\put(6,3){\circle*{3}}
\put(0,3){\circle*{3}}
\put(6,10){\circle*{3}}
\put(5,-6){\tiny #1}
\put(8,1){\tiny #2}
\put(6,3){\line(0,1){7}}
\put(-6,1){\tiny #3}
\put(8,8){\tiny #4}
\end{picture}}

\newcommand{\tdtroa}[5]{\begin{picture}(5,5)(0,-1)
\thicklines
\put(-1,-2){\circle*{3}}
\put(-1,-2){\line(0,1){7}}
\put(-1,5){\circle*{3}}
\put(1,-4){\tiny #1}
\put(1,3){\tiny #2}

\put(8,-2){\circle*{3}}
\put(8,-2){\line(0,1){7}}
\put(8,5){\circle*{3}}
\put(10,-4){\tiny #3}
\put(10,3){\tiny #4}

\put(17,1.5){\circle*{3}}
\put(19,-0.5){\tiny #5}
\end{picture}}

\newcommand{\tdtrob}[4]{\begin{picture}(5,5)(0,-1)
\thicklines
\put(-1,1.5){\circle*{3}}
\put(1,-0.5){\tiny #1}

\put(8,-2){\circle*{3}}
\put(8,-2){\line(0,1){7}}
\put(8,5){\circle*{3}}
\put(10,-4){\tiny #2}
\put(10,3){\tiny #3}

\put(17,1.5){\circle*{3}}
\put(19,-0.5){\tiny #4}
\end{picture}}

\newcommand{\tdtroc}[4]{\begin{picture}(5,5)(0,-1)
\thicklines
\put(-1,-2){\circle*{3}}
\put(-1,-2){\line(0,1){7}}
\put(-1,5){\circle*{3}}
\put(1,-4){\tiny #1}
\put(1,3){\tiny #2}

\put(8,1.5){\circle*{3}}
\put(10,-0.5){\tiny #3}

\put(17,1.5){\circle*{3}}
\put(19,-0.5){\tiny #4}
\end{picture}}

\newcommand{\tdtrod}[3]{\begin{picture}(5,5)(0,-1)
\thicklines
\put(-1,-2){\circle*{3}}
\put(-1,-2){\line(0,1){7}}
\put(-1,5){\circle*{3}}
\put(1,-4){\tiny #1}
\put(1,3){\tiny #2}

\put(8,1.5){\circle*{3}}
\put(10,-0.5){\tiny #3}
\end{picture}}

\newcommand{\tdtroe}[3]{\begin{picture}(5,5)(0,-1)
\put(-1,0){\circle*{3}}
\put(1,-2){\tiny #1}

\put(8,0){\circle*{3}}
\put(10,-2){\tiny #2}

\put(17,0){\circle*{3}}
\put(19,-2){\tiny #3}
\end{picture}}

\newcommand{\tdtrof}[2]{\begin{picture}(5,5)(0,-1)
\put(-1,0){\circle*{3}}
\put(1,-2){\tiny #1}

\put(8,0){\circle*{3}}
\put(10,-2){\tiny #2}
\end{picture}}

\newcommand{\tdtrog}[4]{\begin{picture}(5,5)(0,-1)
\thicklines
\put(-1,-2){\circle*{3}}
\put(-1,-2){\line(0,1){7}}
\put(-1,5){\circle*{3}}
\put(1,-4){\tiny #1}
\put(1,3){\tiny #2}

\put(8,-2){\circle*{3}}
\put(8,-2){\line(0,1){7}}
\put(8,5){\circle*{3}}
\put(10,-4){\tiny #3}
\put(10,3){\tiny #4}
\end{picture}}

\begin{document}
	
\title[Rough differential equations and planarly branched universal limit theorem]{Rough differential equations and planarly branched universal limit theorem}
%
%=========================================================================
%\author{o}
%\address{School of Mathematics and Statistics, Lanzhou, Gansu 730000, P.\,R. China}
%\email{}
%

\author{Xing Gao}
\address{School of Mathematics and Statistics, Lanzhou University
Lanzhou, 730000, China;
Gansu Provincial Research Center for Basic Disciplines of Mathematics
and Statistics, Lanzhou, 730070, China
}
\email{gaoxing@lzu.edu.cn}

\author{Nannan Li}
\address{School of Mathematics and Statistics, Lanzhou University
Lanzhou, 730000, China
}
\email{linn2024@lzu.edu.cn}

\author{Dominique Manchon}
\address{Laboratoire de Math\'{e}matiques Blaise Pascal, CNRS-Universit\'{e} Clermont-Auvergne, 3 place Vasar\'ely, CS 60026, 63178 Aubi\`ere, France}
\email{Dominique.Manchon@uca.fr}

%========================================================================
\date{\today}
%========================================================================
\begin{abstract}
We prove the universal limit theorem for planarly branched rough paths with roughness $\frac{1}{4} < \alpha \leq \frac{1}{3}$, using a fixed-point approach based on the Banach contraction principle. Planarly branched rough paths generalize both classical rough paths and branched rough paths, in a manner analogous to how post-Lie algebras generalize both Lie and pre-Lie algebras. In particular, the primitive elements in the graded dual of the Hopf algebra associated with planarly branched rough paths form a post-Lie algebra, subsuming the Lie and pre-Lie structures arising in the geometric and branched settings. This result extends the scope of the universal limit theorem, previously established for: (i) rough paths with roughness $\frac{1}{3} < \alpha \leq \frac{1}{2}$; (ii) geometric rough paths with $0 < \alpha \leq 1$; and (iii) branched rough paths with $0 < \alpha \leq 1$.
\end{abstract}

\makeatletter
\@namedef{subjclassname@2020}{\textup{2020} Mathematics Subject Classification}
\makeatother
\subjclass[2020]{
60L20, %Rough analysis		Rough paths
60L50, %Rough partial differential equations
60H99, %Stochastic analysis
34K50, %stochastic functional-differential equations
37H10,  %Generation, random and stochastic difference and differential equations
05C05.   %Trees
}

\keywords{planarly branched rough path; controlled planarly branched rough path; universal limit theorem; Banach fixed point theorem}

\maketitle

\tableofcontents

\setcounter{section}{0}
\allowdisplaybreaks

%%%%%========================================================================
\section{Introduction}
This paper generalizes the universal limit theorem to planarly branched rough paths with roughness $\frac{1}{4} < \alpha \leq \frac{1}{3}$, thereby extending its validity beyond the classical rough and branched rough path frameworks.

\subsection{Rough paths and (planarly) branched rough paths}
Rough paths were introduced by T.~Lyons in 1998~\cite{L98} as a framework for solving rough ordinary differential equations (RODEs) in integral form:
\begin{equation}\label{ivp}
Y_t = Y_0 + \int_0^t \sum_{i=1}^d f_i(Y_s)\, dX^i_s,
\end{equation}
where $ X: [0,T] \to \mathbb{R}^d $ is a driving signal and $ (f_i)_{i=1}^d $ is a collection of smooth vector fields on $ \mathbb{R}^n $. The theory provides a way to interpret such integrals when the path $ X $ is only $\alpha$-H\"older continuous for some $\alpha \in (0,1]$, or more generally of finite $1/\alpha$-variation, where classical integration theory fails.
To address this, rough paths are constructed as substitutes for Chen's iterated integrals of $ X $. More precisely, a rough path $ \mathbf{X} $ over $ X $ is a two-parameter family $ (\mathbf{X}_{st})_{0 \leq s \leq t \leq T} $ of linear functionals on the shuffle Hopf algebra generated by words over the alphabet $ [d] := \{1, \ldots, d\} $. These functionals encode the iterated integral structure of $ X $ and satisfy Chen's identity and suitable estimates.
A rough path $\mathbf{X}$ is called weakly geometric~\cite{FH2, Geng, Kel} if each $\mathbf{X}_{st}$ is a character on the shuffle Hopf algebra. It is called geometric~\cite{FH2, Geng, Kel} if there exists a sequence of smooth (bounded variation) paths $X^{(m)} $ whose canonical rough path lifts $ \mathbf{X}^{(m)}$ converge to $\mathbf{X} $ in the rough path topology (e.g., the H\"older or $p$-variation metric). By construction, geometric rough paths are weakly geometric~\cite[Lemma~1.14]{Geng}, but the converse may fail in infinite-dimensional settings~\cite{FV2, Geng}. Thus, the distinction between weakly geometric and geometric rough paths becomes essential when extending the theory beyond finite-dimensional Euclidean spaces.

Whereas the signature of a differentiable path (i.e. the collection of its iterated Chen integrals) is uniquely defined, rough paths above a given $\alpha$-H\" older-continuous path always exist but are not unique when $\alpha <1/2$~\cite{LV07, TZ20}. As a consequence of the sewing lemma~\cite{FP06, Gub04}, it is however sufficient to evaluate the rough path on words of length smaller or equal to $N=\lfloor1/\alpha\rfloor$, the evaluations on longer words coming for free.

Branched rough paths, first introduced by M.~Gubinelli~\cite{Gub1}, are defined analogously to T.~Lyons' rough paths, except that the shuffle Hopf algebra of words with letters in $[d] := \{1, \ldots, d\}$ is replaced by the Connes-Kreimer Hopf algebra of rooted forests decorated by $[d]$. Branched rough paths have been studied extensively; see~\cite{CFMK, D22} and references therein. Notably, they play a significant role in the theory of regularity structures~\cite{BHZ, Hair14, LOTT}, developed by M.~Hairer to address stochastic partial differential equations (SPDEs).
By replacing the Connes-Kreimer Hopf algebra with the Munthe-Kaas-Wright Hopf algebra of planar rooted forests~\cite{MKW}, one arrives at the notion of planarly branched rough paths~\cite{CFMK, KL23}, which have found applications in solving RODEs on homogeneous spaces~\cite{CFMK}.
Planarly branched rough paths (resp.\ rough paths, resp.\ branched rough paths) are based on the Munthe-Kaas-Wright (resp.\ shuffle, resp.\ Connes-Kreimer) Hopf algebra, whose graded dual Hopf algebra has primitive elements forming a post-Lie (resp.\ Lie, resp.\ pre-Lie) algebra. In this sense, planarly branched rough paths generalize both rough paths and branched rough paths, just as post-Lie algebras generalize both Lie and pre-Lie algebras.

\subsection{Controlled rough paths and universal limit theorem}
Controlled rough paths were introduced by M.~Gubinelli, both in the shuffle setting~\cite{Gub04} and the branched setting~\cite{Gub1} (see also~\cite{Kel}). Given a rough path $\mathbf{X}$ above a path $X$, an $\mathbf{X}$-controlled rough path on $\mathbb{R}^n$ is a collection $\mathbf{Y} = (\mathbf{Y}^1, \ldots, \mathbf{Y}^n)$ of $n$ functions from $[0,T]$ into a suitably truncated Hopf algebra, satisfying the expansion
\[
\langle \tau, \mathbf{Y}_t^i \rangle = \langle \mathbf{X}_{st} \star \tau, \mathbf{Y}_s^i \rangle + \text{small remainder},
\]
for all basis elements $\tau$ and $i = 1, \ldots, n$. The space of $\mathbf{X}$-controlled rough paths forms a Banach space. Furthermore, any smooth function $f: \mathbb{R}^n \to \mathbb{R}^n$ gives rise to a unique controlled rough path $f(\mathbf{Y})$ above the path $f \circ Y$. This allows one to lift the differential equation~\eqref{ivp} to the level of controlled rough paths:
\vspace{-0.1in}
\begin{equation}\label{ivplifted}
\mathbf{Y}_t = \mathbf{Y}_0 + \int_0^t \sum_{i=1}^d f_i(\mathbf{Y}_s)\, d\mathbf{X}^i_s.
\end{equation}
In this framework, integrals of the form $\int_0^t \mathbf{Z}_s\, dX_s^i$ can be rigorously defined for any $\mathbf{X}$-controlled rough path $\mathbf{Z}$. Applying the standard Banach fixed point theorem to the mapping
\begin{equation}\label{eq:firstm}
\mathcal{M}: \mathbf{Y} \longmapsto \mathbf{Y}_0 + \int_0^\bullet \sum_{i=1}^d f_i(\mathbf{Y}_s)\, dX_s^i
\end{equation}
establishes existence and uniqueness of the solution to~\eqref{ivplifted}. This result is known as the \emph{universal limit theorem}.

The universal limit theorem, originally formulated by T.~Lyons~\cite{L98} and named by P.~Malliavin~\cite{M97}, is a cornerstone of rough path theory. It was first proved for weakly geometric rough paths with roughness $\frac{1}{3} < \alpha \leq \frac{1}{2}$; see also~\cite{FV2, Geng}. In the same year, Lyons and Qian~\cite{LQ98} extended the result to general rough paths within the same roughness regime. Later, Lyons and Qian~\cite{LQ02} extended the result to all geometric rough paths with roughness $0 < \alpha \leq 1$. More recently, Friz and Zhang~\cite{FH18} established the universal limit theorem for all branched rough paths with roughness $0 < \alpha \leq 1$ in the $1/\alpha$-variation topology.
In this paper, we have chosen to stick to the $\alpha$-H\"{o}lder case, different from bounded $1/\alpha$-variation case but classical by now ~\cite{CFMK, FH2, Gub1, Geng, L98}.
%Within this framework, we extend the universal limit theorem from rough paths and branched rough paths to the setting of planarly branched rough paths with roughness $\frac{1}{4} < \alpha \leq \frac{1}{3}$.

It is noteworthy that the underlying idea of the universal limit theorem had already been applied to stochastic differential equations (SDEs) prior to the development of rough path theory~\cite{L98}, since both It$\hat{\mathrm{o}}$-type and Stratonovich-type SDEs can be viewed as special instances of the rough differential equation~\eqref{ivp}. In particular, E.~Wong and M.~Zakai~\cite{WZ1, WZ2} were the first to establish a version of the universal limit theorem for Stratonovich-type SDEs. For further developments in this direction, we refer the reader to~\cite{D07, M97, Pri}.
The universal limit theorem also has important applications beyond stochastic analysis. For instance, it has been used in the study of weak Poincar\'e inequalities~\cite{A04}, and in the semiclassical approximation of Schr\"odinger operators on path spaces via Schr\"odinger operators with quadratic potentials on Wiener spaces~\cite{A07}.

\subsection{Motivation and outline of the paper}
Our motivation is twofold. On the one hand, planarly branched rough paths serve as natural generalizations of both rough paths and branched rough paths. This broader framework allows one to seek solutions of RODEs on manifolds whose tangent vector fields form a post-Lie algebra---extending the classical setting where rough paths correspond to Lie algebras and branched rough paths to pre-Lie algebras. The following diagram illustrates the relationships among various types of rough paths; see~\mcite{Man25} for further details:
\begin{displaymath}
\xymatrix@C=2.2cm@R=1cm{
\mathcal{H}_{{\rm LOT}}^A \ar@{^(->}[r] \ar[d]_{{\bf X}_{st}^{{\rm LOT}}}& \mathcal{H}_{{\rm CK}}^A  \ar@<0.1cm>@{^(->}[d]\ar@{->>}[r] \ar[ld]_{{\bf X}_{st}^{{\rm CK}}}& \mathcal{H}_\shuffle^A\ar[dll]_{{\bf X}_{st}^{\shuffle}}\\
\mathbb{R}&\mathcal{H}_{{\rm MKW}}^A\ar@{->>}[ur]\ar[l]^{{\bf X}_{st}^{{\rm MKW}}}&
}
\end{displaymath}
Here, $\mathcal{H}_{\mathrm{LOT}}^{A}$, $\mathcal{H}_{\mathrm{CK}}^{A}$, $\mathcal{H}_{\shuffle}^{A}$, and $\mathcal{H}_{\mathrm{MKW}}^{A}$ denote, respectively, the $A$-decorated Linares-Otto-Tempelmayr Hopf algebra~\cite{BH,LOT,ZGM}, the Connes-Kreimer Hopf algebra, the shuffle Hopf algebra, and the Munthe-Kaas-Wright Hopf algebra. The associated rough paths are denoted by $\mathbf{X}^{\mathrm{LOT}}$, $\mathbf{X}^{\mathrm{CK}}$, $\mathbf{X}^{\shuffle}$, and $\mathbf{X}^{\mathrm{MKW}}$.
On the other hand, with regard to the universal limit theorem, we extend the known roughness regime from $\frac{1}{3} < \alpha \leq \frac{1}{2}$ (as in the classical theory of rough paths) to $\frac{1}{4} < \alpha \leq \frac{1}{3}$ in the setting of planarly branched rough paths.

In this paper, we focus on the case $\frac{1}{4} < \alpha \leq \frac{1}{3}$, corresponding to a truncation level $N = 3$, and establish the universal limit theorem in the framework of planarly branched rough paths using the Banach fixed point theorem. The \emph{main difficulty} lies in verifying two key properties: (1) the composition of a controlled planarly branched rough path with a regular function remains a controlled planarly branched rough path, and (2) the fixed point map is contractive. Both aspects require substantial and intricate computations.
To the best of our knowledge, this is the first work combining the universal limit theorem with the theory of planarly branched rough paths. A generalization to arbitrary truncation levels $N$, encompassing a broader class of connected graded Hopf algebras, will be addressed in a forthcoming paper.

\vspace{0.5em}
\noindent{\bf Outline.} The structure of the paper is as follows. In Subsection~\ref{sec:2.1}, we recall the definition of a planarly branched rough path and introduce the notion of a controlled planarly branched rough path, which is central to our analysis.
Subsection~\ref{sec2.2} is devoted to establishing an upper bound on the norm of a new controlled planarly branched rough path obtained by composing a controlled planarly branched rough path with a sufficiently regular function (Theorem~\mref{thm:stab1}).
In Subsection~\ref{sec2.3}, we prove the stability of such compositions within the space of controlled planarly branched rough paths (Theorem~\mref{thm:stab2}).

Subsection~\ref{sec3.1} introduces the rough integral of a controlled planarly branched rough path against a planarly branched rough path.
In Subsection~\ref{sec3.2}, we prove the universal limit theorem using the Banach fixed point theorem. We begin by deriving an upper bound (Lemma~\mref{lem:jstab1}) and establishing continuity (Lemma~\mref{lem:jstab2}) of the fixed point map $\mathcal{M}$ defined in~\eqref{eq:firstm}, applied to a controlled planarly branched rough path $\mathbf{Y}$.
We then prove the local existence and uniqueness of solutions to the RODE~\eqref{ivp} (Theorem~\mref{thm:leau}). Finally, with all preparatory results in place, we prove the universal limit theorem for planarly branched rough paths (Theorem~\mref{thm:unilthm}).

\vspace{0.5em}
\noindent \textbf{Notation.} Throughout this paper, we fix two positive integers $d$ and $n$, where $d$ denotes the number of vector fields driving equation~\eqref{ivp}, and $n$ is the dimension of the target space in which the solution takes values. We write $\RR$ for the field of real numbers. The notation $M(\bullet)$ denotes a universal function, increasing in all of its arguments, which may vary from line to line.

\section{An upper bound and stability of the composition with regular functions}\label{sec:rpath}
In this section, we begin by recalling the notion of planarly branched rough paths and introduce the concept of controlled planarly branched rough paths, which serve as the central objects of study in this paper. We then establish an upper bound for the composition of controlled planarly branched rough paths with regular functions and prove the stability of such compositions within the corresponding functional framework.

\subsection{Planarly branched rough paths and controlled planarly branched rough paths}\label{sec:2.1}
%%%
Let $A$ be a finite set with $d$ elements. Denote by $\calT$ (resp. $\f$) the set of $A$-decorated planar rooted trees (resp. forests).
For $\tau\in \f$, define $|\tau|$ to be the number of vertices of $\tau$.
There is a unique forest $\etree$ with $|\etree|=0$, called the empty forest.
Denote respectively by
\begin{align*}
\hm:=&\ (\mathbb{R} \calF, \shuffle, \etree, \Delta_{\text {MKW}},\etree^*)=\bigoplus_{n\geq 0} \hmnn,\\
\he:=&\ (\mathbb{R} \f,\star, \etree, \Delta_{\shuffle},\etree^*)= \bigoplus_{n\geq 0} \henn,
\end{align*}
the Munthe-Kaas-Wright graded Hopf algebra~\cite{MKW} and its graded dual Hopf algebra under the pairing
$
\langle\tau,\tau'\rangle=\delta_\tau^{\tau'},
$
where $\shuffle$ is the shuffle product and $\delta $ is the Kronecker delta. Here for $ n, N\geq 0$, set $${\mathcal{F}}^{n}:=\{\tau\in \f \mid |\tau| = n \},\quad  {\mathcal{F}}^{\le N}:=\{\tau\in \f \mid |\tau| \leq N \}, \quad \hmnn:= \mathbb{R} {\mathcal{F}}^{n}=:\henn.$$
The coproduct $\Delta_{\text {MKW}}$ and the multiplication $\star$ are characterized respectively by the left admissible cut and the left grafting~\cite{ER25,MKW}.
We recall here the coproduct $\Delta_{\text {MKW}}$ in detail for later use. Let $\tau$ be an $A$-decorated planarly rooted tree in $\calT$. A subset $c$ of edges in $\tau$ is called {\bf a left-admissible cut} if
\begin{enumerate}
\item For every edge $e$ in $c$, every other edge outgoing the same vertex as $e$ and to the left of $e$ in the planar embedding, is also in $c$.

\item Every path from the root of $\tau$ to any of its leaves contains at most one edge in $c$.
\end{enumerate}
Let $c$ be a left-admissible cut of $\tau$ and remove the edges in $c$ from $\tau$.
Denote by $T^c(\tau)$ the trunk part, which is the connected component containing the root.
Let $P^c(\tau)$ be the pruned part, which is obtained by keeping the planar order of those subtrees that are cut off from the same vertex and then shuffling together the resulting planarly rooted forests that are cut off from
different vertices.
The coproduct $\Delta_{\text {MKW}}$ is described as
\begin{align*}
\Delta_{\text {MKW}}(\tau):=&\ \sum_{c \text{\ left-admissible cut}}P^c(\tau)\otimes T^c(\tau),\quad \forall \tau \in \calT,\\
\Delta_{\text {MKW}}(\omega ):=&\ (\text{id}\otimes B_{-})(\Delta_{\text {MKW}}(B_{+}(\omega))-B_{+}(\omega)\otimes \etree),\quad  \forall \omega \in \calF,
\end{align*}
where $B_{+}$ is the usual grafting operator, grafting all planarly rooted trees in the input planarly rooted forest onto a new root (no need decoration) and $B_{-}$ is the inverse of $B_{+}$. We expose two examples~\cite{ER25} for better understanding:
\begin{align*}
\Delta_{\text {MKW}}(\ \tdtroisunaa{$a$}{$c$}{$b$}{$d$}\ \ \ )=&\ \etree\otimes\, \tdtroisunaa{$a$}{$c$}{$b$}{$d$}\ \ +\tdtroisunab{$b$}\ \,\otimes \tdtroisunac{$a$}{$c$}{$d$}\ \ +\tdtroisunab{$d$}\ \,\otimes\, \tdtroisunad{$a$}{$c$}{$b$}\ \,+\tdtroisunab{$b$}\ \,\shuffle \tdtroisunab{$d$}\ \,\otimes \tdtroisunae{$a$}{$c$}\ \,+\ \tdtroisunaf{$c$}{$d$}{$b$}\ \ \ \,\otimes \tdtroisunab{$a$}\ +\,\tdtroisunaa{$a$}{$c$}{$b$}{$d$}\ \ \otimes\etree,\\
\Delta_{\text {MKW}}(\ \ \tdtroa{$a$}{$b$}{$c$}{$d$}{$f$}\ \ \ \ \ \ \ )=&\ \etree\otimes\,\, \tdtroa{$a$}{$b$}{$c$}{$d$}{$f$}\ \ \ \ \ \ \ +\tdtroisunab{$b$}\ \,\otimes\  \tdtrob{$a$}{$c$}{$d$}{$f$}\ \ \ \ \ \ \, +\tdtroisunab{$d$}\ \,\otimes\ \tdtroc{$a$}{$b$}{$c$}{$f$}\ \ \ \ \ \ \, +\tdtroisunae{$a$}{$b$}\ \,\otimes\ \tdtrod{$c$}{$d$}{$f$}\ \ \, \ +\tdtroisunab{$b$}\ \,\shuffle \tdtroisunab{$d$}\ \,\,\otimes \ \tdtroe{$a$}{$c$}{$f$}\\
&\ +\tdtroisunae{$a$}{$b$}\ \,\shuffle \tdtroisunab{$d$}\ \,\,\otimes\, \, \tdtrof{$c$}{$f$}\ \ \ \ +\ \tdtrog{$a$}{$b$}{$c$}{$d$}\ \ \ \ \otimes\tdtroisunab{$f$}\ \,+\ \tdtroa{$a$}{$b$}{$c$}{$d$}{$f$}\ \ \ \ \ \ \ \otimes\etree.
\end{align*}

Both $$\hmsn := \bigoplus_{0\leq k\leq N} \hm^k, \quad \hesn := \bigoplus_{0\leq k\leq N} \mathcal{H}_{\text {MKW}}^{*k}$$ are endowed with a structure of connected graded (and finite-dimensional) algebra and  coalgebra as follows: the vector subspace $I_N:= \RR\{\tau\in \f \mid |\tau| > N \}$ of $\hm$ is a graded ideal (but not a bi-ideal), hence $\hm/I_N$ is a graded algebra.
On the other hand, the restriction of the projection $\pi_N:\hm\to\hskip -9pt\to\hm/I_N$ to the subcoalgebra $\hmsn$ is an isomorphism of graded vector spaces.
The graded algebra structure of $\hm/I_N$ can therefore be transported on $\hmsn$, making  it both an algebra and a coalgebra, denoted by
$$\hmsn= ( \mathbb{R} \fsN, \shuffle, \etree, \Delta_{\text {MKW}},\etree^*)$$
by a slight abuse of notations. Since $\dim \hmsn$ is finite, we also have the dual algebra/coalgebra
$$\hesn = (\hmsn)^*=(\mathbb{R} \fsN,\star, \etree^*, \Delta_{\shuffle},\etree)$$ under the above pairing. The following is
 the concept of planarly branched rough path~\cite[Definition 4.1]{CFMK}.

\begin{definition}
Let $\alpha \in(0,1]$ and $N= \lfloor \frac{1}{\alpha} \rfloor$. An {\bf $\alpha$-H\"{o}lder planarly branched rough path} is a map $\X:[0,T]^2\rightarrow \hesn$ such that
\begin{enumerate}
\item  $\X_{s,u} \star \X_{u,t} =\X_{s,t}$,\mlabel{it:bitem1}

\item $\langle \X_{s,t}, \sigma \shuffle \tau\rangle  = \langle \X_{s,t}, \sigma\rangle\langle \X_{s,t}, \tau\rangle $,\mlabel{it:bitem2}

\item $\sup\limits_{s\ne t\in [ 0,T ] } \frac{| \left \langle \X_{s,t}, \tau \right \rangle  |}{| t-s | ^{| \tau |\alpha  } } <\infty$,\quad $\forall \sigma,\tau\in \fsN$ and $s, t, u\in \left[ 0, T\right] $. \mlabel{it:bitem3}
\end{enumerate}
Further for a path $X=(X^1,\ \ldots,X^d):[0,T]\rightarrow \mathbb{R}^d$, if
$$\langle \X_{s,t} , \bullet_i \rangle=X_{t}^i - X_{s}^i\, \text{ for each }\, i=1,\ \ldots, d, $$
then we call $\X$ a {\bf planarly branched rough path } above $X$.
\mlabel{def:pbrp}
\end{definition}

Estimate \eqref{it:bitem3} in Definition~\mref{def:pbrp}~ is equivalently written as $| \left \langle \X_{s,t}, \tau \right \rangle  | = O(| t-s | ^{| \tau |\alpha  } )$.
We will use both notations interchangeably in the sequel of the paper.\\

Let $\brp $ be the set of $\alpha$-H\"{o}lder planarly branched rough paths. For any map $\X:[0,T]^2\rightarrow \hesn$ and $\tau\in \f$, define
\begin{equation}
\|\X^{\tau}\|_{|\tau|\alpha} : =\sup_{s\ne t\in [ 0,T ] } \frac{| \X^\tau_{s,t}  |}{| t-s | ^{|\tau |\alpha  } }.
\mlabel{eq:normx}
\end{equation}
Now endow the set $\brp$ with the distance
\begin{equation}\label{distance-holder}
d_{\alpha } ( \X,\tilde{\X}  ) : = \sum_{\tau\in \fsN\setminus \{\etree\}}\| \X^{\tau}-\tilde{\X}^{\tau}    \|_{| \tau |\alpha  }, \quad \forall \X, \tilde{\X}\in \brp
\end{equation}
and define
\begin{equation}
\fan{\X}_{\alpha }  : = d_{\alpha } ( \X, \bfone  )=\sum_{\tau\in \fsN\setminus \{\etree\}}\| \X^{\tau}-\bfone ^{\tau}    \|_{| \tau |\alpha  } = \sum_{\tau\in \fsN\setminus \{\etree\} } \| \X^{\tau}    \|_{| \tau |\alpha  },  \quad \forall \X\in \brp
\mlabel{eq:distx}
\end{equation}
where $\bfone:[0,T]^2\rightarrow \hesn, \, (s,t)\mapsto  \etree.$\\

We are going to propose the planar version of
controlled branched rough paths~\mcite{Gub1}.
Let $\hmr$ be the reduced coproduct of $\hm$. That is,
$\hmr(\sigma) := \Delta_{\text {MKW}}(\sigma)-1\otimes \sigma-\sigma\otimes 1$ for $\sigma\in \calF$.
Write
$$\hmr(\sigma)=\sum_{\sigma',\sigma''\in \calF}c(\sigma; \sigma', \sigma'')\sigma' \otimes\sigma''.$$
Here $\sigma'$ is the remaining part containing the root and $\sigma''$ is the pruned part.
Let $N= \lfloor \frac{1}{\alpha } \rfloor$, $\X\in  \brp $ and $\Y: [0,T]\rightarrow \hmsN$ be a path. Denote
\begin{equation}
\begin{aligned}
R\Y^{\etree}_{s,t}:=&\ \Y_{s,t}^{\etree}-\sum_{\tau\in \mathcal{F}^{\leq (N-1)}\setminus\{\etree\}}\Y^{\tau}_{s}\X^{\tau}_{s,t}, \\
R\Y^{\tau}_{s,t}:=&\ \Y_{s,t}^{\tau}-\sum_{\sigma,\rho\in\mathcal{F}^{\leq (N-1)}} c(\sigma; \tau, \rho)\Y^{\sigma}_{s}\X^{\rho}_{s,t}, \quad \forall \tau\in \mathcal{F}^{\leq (N-1)}\setminus \{\etree\},
\end{aligned}
\mlabel{eq:remain}
\end{equation}
where
$$\X^\tau_{s,t}:= \langle \X_{s,t}, \tau\rangle,\quad \Y^\tau_{t}:= \langle \tau, \Y_{t}\rangle,\quad \Y_{s,t}^{\tau}:=\Y_{t}^{\tau}-\Y_{s}^{\tau}.$$

\begin{definition}
Let $\X \in \brp$. A path $\Y: [0,T]\rightarrow \hmsN $ is called an {\bf $\X$-controlled planarly branched rough path} if
\begin{equation}
\|R\Y^\etree\|_{N\alpha}:=\sup_{s\ne t\in [ 0,T ] } \frac{|R\Y^{\etree}_{s,t} |}{| t-s |^{N\alpha}} < \infty,\quad \|R\Y^\tau\|_{(N-|\tau|)\alpha}:=\sup_{s\ne t\in [ 0,T ] } \frac{|R\Y^{\tau}_{s,t} |}{| t-s |^{(N-|\tau|)\alpha}} < \infty,\quad \forall \tau\in \mathcal{F}^{\leq (N-1)}\setminus \{\etree\}.
\mlabel{eq:remainry}
\end{equation}
Further for a path $Y:[0,T]\rightarrow \mathbb{R}$, if
$\langle \etree, \Y_{t} \rangle=Y_{t},$
then we call $\Y$ above $Y$ or $Y$ can be lifted to $\Y$.
\mlabel{defn:cbrp}
\end{definition}

\begin{remark}
 ~(\ref{eq:remain}) is equivalent to equation
\begin{equation*}
R\Y^{\tau}_{s,t}=\langle \tau, \Y_t\rangle-\langle \tau\star\X_{s,t}, \Y_s\rangle, \quad \forall \tau\in \mathcal{F}^{\leq (N-1)}.
\end{equation*}
Indeed, if $\tau=\etree$, then
\begin{align*}
&\ \langle \etree, \Y_t\rangle-\langle \etree\star\X_{s,t}, \Y_s\rangle
= \langle \etree, \Y_t\rangle-\langle \X_{s,t}, \Y_s\rangle
= \Y_{t}^{\etree}-\sum_{\tau\in \mathcal{F}^{\leq (N-1)}}\Y^{\tau}_{s}\X^{\tau}_{s,t}\\
=&\ \Y_{t}^{\etree}-\Y_{s}^{\etree}-\sum_{\tau\in \mathcal{F}^{\leq (N-1)}\setminus \{\etree\}}\Y^{\tau}_{s}\X^{\tau}_{s,t}
= \Y_{s,t}^{\etree}-\sum_{\tau\in \mathcal{F}^{\leq (N-1)}\setminus \{\etree\}}\Y^{\tau}_{s}\X^{\tau}_{s,t}.
\end{align*}
If $\tau\ne\etree$, then
\begin{align*}
\langle \tau, \Y_t\rangle-\langle \tau\star\X_{s,t}, \Y_s\rangle
=&\ \langle \tau, \Y_t\rangle-\langle \tau\otimes\X_{s,t}, \Delta_{\rm \text {MKW}} (\Y_s)\rangle\\
=&\ \langle \tau, \Y_t\rangle-\langle \tau, \Y_s\rangle-\langle \tau\otimes\X_{s,t}, \hmr(\Y_s)\rangle\\
=&\ \langle \tau, \Y_{s,t}\rangle-\left\langle \tau\otimes\X_{s,t}, \hmr\Big(\sum_{\sigma\in \mathcal{F}^{\leq (N-1)}}  \Y^{\sigma}_{s} \sigma \Big) \right\rangle\\
%
%=&\ \langle \tau, \Y_{s,t}\rangle-\sum_{\sigma\in \mathcal{F}^{\leq (N-1)}}\Y^{\sigma}_{s}\langle \tau\otimes\X_{s,t}, \ckr(\sigma)\rangle\\
%
=&\ \langle \tau, \Y_{s,t}\rangle-\sum_{\sigma\in \mathcal{F}^{\leq (N-1)}}\sum_{\sigma',\sigma''\in \mathcal{F}^{\leq (N-1)}}c(\sigma; \sigma', \sigma'')  \Y^{\sigma}_{s}\langle \tau\otimes\X_{s,t}, \sigma' \otimes\sigma''\rangle\\
=&\ \langle \tau, \Y_{s,t}\rangle-\sum_{\sigma\in \mathcal{F}^{\leq (N-1)}}\sum_{\rho\in \mathcal{F}^{\leq (N-1)}}c(\sigma; \tau, \rho) \Y^{\sigma}_{s}\langle \tau, \tau\rangle \langle \X_{s,t},\rho \rangle\\
=&\ \Y_{s,t}^{\tau}-\sum_{\sigma,\rho\in\mathcal{F}^{\leq (N-1)}} c(\sigma; \tau, \rho) \Y^{\sigma}_{s}\X^{\rho}_{s,t}.
\end{align*}
\end{remark}

\begin{remark}
One can adapt Definition~\mref{defn:cbrp} component-wisely to
$$\Y=(\Y^1, \ldots, \Y^n): [0,T]\rightarrow (\hmsN)^n,\quad t\mapsto \Y_t=(\Y^1_t, \ldots, \Y^n_t), $$
where $(\hmsN)^n$ denotes the $n$-th cartesian power of $\hmsN$. Hence
\begin{equation}
\Y_t^\tau :=\langle \tau, \Y_{t} \rangle\in \RR^n, \quad \forall \tau\in \mathcal{F}^{\leq (N-1)}.
\mlabel{eq:pairy}
\end{equation}
\mlabel{rk:gcont}
\end{remark}

For $\X\in \brp$, denote by $\cbrpx$ the space of
$\X$-controlled planarly branched rough paths
given in Remark~\mref{rk:gcont}.
For $\Y\in \cbrpx$ and $\tilde \Y\in \cbrptx$, define the norm
\begin{equation}
\fan{\Y} _{\X;\alpha } :=\sum_{\tau\in \f^{\leq (N-1)}}\| R\Y^{\tau} \|_{( N-| \tau |) \alpha  }
\mlabel{eq:normy}
\end{equation}
and the metric
\begin{equation}
\fan{\Y, \tilde \Y} _{\X,\tilde \X;\alpha } :=\sum_{\tau\in \f^{\leq (N-1)}}\| R\Y^{\tau}-  R\tilde \Y^{\tau}  \|_{( N-| \tau |) \alpha  }=\fan{\Y-\tilde \Y} _{\X;\alpha }.
\mlabel{eq:dist}
\end{equation}
Similar to the case of controlled branched rough paths~\cite{Kel}, the $(\cbrpx, |||\cdot|||_{\X;\alpha})$ is a Banach space.

Our goal is to establish the universal limit theorem for planarly branched rough paths using the fixed point method, which relies on a series of detailed computations. Accordingly, for the remainder of the paper, we restrict our attention to the case $\alpha \in (\frac{1}{4}, \frac{1}{3}]$. For the fixed point approach in the case $\alpha \in (\frac{1}{3}, \frac{1}{2}]$ within the setting of Lyons' rough paths, we refer the reader to~\cite{FH2, Geng}.

Let $\alpha\in (\frac{1}{4},\frac{1}{3}]$. Then $N= \lfloor \frac{1}{\alpha } \rfloor = 3$
and $\calF^{\leq 2}=\{\etree, \bullet_a, \bullet_a\bullet_b, \tddeuxa{$a$}{$b$}\  \mid a,b\in A\}$. We have
\begin{align*}
\hmr(\bullet_a \bullet_b) =\bullet_b\otimes \bullet_a\,\text{ and }\,
\hmr(\tddeuxa{$a$}{$b$}\,\,\,)=\bullet_a\otimes \bullet_b,
\end{align*}
and so $$c(\bullet_a\bullet_b; \bullet_b, \bullet_a) = 1 = c(\tddeuxa{$a$}{$b$}\,\,\, ; \bullet_a, \bullet_b) .$$
Further for $\sigma, \sigma_1\in \calF^{\leq 2}$, we get
$$c(\sigma; \sigma',\bullet_a\bullet_b) = 0 = c(\sigma; \sigma',\tddeuxa{$a$}{$b$}\,\,\,).$$
Hence Definition~\mref{defn:cbrp} in the case of $\alpha\in (\frac{1}{4},\frac{1}{3}]$ can be recast as follows.

\begin{definition}
Let $\alpha\in (\frac{1}{4},\frac{1}{3}]$ and $\X\in  \brpt $.
A path $\Y: [0,T]\rightarrow \hm^{\leq 2} $ is called an {\bf $\X$-controlled planarly branched rough path} if
for any $a,b\in A$,
\begin{align*}
&R\Y_{s,t}^{\etree} = O(| t-s | ^{3\alpha  } ),\quad  R\Y_{s,t}^{\bullet_a} = O(| t-s | ^{2\alpha  } ), \\
& R\Y_{s,t}^{\bullet_a \bullet_b } = O(| t-s | ^{\alpha  } ),\quad  R\Y_{s,t}^{\tddeuxa{$a$}{$b$}}\ =O(| t-s | ^{\alpha  } ),
\end{align*}
where
\begin{align}
R\Y^{\etree}_{s,t}: =&\ \Y^{\etree}_{s,t}-\Big(\sum_{a\in A}\Y_{s} ^{\bullet_a}\X_{s,t} ^{\bullet_a}+\sum_{a,b\in A}\Y_{s} ^{\bullet_a\bullet_b} \X_{s,t} ^{\bullet_a\bullet_b} +
\sum_{a,b\in A}\Y_{s} ^{\tddeuxa{$a$}{$b$}}\ \, \X_{s,t} ^{\tddeuxa{$a$}{$b$}}\ \Big),\mlabel{eq:cbrp1} \\
R\Y_{s,t}^{\bullet_a}:=&\ \Y_{s,t}^{\bullet_a}-\sum_{b\in A}(\Y_{s}^{\bullet_a\bullet_b}+\Y_{s}^{\tddeuxa{$a$}{$b$}}\ \,)\X_{s,t}^{\bullet_b},\mlabel{eq:cbrp2}\\
R\Y_{s,t}^{\bullet_a \bullet_b}:=&\ \Y_{s,t}^{\bullet_a \bullet_b},\mlabel{eq:cbrp3}\\
R\Y_{s,t}^{\tddeuxa{$a$}{$b$}}:=&\ \Y_{s,t}^{\tddeuxa{$a$}{$b$}}.\mlabel{eq:cbrp4}
\end{align}
\mlabel{defn:cbrp1}
\end{definition}

\begin{remark}
In analogy to~\cite[Lemma 1.26]{Geng}, we have
\begin{align*}
\dfrac{|\Y^{\etree}_{s,t}|}{|t-s|^{\alpha}}=&\ \dfrac{|\sum_{a\in A}\Y_{s} ^{\bullet_a}\X_{s,t} ^{\bullet_a}+\sum_{a,b\in A}\Y_{s} ^{\bullet_a\bullet_b} \X_{s,t} ^{\bullet_a\bullet_b} +
\sum_{a,b\in A}\Y_{s} ^{\tddeuxa{$a$}{$b$}}\ \, \X_{s,t} ^{\tddeuxa{$a$}{$b$}}+R\Y_{s,t}^{\etree}|}{|t-s|^{\alpha}} \hspace{0.5cm} \text{(by  ~(\ref{eq:cbrp1}))}\\
\le&\ \sum_{a\in A}\dfrac{|\Y_{s} ^{\bullet_a}|\,|\X_{s,t} ^{\bullet_a}|}{|t-s|^{\alpha}}+\sum_{a,b\in A}\dfrac{|\Y_{s} ^{\bullet_a\bullet_b}|\,|\X_{s,t} ^{\bullet_a\bullet_b}|}{|t-s|^{\alpha}}+\sum_{a,b\in A}\dfrac{|\Y_{s} ^{\tddeuxa{$a$}{$b$}}\ \,|\,|\X_{s,t} ^{\tddeuxa{$a$}{$b$}}\ \,|}{|t-s|^{\alpha}}+\dfrac{|R\Y_{s,t}^{\etree}|}{|t-s|^{\alpha}}\\
\le&\ \sum_{a\in A}\dfrac{|\Y_{s} ^{\bullet_a}|\,|\X_{s,t} ^{\bullet_a}|}{|t-s|^{\alpha}}+\sum_{a,b\in A}T^{\alpha}\dfrac{|\Y_{s} ^{\bullet_a\bullet_b}|\,|\X_{s,t} ^{\bullet_a\bullet_b}|}{|t-s|^{2\alpha}}+\sum_{a,b\in A}T^{\alpha}\dfrac{|\Y_{s} ^{\tddeuxa{$a$}{$b$}}\ \,|\,|\X_{s,t} ^{\tddeuxa{$a$}{$b$}}\ \,|}{|t-s|^{2\alpha}}+T^{2\alpha}\dfrac{|R\Y_{s,t}^{\etree}|}{|t-s|^{3\alpha}} \\
\le&\ dk\|\X^{\bullet_a}\|_{\alpha}+d^2kT\|\X^{\bullet_a\bullet_b}\|_{2\alpha}+d^2kT\|\X^{\tddeuxa{$a$}{$b$}}\ \,\|_{2\alpha}+T^{3\alpha}\|R\Y^{\etree}\|_{3\alpha},
\end{align*}
where
$$k:=\max_{0\le s\le T} \Big\{|\Y_{s} ^{\bullet_a}|, |\Y_{s} ^{\bullet_a\bullet_b}|, |\Y_{s}^{\tddeuxa{$a$}{$b$}}\ \,| \, \big|\, a,b\in A\Big\}.$$
That is
$
\Y_{s,t}^{\etree} = O(| t-s | ^{\alpha  } ).
$ Similarly,
\begin{equation*}
\Y_{s,t}^{\bullet_a} =  O(| t-s | ^{\alpha  } ),\ \Y_{s,t}^{\bullet_a\bullet_b} =  O(| t-s | ^{\alpha  } )\,\text{ and }\, \Y_{s,t}^{\tddeuxa{$a$}{$b$}}\ \,= O(| t-s | ^{\alpha  } ).
\end{equation*}
\mlabel{re:ieq}
\end{remark}

\subsection{An upper bound of the composition with regular functions}\label{sec2.2}
In this subsection, we first show that the composition of a controlled planarly branched rough path with a regular function yields another controlled planarly branched rough path. Building on this result, we then establish an upper bound on the norm of the resulting controlled path.

Denote by $C^m(\mathbb{R}^n ; \mathbb{R}^n)$ the space of $m$ times continuously differentiable maps from $\mathbb{R}^n$ to $\mathbb{R}^n$ for $m\in \ZZ_{\geq 1}$.
For any $F\in  C^m(\mathbb{R}^n ; \mathbb{R}^n)$, define
\begin{equation}
\|F\|_{\infty }:=\sup_{x\in \mathbb{R}^n}|F(x)|.
\mlabel{eq:infF}
\end{equation}
A sub-space $C^{m}_b(\RR^n;\RR^n)\subset  C^m(\mathbb{R}^n ; \mathbb{R}^n) $ is given by those $F$ in $C^m(\mathbb{R}^n ; \mathbb{R}^n) $ such that
\begin{equation}
\|F\|_{C^{m}_b}:=\|F\|_{\infty }+\|F'\|_{\infty }+\cdots +\|F^{(m)}\|_{\infty } < \infty.
\mlabel{eq:fcbn}
\end{equation}

\begin{proposition}
Let $\alpha\in (\frac{1}{4},\frac{1}{3}]$ and $\X\in  \brpt $. Let $F\in C_{b}^{3}(\mathbb{R}^n; \mathbb{R}^n) $ and $\Y:[0,T]\rightarrow (\mathcal{H}_{\mathrm {MKW}}^{\leq 2})^n$ an $\X$-controlled planarly branched rough path above $Y:[0, T]\longrightarrow \RR^n$. Define
$$\Z: [0,T]\rightarrow (\mathcal{H}_{\mathrm {MKW}}^{\leq 2})^n,\quad t\mapsto \Z_{t} $$ by
\begin{equation}
\begin{aligned}
\Z^{\etree}_{t}:=&\ F(\Y^{\etree}_{t}) , \quad \Z^{\bullet_{a}}_{t}:=\ F'(\Y^{\etree}_{t})\Y^{\bullet_{a}}_{t},  \quad
\Z^{\mathbf{\tddeuxa{$a$}{$b$}}}_{t}\ \,:=\  F'(\Y^{\etree}_{t})\Y^{\mathbf{\tddeuxa{$a$}{$b$}}}_{t}\ \,,\\
\Z^{\bullet_{a}\bullet_{b}}_{t}:=&\  F'(\Y^{\etree}_{t})\Y^{\bullet_{a}\bullet_{b}}_{t}
+F''(\Y^{\etree}_{t})\Y^{\bullet_{a}}_{t}\Y^{\bullet_{b}}_{t}, \quad \forall a,b \in A.
\end{aligned}
\mlabel{eq:zcbrp1}
\end{equation}
Then $\Z$ is an $\X$-controlled planarly branched rough path above $F(Y)$, denoted by $F(\Y)$.
\mlabel{pp:regu}
\end{proposition}

\begin{proof}
By Definition~\mref{defn:cbrp1}, it suffices to prove that
\begin{align*}
R\Z^{\etree}=&\ O(| t-s | ^{3\alpha  } ),\quad R\Z^{\bullet}= O(| t-s | ^{2\alpha  } ),\\
R\Z^{\bullet \bullet}=&\ O(| t-s | ^{\alpha  } ),\quad R\Z^{\tddeux}=O(| t-s | ^{\alpha  } ),
\end{align*}
where
\begin{align}
R\Z^{\etree}_{s,t}:=&\ \Z^{\etree}_{s,t}-\Big(\sum_{a\in A}\Z_{s} ^{\bullet_a}\X_{s,t} ^{\bullet_a}+\sum_{a,b\in A}\Z_{s} ^{\bullet_a\bullet_b} \X_{s,t} ^{\bullet_a\bullet_b} +
\sum_{a,b\in A}\Z_{s} ^{\tddeuxa{$a$}{$b$}}\ \, \X_{s,t} ^{\tddeuxa{$a$}{$b$}}\ \Big),\mlabel{eq:zcbrp2}  \\
R\Z_{s,t}^{\bullet_a}:=&\ \Z_{s,t}^{\bullet_a}-\sum_{b\in A}(\Z_{s}^{\bullet_a\bullet_b}+\Z_{s}^{\tddeuxa{$a$}{$b$}}\ \,)\X_{s,t}^{\bullet_b},\mlabel{eq:zcbrp3}\\
R\Z_{s,t}^{\tddeuxa{$a$}{$b$}}:=&\ \Z_{t}^{\tddeuxa{$a$}{$b$}}\ \,-\Z_{s}^{\tddeuxa{$a$}{$b$}},\mlabel{eq:zcbrp4}\\
R\Z_{s,t}^{\bullet_a \bullet_b}:=&\ \Z_{t}^{\bullet_a \bullet_b}-\Z_{s}^{\bullet_a \bullet_b}, \quad \forall a,b\in A.
\mlabel{eq:zcbrp5}
\end{align}
We divide the proceeding proof into four steps.\\

\noindent {\bf Step 1:} $R\Z^{\etree}=O(| t-s | ^{3\alpha  } )$ follows from
\begin{align*}
\Z_{s,t}^{\etree} =&\  F( \Y_{t}^{\etree} )-F( \Y_{s}^{\etree} )\\
=&\ F'( \Y_{s}^{\etree} ) \Y_{s,t}^{\etree}+\frac{1}{2}F''(\Y_{s}^{\etree}  ) (\Y_{s,t}^{\etree})^2+O( | t-s |^{3\alpha }   ) \hspace{0.5cm} (\text{by the Taylor expression of $F$ at $\Y_s^\etree$}) \\
=&\ F'( \Y_{s}^{\etree} )(\sum_{a\in A}\Y_{s} ^{\bullet_a}\X_{s,t} ^{\bullet_a}+\sum_{a,b\in A}\Y_{s} ^{\bullet_a\bullet_b} \X_{s,t} ^{\bullet_a\bullet_b} +
\sum_{a,b\in A}\Y_{s} ^{\tddeuxa{$a$}{$b$}}\ \, \X_{s,t} ^{\tddeuxa{$a$}{$b$}}\ )\\
&\ +\frac{1}{2}F''(\Y_{s}^{\etree}  )(\sum_{a\in A}\Y_{s} ^{\bullet_a}\X_{s,t} ^{\bullet_a}+\sum_{a,b\in A}\Y_{s} ^{\bullet_a\bullet_b} \X_{s,t} ^{\bullet_a\bullet_b} +
\sum_{a,b\in A}\Y_{s} ^{\tddeuxa{$a$}{$b$}}\ \, \X_{s,t} ^{\tddeuxa{$a$}{$b$}}\ )^2+O( | t-s |^{3\alpha }   ) \\
&\ \hspace{7cm} (\text{by (\ref{eq:cbrp1}) \text{ and } (\ref{eq:cbrp2})})\\
=&\  F'(\Y_{s}^{\etree})(\sum_{a\in A}\Y_{s} ^{\bullet_a}\X_{s,t} ^{\bullet_a}+\sum_{a,b\in A}\Y_{s} ^{\bullet_a\bullet_b} \X_{s,t} ^{\bullet_a\bullet_b} +
\sum_{a,b\in A}\Y_{s} ^{\tddeuxa{$a$}{$b$}}\ \, \X_{s,t} ^{\tddeuxa{$a$}{$b$}}\ )+\frac{1}{2}F''(\Y_{s}^{\etree}) (\sum_{a\in A}\Y_{s} ^{\bullet_a}\X_{s,t} ^{\bullet_a})^2+O(| t-s |^{3\alpha })\\
&\ \hspace{7cm} (\text{by Definition~\ref{def:pbrp}~(c)})\\
=&\ F'(\Y_{s}^{\etree})(\sum_{a\in A}\Y_{s} ^{\bullet_a}\X_{s,t} ^{\bullet_a}+\sum_{a,b\in A}\Y_{s} ^{\bullet_a\bullet_b} \X_{s,t} ^{\bullet_a\bullet_b} +
\sum_{b,c\in A}\Y_{s} ^{\tddeuxa{$b$}{$c$}}\ \, \X_{s,t} ^{\tddeuxa{$b$}{$c$}}\ )\\
&\ +\frac{1}{2}F''(\Y_{s}^{\etree})\sum_{a,b\in A}\Y_{s}^{\bullet_a}\Y_{s}^{\bullet_b}(\X _{s,t}^{\bullet_a}\X _{s,t}^{\bullet_b})+O(|t-s|^{3\alpha }).\\
=&\ F'(\Y_{s}^{\etree})(\sum_{a\in A}\Y_{s} ^{\bullet_a}\X_{s,t} ^{\bullet_a}+\sum_{a,b\in A}\Y_{s} ^{\bullet_a\bullet_b} \X_{s,t} ^{\bullet_a\bullet_b} +
\sum_{a,b\in A}\Y_{s} ^{\tddeuxa{$a$}{$b$}}\ \, \X_{s,t} ^{\tddeuxa{$a$}{$b$}}\ )\\
&\ +F''(\Y_{s}^{\etree})\sum_{a,b\in A}\Y_{s}^{\bullet_a}\Y_{s}^{\bullet_b}\X _{s,t}^{\bullet_a\bullet_b}+O(|t-s|^{3\alpha}) \hspace{0.5cm} (\text{by Definition~\ref{def:pbrp}~(b)})\\
=&\ \sum_{a\in A}(F'(\Y_{s}^{\etree})\Y_{s} ^{\bullet_a})\X _{s,t}^{\bullet_a}+\sum_{a,b\in A}(F'(\Y_{s}^{\etree})\Y_{s}^{\bullet_a\bullet_b}+F''(\Y_{s}^{\etree}) \Y_{s}^{\bullet_a}\Y_{s}^{\bullet_b})\X _{s,t}^{\bullet_a\bullet_b}\\
&\ +\sum_{a,b\in A}(F'(\Y_{s}^{\etree})\Y_{s}^{\tddeuxa{$a$}{$b$}}\ \,)\X _{s,t}^{\tddeuxa{$a$}{$b$}}\ \,+O(|t-s|^{3\alpha}) \\
=&\ \sum_{a\in A}\Z_{s} ^{\bullet_a}\X_{s,t} ^{\bullet_a}+\sum_{a,b\in A}\Z_{s} ^{\bullet_a\bullet_b} \X_{s,t} ^{\bullet_a\bullet_b} +
\sum_{a,b\in A}\Z_{s} ^{\tddeuxa{$a$}{$b$}}\ \, \X_{s,t} ^{\tddeuxa{$a$}{$b$}}\ +O(|t-s|^{3\alpha}).\hspace{0.5cm} (\text{by (\mref{eq:zcbrp1}}))
\end{align*}

\noindent {\bf Step 2:} $R\Z^{\bullet}=O(| t-s | ^{2\alpha  } )$ is obtained from
\begin{align*}
\Z_{s,t}^{\bullet_a}=&\ F'(\Y_{t}^{\etree})\Y_{t}^{\bullet_a}-F'(\Y_{s}^{\etree})\Y_{s}^{\bullet_a} \hspace{0.5cm} (\text{by (\mref{eq:zcbrp1}}))\\
=&\ (F'(\Y_{t}^{\etree})-F'(\Y_{s}^{\etree}))\Y_{t}^{\bullet_a}+F'(\Y_{s}^{\etree})\Y_{s,t}^{\bullet_a}\\
=&\ (F''(\Y_{s}^{\etree})\Y_{s,t}^{\etree}+O(|\Y_{s,t}^{\etree}|^{2}))\Y_{t}^{\bullet_a}+F'(\Y_{s}^{\etree})\Y_{s,t}^{\bullet_a}  \hspace{0.5cm} (\text{by the Taylor expression of $F'$ at $\Y_s^\etree$})\\
=&\ F''(\Y_{s}^{\etree})(\Y_{s,t}^{\etree} )\Y_{t}^{\bullet_a}+F'(\Y_{s}^{\etree})\Y_{s,t}^{\bullet_a} +O(|t-s|^{2\alpha}) \hspace{1cm} (\text{by Remark~\ref{re:ieq}  })\\
=&\ F''(\Y_{s}^{\etree})(\Y_{s,t}^{\etree})(\Y_{s}^{\bullet_a}+(\Y_{t}^{\bullet_a}-\Y_{s}^{\bullet_a}))+F'(\Y_{s}^{\etree})\Y_{s,t}^{\bullet_a} +O(|t-s|^{2\alpha})  \\
=&\ F''(\Y_{s}^{\etree})(\Y_{s,t}^{\etree})\Y_{s}^{\bullet_a}+F''(\Y_{s}^{\etree})(\Y_{s,t}^{\etree})\Y_{s,t}^{\bullet_a}+F'(\Y_{s}^{\etree})\Y_{s,t}^{\bullet_a} +O(|t-s|^{2\alpha }), \\
=&\ F''(\Y_{s}^{\etree})(\Y_{s,t}^{\etree})\Y_{s}^{\bullet_a}+F'(\Y_{s}^{\etree})\Y_{s,t}^{\bullet_a} +O(|t-s|^{2\alpha}) \hspace{1cm} (\text{by Remark~\ref{re:ieq}})\\
=&\ F''(\Y_{s}^{\etree})(\sum_{b\in A}\Y_{s} ^{\bullet_b}\X_{s,t}^{\bullet_b}+\sum_{b,c\in A}\Y_{s} ^{\bullet_b\bullet_c} \X_{s,t} ^{\bullet_b\bullet_c} +
\sum_{b,c\in A}\Y_{s} ^{\tddeuxa{$b$}{$c$}}\ \, \X_{s,t} ^{\tddeuxa{$b$}{$c$}}\ )\Y_{s}^{\bullet_a}+F'(\Y_{s}^{\etree})\Y_{s,t}^{\bullet_a} +O(|t-s|^{2\alpha})  \\
=&\ F''(\Y_{s}^{\etree})\sum_{b\in A}(\Y_{s} ^{\bullet_b}\X_{s,t}^{\bullet_b})\Y_{s}^{\bullet_a}+F'(\Y_{s}^{\etree})\Y_{s,t}^{\bullet_a} +O(|t-s|^{2\alpha})  \hspace{0.5cm} (\text{by Definition~\ref{def:pbrp}~(c)})\\
=&\ F''(\Y_{s}^{\etree})\sum_{b\in A}(\Y_{s} ^{\bullet_a}\Y_{s}^{\bullet_b})\X _{s,t}^{\bullet_b}+F'(\Y_{s}^{\etree})\Y_{s,t}^{\bullet_a} +O(|t-s|^{2\alpha})  \\
=&\ F''(\Y_{s}^{\etree})\sum_{b\in A}(\Y_{s} ^{\bullet_a}\Y_{s}^{\bullet_b})\X _{s,t}^{\bullet_b}+F'( \Y_{s}^{\etree} )(\sum_{b\in A}(\Y_{s}^{\bullet_a\bullet_b}+\Y_{s}^{\tddeuxa{$a$}{$b$}}\ \,)\X_{s,t}^{\bullet_b}+R\Y_{s,t}^{\bullet_a})+O(|t-s|^{2\alpha})  \\
=&\ F''(\Y_{s}^{\etree})\sum_{b\in A}(\Y_{s} ^{\bullet_a}\Y_{s}^{\bullet_b})\X _{s,t}^{\bullet_b}+F'( \Y_{s}^{\etree} )(\sum_{b\in A}(\Y_{s}^{\bullet_a\bullet_b}+\Y_{s}^{\tddeuxa{$a$}{$b$}}\ \,)\X_{s,t}^{\bullet_b})+O(|t-s|^{2\alpha})   \hspace{0.5cm} (\text{by Definition~\ref{defn:cbrp}})\\
=&\ \sum_{b\in A}\bigg(\Big(F''(\Y_{s}^{\etree})\Y_{s} ^{\bullet_a}\Y_{s}^{\bullet_b}+F'(\Y_{s}^{\etree})\Y_{s}^{\bullet_a\bullet_b}\Big)+\sum_{b\in A}\Big(F'( \Y_{s}^{\etree})\Y_{s}^{\tddeuxa{$a$}{$b$}}\ \,\Big)\bigg)\X _{s,t}^{\bullet_b} +O(|t-s|^{2\alpha})\\
=&\ \sum_{b\in A}(\Z_{s}^{\bullet_a\bullet_b}+\Z_{s}^{\tddeuxa{$a$}{$b$}}\ \,)\X _{s,t}^{\bullet_b} +O(|t-s|^{2\alpha}).
\end{align*}

\noindent {\bf Step 3:} $R\Z^{\tddeux}=O(| t-s | ^{\alpha  } )$ can be computed directly by
\begin{align*}
\Z_{s,t}^{\tddeuxa{$a$}{$b$}}=&\ \Z_{t}^{\tddeuxa{$a$}{$b$}}\ \,-\Z_{s}^{\tddeuxa{$a$}{$b$}}\\
=&\  F'( \Y_{t}^{\etree} )\Y_{t}^{\tddeuxa{$a$}{$b$}}\ \,-F'(\Y_{s}^{\etree})\Y_{s}^{\tddeuxa{$a$}{$b$}} \hspace{0.5cm} (\text{by (\mref{eq:zcbrp1}}))\\
=&\  (F'(\Y_{t}^{\etree})-F'(\Y_{s}^{\etree})) \Y_{t}^{\tddeuxa{$a$}{$b$}}\ \,+F'(\Y_{s}^{\etree})\Y_{s,t}^{\tddeuxa{$a$}{$b$}} \\
=&\  F''(\Y_{s}^{\etree})(\Y_{s,t}^{\etree})\Y_{t}^{\tddeuxa{$a$}{$b$}}\ \,+F'(\Y_{s}^{\etree})\Y_{s,t}^{\tddeuxa{$a$}{$b$}}\ \,+O(|t-s|^{2\alpha}) \hspace{0.5cm} (\text{by Lagrange's mean value theorem})\\
=&\ O( | t-s |^{\alpha }).\hspace{2cm} (\text{by Remark~\ref{re:ieq}  }).
\end{align*}

\noindent {\bf Step 4:} $R\Z^{\bullet \bullet}=O(| t-s | ^{\alpha  } )$ can be shown as
%Finally, since $\Y^{\etree}\,,\, \Y^{\bullet}\,\text{ and }\, \Y^{\bullet \bullet}\in %C^{\alpha}([0, T];\RR^e)$, we have
\begin{align*}
\Z_{s,t}^{\bullet_a\bullet_b}=&\ (F'(\Y_{t}^{\etree})\Y_{t}^{\bullet_a\bullet_b}+F''(\Y_{t}^{\etree}) \Y_{t}^{\bullet_b}\Y_{t}^{\bullet_c})-(F'(\Y_{s}^{\etree})\Y_{s}^{\bullet_a\bullet_b}+F''(\Y_{s}^{\etree})\Y_{s}^{\bullet_a}\Y_{s}^{\bullet_b})\\
&\ \hspace{7cm} (\text{by (\mref{eq:zcbrp1}})) \\
=&\ (F'(\Y_{t}^{\etree})\Y_{t}^{\bullet_a\bullet_b}-F'(\Y_{s}^{\etree})\Y_{s}^{\bullet_a\bullet_b})+(F''(\Y_{t}^{\etree})\Y_{t}^{\bullet_a}\Y_{t}^{\bullet_b}-F''(\Y_{s}^{\etree})\Y_{s}^{\bullet_a}\Y_{s}^{\bullet_b})\\
=&\ \Big((F'(\Y_{t}^{\etree})\Y_{t}^{\bullet_a\bullet_b}-F'(\Y_{s}^{\etree})\Y_{t}^{\bullet_a\bullet_b})+(F'(\Y_{s}^{\etree})\Y_{t}^{\bullet_a\bullet_b}-F'(\Y_{s}^{\etree})\Y_{s}^{\bullet_a\bullet_b})\Big)\\
&\ +\Big((F''(\Y_{t}^{\etree})\Y_{t}^{\bullet_a}\Y_{t}^{\bullet_b}-F''(\Y_{s}^{\etree})\Y_{t}^{\bullet_a}\Y_{t}^{\bullet_b})+(F''(\Y_{s}^{\etree})\Y_{t}^{\bullet_b}\Y_{t}^{\bullet_b}-F''(\Y_{s}^{\etree})\Y_{s}^{\bullet_a}\Y_{s}^{\bullet_b})\Big) \\
=&\ \Big((F'(\Y_{t}^{\etree})-F'(\Y_{s}^{\etree}))\Y_{t}^{\bullet_a\bullet_b}+F'(\Y_{s}^{\etree})\Y_{s,t}^{\bullet_a\bullet_b}\Big)\\
&\ +\Big((F''(\Y_{t}^{\etree})-F''(\Y_{s}^{\etree}))\Y_{t}^{\bullet_a}\Y_{t}^{\bullet_b}+F''(\Y_{s}^{\etree})(\Y_{t}^{\bullet_a}\Y_{t}^{\bullet_b}-\Y_{s}^{\bullet_a}\Y_{s}^{\bullet_b})\Big) \\
=&\ \Big((F''(\Y_{s}^{\etree})(\Y_{s,t}^{\etree})\Y_{t}^{\bullet_a\bullet_b}+O(|t-s|^{2\alpha}))+F'(\Y_{s}^{\etree})\Y_{s,t}^{\bullet_a\bullet_b}\Big) \\
&\ +\Big((F'''(\Y_{s}^{\etree})(\Y_{s,t}^{\etree})+O(|t-s|^{2\alpha}))\Y_{t}^{\bullet_a}\Y_{t}^{\bullet_b}+F''(\Y_{s}^{\etree})(\Y_{t}^{\bullet_a}\Y_{t}^{\bullet_b}-\Y_{t}^{\bullet_a}\Y_{s}^{\bullet_b}+\Y_{t}^{\bullet_a}\Y_{s}^{\bullet_b}-\Y_{s}^{\bullet_a}\Y_{s}^{\bullet_b})\Big)\\
&\ \hspace{6cm} (\text{by the Taylor expression of $F'$ at $\Y_s^\etree$ and  $F''$ at $\Y_s^\etree$})\\
=&\ F''(\Y_{s}^{\etree})\Y_{s,t}^{\etree}\Y_{t}^{\bullet_a\bullet_b}+F'(\Y_{s}^{\etree})\Y_{s,t}^{\bullet_a\bullet_b}+F'''(\Y_{s}^{\etree})\Y_{s,t}^{\etree}\Y_{t}^{\bullet_a}\Y_{t}^{\bullet_b}+F''(\Y_{s}^{\etree})\Y_{t}^{\bullet_a}\Y_{s,t}^{\bullet_b}\\
&\ +F''(\Y_{s}^{\etree})\Y_{s,t}^{\bullet_a}\Y_{s}^{\bullet_b}+O(|t-s|^{2\alpha}) \\
=&\ O( | t-s |^{\alpha } ).\hspace{4cm} (\text{by Remark~\ref{re:ieq} }) \qedhere
\end{align*}
\end{proof}

We now state the main result of this subsection, which provides an upper bound for the planarly branched rough path $\Z$ in Proposition~\mref{pp:regu}. Recall that $M(\bullet)$ denotes a universal function, increasing in all variables, whose expression may vary from line to line.

\begin{theorem}
With the setting in Proposition~\mref{pp:regu}, the following estimate holds
\begin{equation*}
\fan{\Z} _{\X;\alpha } \le 4d^2\|F\|_{C_{b}^{3}} M\Big(T, \sum_{a\in A}|\Y_{0}^{\bullet_a}|, \sum_{a,b\in A}|\Y_{0}^{\bullet_a\bullet_b}|, \sum_{a,b\in A}| \Y_{0}^{\tddeuxa{$a$}{$b$}}\ \,|, \fan{\Y} _{\X;\alpha }, \fan{\X} _{\alpha } \Big).
\end{equation*}
\mlabel{thm:stab1}
\end{theorem}

\begin{proof}
First, we give three upper bounds of $\Y^{\bullet_a}$, $\Y^{\tddeuxa{$a$}{$b$}}$\ \, and $\Y^{\bullet_a\bullet_b}$, respectively. For $\Y^{\bullet_a}$,
\begin{equation}
| \Y_{s}^{\bullet_a}  |= |\Y_{0}^{\bullet_a}+\Y_{0,s}^{\bullet_a}|\le |\Y_{0}^{\bullet_a}|+|\Y_{0,s}^{\bullet_a}|
\le |\Y_{0}^{\bullet_a}|+T^{\alpha}\frac{|\Y_{0,s}^{\bullet_a}|}{|s-0|^\alpha}
\le |\Y_{0}^{\bullet_a}|+T^{\alpha}\|\Y^{\bullet_a}\|_{\alpha}.
\mlabel{eq:iequ1}
\end{equation}
Similarly,
\begin{equation}
|\Y_{s}^{\tddeuxa{$a$}{$b$}}\ \,|\le |\Y_{0}^{\tddeuxa{$a$}{$b$}}\ \,|+T^{\alpha }\| \Y^{\tddeuxa{$a$}{$b$}}\ \,\|_{\alpha }
\mlabel{eq:iequ3}
\end{equation}
and
\begin{equation}
|\Y_{s}^{\bullet_a\bullet_b}|\le |\Y_{0}^{\bullet_a\bullet_b}|+T^{\alpha }\| \Y^{\bullet_a\bullet_b}\|_{\alpha }.
\mlabel{eq:iequ2}
\end{equation}
Further, since
\begin{align*}
|\Y_{s,t}^{\bullet_a}| &=|\sum_{b\in A}(\Y_{s}^{\bullet_a\bullet_b}+\Y_{s}^{\tddeuxa{$a$}{$b$}}\ \,)\X_{s,t}^{\bullet_b}+R\Y_{s,t}^{\bullet_a}| \hspace{0.5cm} \text{(by  ~(\ref{eq:cbrp2}))}\\
&\le \sum_{b\in A}(|\Y_{s}^{\bullet_a\bullet_b}|+|\Y_{s}^{\tddeuxa{$a$}{$b$}}\ \,|)|\X _{s,t}^{\bullet_b}|+|R\Y_{s,t}^{\bullet_a}|\\
&\le \sum_{b\in A}\Big((|\Y_{0}^{\bullet_a\bullet_b}|+T^{\alpha }\| \Y^{\bullet_a\bullet_b}\|_{\alpha })+(|\Y_{0}^{\tddeuxa{$a$}{$b$}}\ \,|+T^{\alpha }\| \Y^{\tddeuxa{$a$}{$b$}}\ \,\|_{\alpha })\Big)|\X _{s,t}^{\bullet_b}|+|R\Y_{s,t}^{\bullet_a}|, \hspace{0.5cm} (\text{by ~(\ref{eq:iequ3}) \text{ and } (\ref{eq:iequ2})})
\end{align*}
we have
\begin{equation}
\|\Y^{\bullet_a}\|_{\alpha } =\sup_{s\ne t\in [ 0,T ] } \frac{|\Y_{s,t}^{\bullet_a}|}{| t-s |^{\alpha}}  \le \sum_{b\in A}\Big((|\Y_{0}^{\bullet_a\bullet_b}|+T^{\alpha }\| \Y^{\bullet_a\bullet_b}\|_{\alpha })+(|\Y_{0}^{\tddeuxa{$a$}{$b$}}\ \,|+T^{\alpha }\| \Y^{\tddeuxa{$a$}{$b$}}\ \,\|_{\alpha })\Big)\| \X^{\bullet_b} \|_{\alpha }+T^{\alpha }\|R\Y^{\bullet_a}\|_{2\alpha },
\mlabel{eq:yaalpha}
\end{equation}
and so
\begin{equation}
\begin{aligned}
|\Y_{s}^{\bullet_a}|&\ \le |\Y_{0}^{\bullet_a}|+T^{\alpha }\| \Y^{\bullet_a}\|_{\alpha } \hspace{4cm}  (\text{by  ~(\ref{eq:iequ1})})\\
&\ \le |\Y_{0}^{\bullet_a}|+T^{\alpha}\sum_{b\in A}\Big((|\Y_{0}^{\bullet_a\bullet_b}|+T^{\alpha }\| \Y^{\bullet_a\bullet_b}\|_{\alpha })+(|\Y_{0}^{\tddeuxa{$a$}{$b$}}\ \,|+T^{\alpha }\| \Y^{\tddeuxa{$a$}{$b$}}\ \,\|_{\alpha })\Big)\| \X^{\bullet_b} \|_{\alpha }+T^{2\alpha }\|R\Y^{\bullet_a}\|_{2\alpha }. \\
&\ \hspace{7cm} (\text{by  ~(\ref{eq:yaalpha})})
\mlabel{eq:iequ4}
\end{aligned}
\end{equation}

\noindent Next, it follows from  ~(\ref{eq:normy}) that
$$\fan{\Z}_{\X;\alpha }= \|R\Z^{\etree }\|_{3\alpha}+\sum_{a\in A}\|R\Z^{\bullet_a}\|_{2\alpha }+\sum_{a,b\in A}\|R\Z^{\tddeuxa{$a$}{$b$}}\ \,\|_{\alpha }+\sum_{a,b\in A}\| R\Z^{\bullet_a\bullet_b}\|_{\alpha }, $$
whence we only need to estimate each term on the right hand side of the above equation.
We break the remaining proof down into four steps.\\

\noindent{\bf Step 1:}
For the first term $\| R\Z^{\etree } \| _{3\alpha }$, we have
\begin{align}
|R\Z_{s,t}^{\etree}|=&\ \Big|\Z^{\etree}_{s,t}-\Big(\sum_{a\in A}\Z_{s} ^{\bullet_a}\X_{s,t} ^{\bullet_a}+\sum_{a,b\in A}\Z_{s} ^{\bullet_a\bullet_b} \X_{s,t} ^{\bullet_a\bullet_b} +
\sum_{a,b\in A}\Z_{s} ^{\tddeuxa{$a$}{$b$}}\ \, \X_{s,t} ^{\tddeuxa{$a$}{$b$}}\ \Big)\Big| \hspace{0.5cm} (\text{by ~(\ref{eq:cbrp1})})\nonumber\\
%2
=&\ \Big|F(\Y_{t}^{\etree})-F(\Y_{s}^{\etree})-\Big(\sum_{a\in A}\Z_{s} ^{\bullet_a}\X_{s,t} ^{\bullet_a}+\sum_{a,b\in A}\Z_{s} ^{\bullet_a\bullet_b} \X_{s,t} ^{\bullet_a\bullet_b}+\sum_{a,b\in A}\Z_{s} ^{\tddeuxa{$a$}{$b$}}\ \, \X_{s,t} ^{\tddeuxa{$a$}{$b$}}\ \,\Big)\Big| \hspace{0.5cm} (\text{by  ~(\ref{eq:zcbrp1})})\nonumber\\
%3
=& \Big|F'(\Y_{s}^{\etree})\Y_{s,t}^{\etree}+\frac{1}{2}F''(\Y_{s}^{\etree})(\Y_{s,t}^{\etree})^2+\displaystyle\int_0^1\frac{(1-\theta)^2}{2}F'''(\Y_{s}^{\etree} +\theta\Y_{s,t}^{\etree})( \Y_{s,t}^{\etree})^3d\theta\nonumber\\
&\ -\Big(\sum_{a\in A}\Z_{s} ^{\bullet_a}\X_{s,t} ^{\bullet_a}+\sum_{a,b\in A}\Z_{s} ^{\bullet_a\bullet_b} \X_{s,t} ^{\bullet_a\bullet_b}+\sum_{a,b\in A}\Z_{s} ^{\tddeuxa{$a$}{$b$}}\ \, \X_{s,t} ^{\tddeuxa{$a$}{$b$}}\ \,\Big)\Big| \nonumber\\
&\ \hspace{7cm} (\text{by the Taylor expression of $F$ at $\Y_s^\etree$})\nonumber\\
%4
=&\ \Big|F'(\Y_{s}^{\etree})\Big(\sum_{a\in A}\Y_{s} ^{\bullet_a}\X_{s,t} ^{\bullet_a}+\sum_{a,b\in A}\Y_{s} ^{\bullet_a\bullet_b} \X_{s,t} ^{\bullet_a\bullet_b} +
\sum_{a,b\in A}\Y_{s} ^{\tddeuxa{$a$}{$b$}}\ \, \X_{s,t} ^{\tddeuxa{$a$}{$b$}}\ \,+R\Y^{\etree}_{s,t}\Big)+\frac{1}{2}F''(\Y_{s}^{\etree})(\Y_{s,t}^{\etree})^2\nonumber\\
&\ +\displaystyle\int_0^1\frac{(1-\theta)^2}{2}F'''(\Y_{s}^{\etree} +\theta\Y_{s,t}^{\etree})( \Y_{s,t}^{\etree})^3d\theta-\Big(\sum_{a\in A}\Z_{s} ^{\bullet_a}\X_{s,t} ^{\bullet_a}+\sum_{a,b\in A}\Z_{s} ^{\bullet_a\bullet_b} \X_{s,t} ^{\bullet_a\bullet_b}+\sum_{a,b\in A}\Z_{s} ^{\tddeuxa{$a$}{$b$}}\ \, \X_{s,t} ^{\tddeuxa{$a$}{$b$}}\ \,\Big)\Big| \nonumber\\
&\ \hspace{7cm} (\text{by  ~(\ref{eq:cbrp1})}) \nonumber\\
%5
=&\ \Big|\sum_{a\in A}(F'(\Y_{s}^{\etree})\Y_{s} ^{\bullet_a}-\Z_{s} ^{\bullet_a})\X_{s,t} ^{\bullet_a}+\sum_{a,b\in A}(F'(\Y_{s}^{\etree})\Y_{s} ^{\bullet_a\bullet_b}-\Z_{s} ^{\bullet_a\bullet_b})\X_{s,t} ^{\bullet_a\bullet_b}+\sum_{a,b\in A}(F'(\Y_{s}^{\etree})\Y_{s} ^{\tddeuxa{$a$}{$b$}}\ \, -\Z_{s} ^{\tddeuxa{$a$}{$b$}}\ \, )\X_{s,t} ^{\tddeuxa{$a$}{$b$}} \nonumber\\
&\ +F'(\Y_{s}^{\etree})R\Y^{\etree}_{s,t}+\frac{1}{2}
F''(\Y_{s}^{\etree})(\Y_{s,t}^{\etree})^2+\displaystyle\int_0^1\frac{(1-\theta)^2}{2}F'''(\Y_{s}^{\etree} +\theta\Y_{s,t}^{\etree})( \Y_{s,t}^{\etree})^3d\theta\Big|  \nonumber\\
&\ \hspace{7cm} (\text{by combining like terms})\nonumber\\
%6
=&\ \Big|-\sum_{a,b\in A}F''(\Y_{s}^{\etree})\Y_{s}^{\bullet_a}\Y_{s}^{\bullet_b}\X_{s,t} ^{\bullet_a\bullet_b}+F'(\Y_{s}^{\etree})R\Y^{\etree}_{s,t}+\frac{1}{2}F''(\Y_{s}^{\etree})(\Y_{s,t}^{\etree})^2\nonumber\\
&+\displaystyle\int_0^1\frac{(1-\theta)^2}{2}F'''(\Y_{s}^{\etree} +\theta\Y_{s,t}^{\etree})( \Y_{s,t}^{\etree})^3d\theta\Big| \nonumber\\
&\ \hspace{0.5cm} (\text{by  ~(\ref{eq:zcbrp1}), the first and third summands vanish and the second summand is simplified}) \nonumber\\
%7
=&\ \Big|-\sum_{a,b\in A}F''(\Y_{s}^{\etree})\Y_{s}^{\bullet_a}\Y_{s}^{\bullet_b}\X_{s,t} ^{\bullet_a\bullet_b} +F'(\Y_{s}^{\etree})R\Y^{\etree}_{s,t}+\frac{1}{2}F''(\Y_{s}^{\etree})\Big(\sum_{a\in A}\Y_{s} ^{\bullet_a}\X_{s,t} ^{\bullet_a}\nonumber\\
&\ +\sum_{a,b\in A}\Y_{s} ^{\bullet_a\bullet_b} \X_{s,t} ^{\bullet_a\bullet_b} +
\sum_{a,b\in A}\Y_{s} ^{\tddeuxa{$a$}{$b$}}\ \, \X_{s,t} ^{\tddeuxa{$a$}{$b$}}\ \, +R\Y^{\etree}_{s,t}\Big)^2+\displaystyle\int_0^1\frac{(1-\theta)^2}{2}F'''(\Y_{s}^{\etree}+\theta\Y_{s,t}^{\etree})( \Y_{s,t}^{\etree})^3d\theta\Big| \\
&\ \hspace{7cm} (\text{by  ~(\ref{eq:cbrp1})})\nonumber\\
%8
=&\ \Big|-\sum_{a,b\in A}F''(\Y_{s}^{\etree})\Y_{s}^{\bullet_a}\Y_{s}^{\bullet_b}\X_{s,t} ^{\bullet_a\bullet_b}+F'(\Y_{s}^{\etree})R\Y^{\etree}_{s,t}+\frac{1}{2}F''(\Y_{s}^{\etree})\Big(2\sum_{a,b\in A}\Y_{s}^{\bullet_a}\Y_{s}^{\bullet_b}\X_{s,t} ^{\bullet_a\bullet_b}\nonumber\\
&\ +2\sum_{a,b,c\in A}\Y_{s}^{\bullet_a}\Y_{s}^{\bullet_b\bullet_c}(\X_{s,t} ^{\bullet_a\bullet_b\bullet_c}+\X_{s,t} ^{\bullet_b\bullet_c\bullet_a})+2\sum_{a,b,c\in A}\Y_{s}^{\bullet_a}\Y_{s}^{\tddeuxa{$b$}{$c$}}\ \,(\X_{s,t} ^{\bullet_a\tddeuxa{$b$}{$c$}}\ \,+\X_{s,t} ^{\tddeuxa{$b$}{$c$}\ \,\bullet_a})+2\Y_{s,t}^{\etree}R\Y^{\etree}_{s,t} \Big)\nonumber\\
&\ +\displaystyle\int_0^1\frac{(1-\theta)^2}{2}F'''(\Y_{s}^{\etree}+\theta\Y_{s,t}^{\etree})( \Y_{s,t}^{\etree})^3d\theta\Big|\nonumber\\
&\ \hspace{3cm} (\text{by Definition~\mref{def:pbrp}~(b), $\X_{s,t}^\tau = 0$ for $|\tau|\geq 4$ and  ~(\ref{eq:cbrp1})} )\nonumber\\
%9
=&\ \Big|F'(\Y_{s}^{\etree})R\Y^{\etree}_{s,t}+F''(\Y_{s}^{\etree})\sum_{a,b,c\in A}\Y_{s}^{\bullet_a}\Y_{s}^{\bullet_b\bullet_c}(\X_{s,t} ^{\bullet_a\bullet_b\bullet_c}+\X_{s,t} ^{\bullet_b\bullet_c\bullet_a})\nonumber\\
&\ +F''(\Y_{s}^{\etree})\sum_{a,b,c\in A}\Y_{s}^{\bullet_a}\Y_{s}^{\tddeuxa{$b$}{$c$}}\ \,(\X_{s,t} ^{\bullet_a\tddeuxa{$b$}{$c$}}\ \,+\X_{s,t} ^{\tddeuxa{$b$}{$c$}\ \,\bullet_a})\nonumber\\
&\ +F''(\Y_{s}^{\etree})\Y_{s,t}^{\etree}R\Y^{\etree}_{s,t}+\displaystyle\int_0^1\frac{(1-\theta)^2}{2}F'''( \Y_{s}^{\etree } +\theta \Y_{s,t}^{\etree } )( \Y_{s,t}^{\etree} ) ^3d\theta \Big| \mlabel{eq:simp1}\\
&\hspace{8cm} (\text{by combining like terms})\nonumber\\
%10
\le&\ \|F'\|_{\infty}|R\Y_{s,t}^{\etree}|+\|F''\|_{\infty}\sum_{a,b,c\in A}(|\Y_{0}^{\bullet_a}|+T^{\alpha}\|\Y^{\bullet_a}\|_{\alpha})(|\Y_{0}^{\bullet_b\bullet_c}|+T^{\alpha }\| \Y^{\bullet_b\bullet_c}\|_{\alpha })(|\X_{s,t} ^{\bullet_a\bullet_b\bullet_c}|+|\X_{s,t} ^{\bullet_b\bullet_c\bullet_a}|)\nonumber\\
&\ +\|F''\|_{\infty}\sum_{a,b,c\in A}(|\Y_{0}^{\bullet_a}|+T^{\alpha}\|\Y^{\bullet_a}\|_{\alpha})(|\Y_{0}^{\tddeuxa{$b$}{$c$}}\ \,|+T^{\alpha }\| \Y^{\tddeuxa{$b$}{$c$}}\ \,\|_{\alpha })(|\X_{s,t} ^{\bullet_a\tddeuxa{$b$}{$c$}}\ \,|+|\X_{s,t} ^{\tddeuxa{$b$}{$c$}\ \,\bullet_a}|)\nonumber\\
&\ +\|F''\|_{\infty}|\Y_{s,t}^{\etree}|\,|R\Y_{s,t}^{\etree}|+\dfrac{1}{6}\|F'''\|_{\infty } |\Y_{s,t}^{\etree} |^3 \nonumber\\
&\hspace{5cm} (\text{by  ~(\ref{eq:iequ1}), (\ref{eq:iequ3}), (\ref{eq:iequ2}) and $\displaystyle\int_0^1 \frac{(1-\theta)^2}{2} d\theta =\frac{1}{6}$}) \nonumber\\
%11
=&\ \|F'\|_{\infty}|R\Y_{s,t}^{\etree}|+\|F''\|_{\infty}\sum_{a,b,c\in A}(|\Y_{0}^{\bullet_a}|+T^{\alpha}\|\Y^{\bullet_a}\|_{\alpha})(|\Y_{0}^{\bullet_b\bullet_c}|+T^{\alpha }\| \Y^{\bullet_b\bullet_c}\|_{\alpha })(|\X_{s,t} ^{\bullet_a\bullet_b\bullet_c}|+|\X_{s,t} ^{\bullet_b\bullet_c\bullet_a}|)\nonumber\\
&\ +\|F''\|_{\infty}\sum_{a,b,c\in A}(|\Y_{0}^{\bullet_a}|+T^{\alpha}\|\Y^{\bullet_a}\|_{\alpha})(|\Y_{0}^{\tddeuxa{$b$}{$c$}}\ \,|+T^{\alpha }\| \Y^{\tddeuxa{$b$}{$c$}}\ \,\|_{\alpha })(|\X_{s,t} ^{\bullet_a\tddeuxa{$b$}{$c$}}\ \,|+|\X_{s,t} ^{\tddeuxa{$b$}{$c$}\ \,\bullet_a}|)\nonumber\\
&\ +\|F''\|_{\infty}\Big|\sum_{a\in A}\Y_{s} ^{\bullet_a}\X_{s,t} ^{\bullet_a}+\sum_{a,b\in A}\Y_{s} ^{\bullet_a\bullet_b} \X_{s,t} ^{\bullet_a\bullet_b} +
\sum_{a,b\in A}\Y_{s} ^{\tddeuxa{$a$}{$b$}}\ \, \X_{s,t} ^{\tddeuxa{$a$}{$b$}}\ \,+R\Y^{\etree}_{s,t}\Big|\,|R\Y_{s,t}^{\etree}|\nonumber\\
&\ +\dfrac{1}{6}\|F'''\|_{\infty }\Big|\sum_{a\in A}\Y_{s} ^{\bullet_a}\X_{s,t} ^{\bullet_a}+\sum_{a,b\in A}\Y_{s} ^{\bullet_a\bullet_b} \X_{s,t} ^{\bullet_a\bullet_b} +
\sum_{a,b\in A}\Y_{s} ^{\tddeuxa{$a$}{$b$}}\ \, \X_{s,t} ^{\tddeuxa{$a$}{$b$}}\ \,+R\Y^{\etree}_{s,t}\Big|^3.\nonumber\\
&\hspace{5cm} (\text{the fourth and fifth summands use  ~(\ref{eq:cbrp1})})\nonumber
\end{align}
Carrying the calculations further, we get
\begin{align}
&\ \dfrac{|R\Z_{s,t}^{\etree}|}{|t-s|^{3\alpha}}\nonumber\\
% 1st
\le&\ \|F'\|_{\infty }\dfrac{|R\Y_{s,t}^{\etree}|}{|t-s|^{3\alpha}}+\|F''\|_{\infty}\sum_{a,b,c\in A } (|\Y_{0}^{\bullet_a}|+T^{\alpha}\|\Y^{\bullet_a} \|_{\alpha}) (|\Y_{0}^{\bullet_b\bullet_c}| +T^{\alpha }\| \Y^{\bullet_b\bullet_c}\|_{\alpha })\Big(\dfrac{|\X_{s,t} ^{\bullet_a\bullet_b\bullet_c}|}{|t-s|^{3\alpha}}+\dfrac{|\X_{s,t} ^{\bullet_b\bullet_c\bullet_a}|}{|t-s|^{3\alpha}}\Big)\nonumber\\
&\ +\|F''\|_{\infty}\sum_{a,b,c\in A}(|\Y_{0}^{\bullet_a}|+T^{\alpha}\|\Y^{\bullet_a}\|_{\alpha})(|\Y_{0}^{\tddeuxa{$b$}{$c$}}\ \,|+T^{\alpha }\| \Y^{\tddeuxa{$b$}{$c$}}\ \,\|_{\alpha })\Big(\dfrac{|\X_{s,t} ^{\bullet_a\tddeuxa{$b$}{$c$}}\ \,|}{|t-s|^{3\alpha}}+\dfrac{|\X_{s,t} ^{\tddeuxa{$b$}{$c$}\ \,\bullet_a}|}{|t-s|^{3\alpha}}\Big)\nonumber\\
&\ +\|F''\|_{\infty}T^{\alpha}\Bigg(\sum_{a\in A}|\Y_{s}^{\bullet_a}\dfrac{|\X _{s,t}^{\bullet_a}|}{|t-s|^{\alpha}}+\sum_{a,b\in A}|\Y_{s}^{\bullet_a\bullet_b}|\dfrac{|\X _{s,t}^{\bullet_a\bullet_b}|}{|t-s|^{\alpha}}+\sum_{a,b\in A}|\Y_{s}^{\tddeuxa{$a$}{$b$}}\ \,|\dfrac{|\X _{s,t}^{\tddeuxa{$a$}{$b$}}\ \,|}{|t-s|^{\alpha}}+\dfrac{|R\Y_{s,t}^{\etree}|}{|t-s|^{\alpha}}\Bigg)\dfrac{|R\Y_{s,t}^{\etree}|}{|t-s|^{3\alpha}}\nonumber\\
&\ + \dfrac{1}{6}\|F'''\|_{\infty} \Bigg(\sum_{a\in A}|\Y_{s}^{\bullet_a}|\dfrac{|\X _{s,t}^{\bullet_a}|}{|t-s|^{\alpha}}+\sum_{a,b\in A}|\Y_{s}^{\bullet_a\bullet_b}|\dfrac{|\X _{s,t}^{\bullet_a\bullet_b}|}{|t-s|^{\alpha}}+\sum_{a,b\in A}|\Y_{s}^{\tddeuxa{$a$}{$b$}}\ \,|\dfrac{|\X _{s,t}^{\tddeuxa{$a$}{$b$}}\ \,|}{|t-s|^{\alpha}}+\dfrac{|R\Y_{s,t}^{\etree}|}{|t-s|^{\alpha}}\Bigg)^3\nonumber\\
% 2nd
\le&\ \|F'\|_{\infty}\|R\Y^{\etree}\|_{3\alpha}+\|F''\|_{\infty}\sum_{a,b,c\in A } (|\Y_{0}^{\bullet_a}|+T^{\alpha}\|\Y^{\bullet_a} \|_{\alpha}) |\Y_{0}^{\bullet_b\bullet_c}| +T^{\alpha }(\| \Y^{\bullet_b\bullet_c}\|_{\alpha })(\|\X^{\bullet_a\bullet_b\bullet_c}\|_{3\alpha}+\|\X^{\bullet_b\bullet_c\bullet_a}\|_{3\alpha})\nonumber\\
&\ +\|F''\|_{\infty}\sum_{a,b,c\in A}(|\Y_{0}^{\bullet_a}|+T^{\alpha}\|\Y^{\bullet_a}\|_{\alpha})(|\Y_{0}^{\tddeuxa{$b$}{$c$}}\ \,|+T^{\alpha }\| \Y^{\tddeuxa{$b$}{$c$}}\ \,\|_{\alpha })(\|\X^{\bullet_a\tddeuxa{$b$}{$c$}}\ \,\|_{3\alpha}+\|\X^{\tddeuxa{$b$}{$c$}\ \,\bullet_a}\|_{3\alpha})\nonumber\\
&\ +\|F''\|_{\infty}T^{\alpha}\Big(\sum_{a\in A}|\Y_{s}^{\bullet_a}|\,\|\X^{\bullet_a}\|_{\alpha}+T^{\alpha}\sum_{a,b\in A}|\Y_{s}^{\bullet_a\bullet_b}|\,\|\X ^{\bullet_a\bullet_b}\|_{2\alpha}+T^{\alpha}\sum_{a,b\in A}|\Y_{s}^{\tddeuxa{$a$}{$b$}}\ \,|\,\|\X^{\tddeuxa{$a$}{$b$}}\ \,|_{2\alpha}+T^{2\alpha}\|R\Y^{\etree}\|_{3\alpha}\Big)\|R\Y^{\etree}\|_{3\alpha}\nonumber\\
&\ + \dfrac{1}{6}\|F'''\|_{\infty} \Big(\sum_{a\in A}|\Y_{s}^{\bullet_a}|\,\|\X^{\bullet_a}\|_{\alpha}+T^{\alpha}\sum_{a,b\in A}|\Y_{s}^{\bullet_a\bullet_b}|\,\|\X ^{\bullet_a\bullet_b}\|_{2\alpha}+T^{\alpha}\sum_{a,b\in A}|\Y_{s}^{\tddeuxa{$a$}{$b$}}\ \,|\,\|\X^{\tddeuxa{$a$}{$b$}}\ \,|_{2\alpha}+T^{2\alpha}\|R\Y^{\etree}\|_{3\alpha}\Big) ^3 \nonumber\\
&\ \hspace{7cm}(\text{by  ~\meqref{eq:normx} and~\meqref{eq:remainry}})\nonumber\\
%3 begin long sum
\le&\ \|F'\|_{\infty}\|R\Y^{\etree}\|_{3\alpha}+\|F''\|_{\infty}\sum_{a,b,c\in A } \Bigg(|\Y_{0}^{\bullet_a}|+T^{\alpha}\sum_{h\in A}\big((|\Y_{0}^{\bullet_a\bullet_h}|+T^{\alpha }\| \Y^{\bullet_a\bullet_h}\|_{\alpha })+(|\Y_{0}^{\tddeuxa{$a$}{$h$}}\ \,|+T^{\alpha }\| \Y^{\tddeuxa{$a$}{$h$}}\ \,\|_{\alpha })\big)\| \X^{\bullet_h} \|_{\alpha }\nonumber\\
&+T^{2\alpha }\|R\Y^{\bullet_a}\|_{2\alpha }\Bigg)(|\Y_{0}^{\bullet_b\bullet_c}| +T^{\alpha }\| \Y^{\bullet_b\bullet_c}\|_{\alpha })(\|\X^{\bullet_a\bullet_b\bullet_c}\|_{3\alpha}+\|\X^{\bullet_b\bullet_c\bullet_a}\|_{3\alpha})\nonumber\\
&\ \hspace{7cm}(\text{by  ~(\ref{eq:yaalpha})})  \nonumber\\
&\ +\|F''\|_{\infty}\sum_{a,b,c\in A}\Bigg(|\Y_{0}^{\bullet_a}|+T^{\alpha}\sum_{h\in A}\bigg((|\Y_{0}^{\bullet_a\bullet_h}|+T^{\alpha }\| \Y^{\bullet_a\bullet_h}\|_{\alpha })+(|\Y_{0}^{\tddeuxa{$a$}{$h$}}\ \,|+T^{\alpha }\| \Y^{\tddeuxa{$a$}{$h$}}\ \,\|_{\alpha })\bigg)\| \X^{\bullet_h} \|_{\alpha }+T^{2\alpha }\|R\Y^{\bullet_a}\|_{2\alpha }\Bigg)\nonumber\\
&\ \hspace{0.47cm}(|\Y_{0}^{\tddeuxa{$b$}{$c$}}\ \,|+T^{\alpha }\| \Y^{\tddeuxa{$b$}{$c$}}\ \,\|_{\alpha })(\|\X^{\bullet_a\tddeuxa{$b$}{$c$}}\ \,\|_{3\alpha}+\|\X^{\tddeuxa{$b$}{$c$}\ \,\bullet_a}\|_{3\alpha})\nonumber\\
&\ \hspace{7cm}(\text{by  ~(\ref{eq:yaalpha})}) \nonumber\\
&\ +\|F''\|_{\infty}T^{\alpha}\Bigg(\sum_{a\in A}\Big(|\Y_{0}^{\bullet_a}|+T^{\alpha}\sum_{b\in A}\big((|\Y_{0}^{\bullet_a\bullet_b}|+T^{\alpha }\| \Y^{\bullet_a\bullet_b}\|_{\alpha })+(|\Y_{0}^{\tddeuxa{$a$}{$b$}}\ \,|+T^{\alpha }\| \Y^{\tddeuxa{$a$}{$b$}}\ \,\|_{\alpha })\big)\nonumber\\
&\ \hspace{0.47cm}\|\X^{\bullet_a}\|_{\alpha }+T^{2\alpha }\|R\Y^{\bullet_a}\|_{2\alpha }\Big)\|\X^{\bullet_b}\|_{\alpha}+T^{\alpha}\sum_{a,b\in A}(|\Y_{0}^{\bullet_a\bullet_b}|+T^{\alpha }\| \Y^{\bullet_a\bullet_b}\|_{\alpha })\,\|\X ^{\bullet_a\bullet_b}\|_{2\alpha}\nonumber\\
&\ +T^{\alpha}\sum_{a,b\in A}(|\Y_{0}^{\tddeuxa{$a$}{$b$}}\ \,|+T^{\alpha }\| \Y^{\tddeuxa{$a$}{$b$}}\ \,\|_{\alpha })\,\|\X^{\tddeuxa{$a$}{$b$}}\ \,|_{2\alpha}+T^{2\alpha}\|R\Y^{\etree}\|_{3\alpha}\Bigg)\|R\Y^{\etree}\|_{3\alpha}\nonumber\\
&\ \hspace{7cm} (\text{by  ~(\ref{eq:iequ3}), (\ref{eq:iequ2}) and (\ref{eq:iequ4})}) \nonumber\\
&\ +\dfrac{1}{6}\|F'''\|_{\infty} \Bigg(\sum_{a\in A}\Big(|\Y_{0}^{\bullet_a}|+T^{\alpha}\sum_{b\in A}\big((|\Y_{0}^{\bullet_a\bullet_b}|+T^{\alpha }\| \Y^{\bullet_a\bullet_b}\|_{\alpha })+(|\Y_{0}^{\tddeuxa{$a$}{$b$}}\ \,|+T^{\alpha }\| \Y^{\tddeuxa{$a$}{$b$}}\ \,\|_{\alpha })\big)\nonumber\\
&\ \hspace{0.47cm}\|\X^{\bullet_a}\|_{\alpha }+T^{2\alpha }\|R\Y^{\bullet_a}\|_{2\alpha }\Big)\|\X^{\bullet_b}\|_{\alpha}+T^{\alpha}\sum_{a,b\in A}(|\Y_{0}^{\bullet_a\bullet_b}|+T^{\alpha }\| \Y^{\bullet_a\bullet_b}\|_{\alpha })\,\|\X ^{\bullet_a\bullet_b}\|_{2\alpha}\nonumber\\
&\ +T^{\alpha}\sum_{a,b\in A}(|\Y_{0}^{\tddeuxa{$a$}{$b$}}\ \,|+T^{\alpha }\| \Y^{\tddeuxa{$a$}{$b$}}\ \,\|_{\alpha })\,\|\X^{\tddeuxa{$a$}{$b$}}\ \,|_{2\alpha}+T^{2\alpha}\|R\Y^{\etree}\|_{3\alpha}\Bigg)^3\mlabel{eq:simp3}\\
&\ \hspace{7cm} (\text{by  ~(\ref{eq:iequ3}), (\ref{eq:iequ2}) and (\ref{eq:iequ4})})\nonumber\\
%4
\le&\  3\|F\|_{C_b^3}\Big(1+T^{\alpha}+T^{2\alpha}+T^{3\alpha} \Big)^3\Big(1+\sum_{a\in A}|\Y_{0}^{\bullet_a}|+\sum_{a,b\in A}|\Y_{0}^{\bullet_a\bullet_b}|+\sum_{a,b\in A}|\Y_{0}^{\tddeuxa{$a$}{$b$}}\ \,| \Big)^3\nonumber\\
&\ \Big(1+\|R\Y^{\etree}\|_{3\alpha}+\sum_{a\in A}\|R\Y^{\bullet_a}\|_{2\alpha}+\sum_{a,b\in A}\|\Y^{\bullet_a\bullet_b}\|_{\alpha}+\sum_{a,b\in A}\|\Y^{\tddeuxa{$a$}{$b$}}\ \,\|_{\alpha} \Big)^3\nonumber\\
&\ \Big(1+\sum_{a\in A}\|\X^{\bullet_a}\|_{\alpha}+\sum_{a,b\in A}\|\X^{\bullet_a\bullet_b}\|_{2\alpha}+\sum_{a,b\in A}\|\X^{\tddeuxa{$a$}{$b$}}\ \,\|_{2\alpha}+\sum_{a,b,c\in A}\|\X^{\bullet_a\bullet_b\bullet_c}\|_{3\alpha}+\sum_{a,b,c\in A}\|\X^{\bullet_a\tddeuxa{$b$}{$c$}}\ \,\|_{3\alpha} \Big)^3 \nonumber\\
&\ \hspace{5cm} (\text{each of the items above appears below})\nonumber\\
=&\  3\|F\|_{C_b^3}\Big(1+T^{\alpha}+T^{2\alpha}+T^{3\alpha} \Big)^3\Big(1+\sum_{a\in A}|\Y_{0}^{\bullet_a}|+\sum_{a,b\in A}|\Y_{0}^{\bullet_a\bullet_b}|+\sum_{a,b\in A}|\Y_{0}^{\tddeuxa{$a$}{$b$}}\ \,| \Big)^3 \Big(1+\fan{\Y}_{\X;\alpha}\Big)^3 \Big(1+\fan{\X}_{\alpha}\Big)^3\nonumber\\
&\ \hspace{8cm} (\text{by  ~(\ref{eq:distx}) \text{ and } (\ref{eq:normy})})\nonumber\\
=&:\  \|F\|_{C_{b}^{3}}M(T, \sum_{a\in A}|\Y_{0}^{\bullet_a}|, \sum_{a,b\in A}|\Y_{0}^{\bullet_a\bullet_b}|, \sum_{b,c\in A}| \Y_{0}^{\tddeuxa{$a$}{$b$}}\ \,|, \fan{\Y} _{\X;\alpha }, \fan{\X} _{\alpha }).\nonumber
\end{align}
Consequently,
\begin{equation*}
\|R\Z^{\etree}\|_{3\alpha} = \sup_{s\ne t\in [ 0,T ] } \dfrac{|R\Z_{s,t}^{\etree}|}{|t-s|^{3\alpha}}\le \|F\|_{C_{b}^{3}}M(T, \sum_{a\in A}|\Y_{0}^{\bullet_a}|, \sum_{a,b\in A}|\Y_{0}^{\bullet_a\bullet_b}|, \sum_{a,b\in A}| \Y_{0}^{\tddeuxa{$a$}{$b$}}\ \,|, \fan{\Y} _{\X;\alpha }, \fan{\X} _{\alpha }).
\end{equation*}

\noindent{\bf Step 2:}
For the second term $\|R\Z^{\bullet_a}\|_{2\alpha}$, we obtain
\begin{align}
&\ |R\Z_{s,t}^{\bullet_a}|\nonumber\\
%1
=&\ |\Z_{s,t}^{\bullet_a}-\sum_{b\in A}(\Z_{s}^{\bullet_a\bullet_b}+\Z_{s}^{\tddeuxa{$a$}{$b$}}\ \,)\X_{s,t}^{\bullet_b}| \hspace{3cm}(\text{by  ~(\ref{eq:cbrp2})})\nonumber\\
%2
=&\ \Big|\Big(F'(\Y^{\etree}_{t})\Y^{\bullet_{a}}_{t}-F'(\Y^{\etree}_{s})\Y^{\bullet_{a}}_{s}\Big)-\sum_{b\in A}(\Z_{s}^{\bullet_a\bullet_b}+\Z_{s}^{\tddeuxa{$a$}{$b$}}\ \,)\X_{s,t}^{\bullet_b}\Big| \hspace{3cm}(\text{by ~(\ref{eq:zcbrp1})})\nonumber\\
%3
=&\ \Big|\Big(F'(\Y^{\etree}_{t})-F'(\Y^{\etree}_{s})\Big)\Y^{\bullet_{a}}_{s}+F'(\Y^{\etree}_{t})\Y^{\bullet_{a}}_{s,t}-\sum_{b\in A}(\Z_{s}^{\bullet_a\bullet_b}+\Z_{s}^{\tddeuxa{$a$}{$b$}}\ \,)\X_{s,t}^{\bullet_b}\Big|
\hspace{1cm}(\text{by combining like terms})\nonumber\\
%4
=&\ \Big|\Big(F'(\Y^{\etree}_{t})-F'(\Y^{\etree}_{s})\Big)\Y^{\bullet_{a}}_{s}+\Big(F'(\Y^{\etree}_{t})-F'(\Y^{\etree}_{s})\Big)\Y^{\bullet_{a}}_{s,t}+F'(\Y^{\etree}_{s})\Y^{\bullet_{a}}_{s,t}-\sum_{b\in A}(\Z_{s}^{\bullet_a\bullet_b}+\Z_{s}^{\tddeuxa{$a$}{$b$}}\ \,)\X_{s,t}^{\bullet_b}\Big|\nonumber\\
%5
=&\ \Big|\Big(F''(\Y^{\etree}_{s})\Y^{\etree}_{s,t}+\R{$\Int_0^1$}(1-\theta)F'''(\Y_{s}^{\etree}+\theta\Y_{s,t}^{\etree})(\Y_{s,t}^{\etree})^2d\theta\Big)\Y^{\bullet_{a}}_{s}\nonumber\\
&\ +\Big(F''(\Y^{\etree}_{s})\Y^{\etree}_{s,t}+\R{$\Int_0^1$}(1-\theta)F'''(\Y_{s}^{\etree}+\theta\Y_{s,t}^{\etree})(\Y_{s,t}^{\etree})^2d\theta\Big)\Y^{\bullet_{a}}_{s,t}+F'(\Y^{\etree}_{s})\Big(\sum_{b\in A}(\Y_{s}^{\bullet_a\bullet_b}+\Y_{s}^{\tddeuxa{$a$}{$b$}}\ \,)\X_{s,t}^{\bullet_b}+R\Y_{s,t}^{\bullet_a}\Big)\nonumber\\
&\ -\sum_{b\in A}\Big(\big(F'(\Y^{\etree}_{s})\Y^{\bullet_{a}\bullet_{b}}_{s}
+F''(\Y^{\etree}_{s})\Y^{\bullet_{a}}_{s}\Y^{\bullet_{b}}_{s}\big)+F'(\Y^{\etree}_{s})\Y_{s}^{\tddeuxa{$a$}{$b$}}\ \,\Big)\X_{s,t}^{\bullet_b}\Big|\nonumber\\
&\ \hspace{3cm} (\text{by the Taylor expression of $F'$ at $\Y_s^\etree$ and  ~(\ref{eq:cbrp2}) and (\ref{eq:zcbrp1})})\nonumber\\
%6
=&\ \Big|\Big(F''(\Y^{\etree}_{s})\Y^{\etree}_{s,t}+\R{$\Int_0^1$}(1-\theta)F'''(\Y_{s}^{\etree}+\theta\Y_{s,t}^{\etree})(\Y_{s,t}^{\etree})^2d\theta\Big)\Y^{\bullet_{a}}_{s}\nonumber\\
&\ +\Big(F''(\Y^{\etree}_{s})\Y^{\etree}_{s,t}+\R{$\Int_0^1$}(1-\theta)F'''(\Y_{s}^{\etree}+\theta\Y_{s,t}^{\etree})(\Y_{s,t}^{\etree})^2d\theta\Big)\Y^{\bullet_{a}}_{s,t}+F'(\Y^{\etree}_{s})R\Y_{s,t}^{\bullet_a}-\sum_{b\in A}F''(\Y^{\etree}_{s})\Y^{\bullet_a}_{s}\Y^{\bullet_{b}}_{s}\X_{s,t}^{\bullet_b}\Big|\nonumber\\
&\ \hspace{7cm} (\text{by combining like terms})\nonumber\\
%7
=&\ \Big|F''(\Y^{\etree}_{s})\Big(\sum_{a\in A}\Y_{s} ^{\bullet_a}\X_{s,t} ^{\bullet_a}+\sum_{a,b\in A}\Y_{s} ^{\bullet_a\bullet_b} \X_{s,t} ^{\bullet_a\bullet_b}+\sum_{a,b\in A}\Y_{s} ^{\tddeuxa{$a$}{$b$}}\ \, \X_{s,t}^{\tddeuxa{$a$}{$b$}}\ \,+R\Y^{\etree}_{s,t}\Big)\Y^{\bullet_{a}}_{s}\nonumber\\
&\ +\Y^{\bullet_{a}}_{s}\R{$\Int_0^1$}(1-\theta)F'''(\Y_{s}^{\etree}+\theta\Y_{s,t}^{\etree})(\Y_{s,t}^{\etree})^2d\theta+\Big(F''(\Y^{\etree}_{s})\Y^{\etree}_{s,t}+\R{$\Int_0^1$}(1-\theta)F'''(\Y_{s}^{\etree}+\theta\Y_{s,t}^{\etree})(\Y_{s,t}^{\etree})^2d\theta\Big)\Y^{\bullet_{a}}_{s,t}\nonumber\\
&\ +F'(\Y^{\etree}_{s})R\Y_{s,t}^{\bullet_a}-\sum_{b\in A}F''(\Y^{\etree}_{s})\Y^{\bullet_a}_{s}\Y^{\bullet_{b}}_{s}\X_{s,t}^{\bullet_b}
\Big|\hspace{2.5cm} (\text{by  ~(\ref{eq:cbrp1})})\nonumber\\
=&\ \Big|F''(\Y^{\etree}_{s})\Big(\sum_{a\in A}\Y_{s} ^{\bullet_a}\X_{s,t} ^{\bullet_a}+\sum_{a,b\in A}\Y_{s} ^{\bullet_a\bullet_b}\X_{s,t} ^{\bullet_a\bullet_b}+\sum_{a,b\in A}\Y_{s} ^{\tddeuxa{$a$}{$b$}}\ \, \X_{s,t}^{\tddeuxa{$a$}{$b$}}\ \,+R\Y^{\etree}_{s,t}\Big)\Y^{\bullet_{a}}_{s}\nonumber\\
&\ +\Y^{\bullet_{a}}_{s}\R{$\Int_0^1$}(1-\theta)F'''(\Y_{s}^{\etree}+\theta\Y_{s,t}^{\etree})(\Y_{s,t}^{\etree})^2d\theta\nonumber\\
&\ +\Big(F''(\Y^{\etree}_{s})\Y^{\etree}_{s,t}+
\R{$\Int_0^1$}(1-\theta)F'''(\Y_{s}^{\etree}+
\theta\Y_{s,t}^{\etree})(\Y_{s,t}^{\etree})^2d\theta\Big)\Y^{\bullet_{a}}_{s,t}
+F'(\Y^{\etree}_{s})R\Y_{s,t}^{\bullet_a}\Big| \mlabel{eq:simp2}\\
&\ \hspace{5cm} (\text{by combining like terms})\nonumber\\
\le&\  \|F''\|_{\infty}\Big(\sum_{a\in A}|\Y_{s} ^{\bullet_a}|\,|\X_{s,t} ^{\bullet_a}|+\sum_{a,b\in A}|\Y_{s} ^{\bullet_a\bullet_b}|\,|\X_{s,t} ^{\bullet_a\bullet_b}|+\sum_{a,b\in A}|\Y_{s} ^{\tddeuxa{$a$}{$b$}}\ \,|\, |\X_{s,t}^{\tddeuxa{$a$}{$b$}}\ \,|+|R\Y^{\etree}_{s,t}|\Big)|\Y^{\bullet_{a}}_{s}|\nonumber\\
&\ +\dfrac{1}{2}|\Y^{\bullet_{a}}_{s}|\,\|F'''\|_{\infty}|\Y_{s,t}^{\etree}|^2
+\|F''\|_{\infty}|\Y_{s,t}^{\etree}|\,|\Y_{s,t}^{\bullet_a}|
+\dfrac{1}{2}\|F'''\|_{\infty}|\Y_{s,t}^{\etree}|^2|\Y^{\bullet_{a}}_{s,t}|+ \|F'\|_{\infty}|\,|R\Y^{\bullet_a}_{s,t}|.\nonumber\\
& \hspace{5cm} (\text{by the triangle inequality,  ~(\ref{eq:infF}) and} \displaystyle\int_0^1 (1-\theta ) d\theta=\frac{1}{2})\nonumber
\end{align}
With a similar argument to  ~(\ref{eq:simp3}),
\begin{equation*}
\|R\Z^{\bullet_a}\|_{2\alpha}= \sup_{s\ne t\in [ 0,T ] } \dfrac{|R\Z_{s,t}^{\bullet_a}|}{|t-s|^{2\alpha}}\le \|F\|_{C_{b}^{3}}M(T, \sum_{a\in A}|\Y_{0}^{\bullet_a}|, \sum_{a,b\in A}|\Y_{0}^{\bullet_a\bullet_b}|, \sum_{a,b\in A}| \Y_{0}^{\tddeuxa{$a$}{$b$}}\ \,|, \fan{\Y} _{\X;\alpha }, \fan{\X} _{\alpha }).
\end{equation*}
Since $a\in A$ and $|A|=d$, we obtain
\begin{equation*}
\sum_{a\in A}\|R\Z^{\bullet_a}\|_{2\alpha }\le d\, \|F\|_{C_{b}^{3}}M(T, \sum_{a\in A}|\Y_{0}^{\bullet_a}|, \sum_{a,b\in A}|\Y_{0}^{\bullet_a\bullet_b}|, \sum_{a,b\in A}| \Y_{0}^{\tddeuxa{$a$}{$b$}}\ \,|, \fan{\Y} _{\X;\alpha }, \fan{\X} _{\alpha }).
\end{equation*}

\noindent{\bf Step 3:} For the third term $\|R\Z^{\tddeuxa{$a$}{$b$}}\ \,\|_{\alpha}$,
we get
\begin{align*}
&\ |R\Z_{s,t}^{\tddeuxa{$a$}{$b$}}\ \,| \\
=&\ | \Z_{s,t}^{\tddeuxa{$a$}{$b$}}\ \,| \hspace{6.4cm}(\text{by  ~(\ref{eq:cbrp4})})\\
=&\ |F'(\Y_{t}^{\etree})\Y_{t}^{\tddeuxa{$a$}{$b$}}\ \,-F'(\Y_{s}^{\etree})\Y_{s}^{\tddeuxa{$a$}{$b$}}\ \,| \hspace{3cm}(\text{by ~(\ref{eq:zcbrp1})})\\
=&\ |(F'(\Y_{t}^{\etree})-F'(\Y_{s}^{\etree}))\Y_{t}^{\tddeuxa{$a$}{$b$}}\ \,+F'(\Y_{s}^{\etree})\Y_{s,t}^{\tddeuxa{$a$}{$b$}}\ \,| \\
\le&\ \|F''\|_{\infty}|\Y_{s,t}^{\etree}|(|\Y_{0}^{\tddeuxa{$a$}{$b$}}\ \,|+T^{\alpha}\|\Y^{\tddeuxa{$a$}{$b$}}\ \,\|_{\alpha})+\|F'\|_{\infty}|\Y_{s,t}^{\tddeuxa{$a$}{$b$}}\ \,| \\
&\ \hspace{3cm} (\text{by the differential mean value theorem of $F'$ and~(\ref{eq:infF}), (\ref{eq:iequ3})})\\
=&\ \|F''\|_{\infty}\Big|\sum_{a\in A}\Y_{s} ^{\bullet_a}\X_{s,t} ^{\bullet_a}+\sum_{a,b\in A}\Y_{s} ^{\bullet_a\bullet_b} \X_{s,t} ^{\bullet_a\bullet_b} +
\sum_{a,b\in A}\Y_{s} ^{\tddeuxa{$a$}{$b$}}\ \, \X_{s,t} ^{\tddeuxa{$a$}{$b$}}\ \,+R\Y^{\etree}_{s,t}\Big|(|\Y_{0}^{\tddeuxa{$a$}{$b$}}\ \,|+T^{\alpha}\|\Y^{\tddeuxa{$a$}{$b$}}\ \,\|_{\alpha})+\|F'\|_{\infty}|\Y_{s,t}^{\tddeuxa{$a$}{$b$}}\ \,|\\
&\ \hspace{9cm}(\text{by  ~(\ref{eq:cbrp1})})\\
\le&\ \|F''\|_{\infty}\Big(\sum_{a\in A}|\Y_{s} ^{\bullet_a}|\,|\X_{s,t}^{\bullet_a}|+\sum_{a,b\in A}|\Y_{s} ^{\bullet_a\bullet_b}|\,|\X_{s,t} ^{\bullet_a\bullet_b}| +
\sum_{a,b\in A}|\Y_{s}^{\tddeuxa{$a$}{$b$}}\ \, |\,|\X_{s,t} ^{\tddeuxa{$a$}{$b$}}\ \,|+|R\Y^{\etree}_{s,t}|\Big)(|\Y_{0}^{\tddeuxa{$a$}{$b$}}\ \,|+T^{\alpha}\|\Y^{\tddeuxa{$a$}{$b$}}\ \,\|_{\alpha})\\
&\ +\|F'\|_{\infty}|\Y_{s,t}^{\tddeuxa{$a$}{$b$}}\ \,|\hspace{6cm}(\text{by the triangle inequality})\\
\le&\ \|F''\|_{\infty}\Bigg(\sum_{a\in A}\Big(|\Y_{0}^{\bullet_a}|+T^{\alpha}\sum_{b\in A}\Big((|\Y_{0}^{\bullet_a\bullet_b}|+T^{\alpha }\| \Y^{\bullet_a\bullet_b}\|_{\alpha })+(|\Y_{0}^{\tddeuxa{$a$}{$b$}}\ \,|+T^{\alpha }\| \Y^{\tddeuxa{$a$}{$b$}}\ \,\|_{\alpha })\Big)\| \X^{\bullet_b} \|_{\alpha }+T^{2\alpha }\|R\Y^{\bullet_e}\|_{2\alpha }\Big)\,|\X_{s,t}^{\bullet_a}|\\
&\ +\sum_{a,b\in A}(|\Y_{0}^{\bullet_a\bullet_b}|+T^{\alpha }\| \Y^{\bullet_a\bullet_b}\|_{\alpha })\,|\X_{s,t} ^{\bullet_a\bullet_b}| +
\sum_{a,b\in A}(|\Y_{0}^{\tddeuxa{$a$}{$b$}}\ \,|+T^{\alpha }\| \Y^{\tddeuxa{$a$}{$b$}}\ \,\|_{\alpha })\,|\X_{s,t} ^{\tddeuxa{$a$}{$b$}}\ \,|+|R\Y^{\etree}_{s,t}|\Bigg)(|\Y_{0}^{\tddeuxa{$a$}{$b$}}\ \,|+T^{\alpha}\|\Y^{\tddeuxa{$a$}{$b$}}\ \,\|_{\alpha})\\
&\ +\|F'\|_{\infty}|\Y_{s,t}^{\tddeuxa{$a$}{$b$}}\ \,|, \hspace{0.5cm}(\text{by  ~(\ref{eq:iequ3}), (\ref{eq:iequ2}) and (\ref{eq:iequ4})})
\end{align*}
which, together with  ~ ~(\ref{eq:normy}) and (\ref{eq:fcbn}), implies
\begin{equation*}
\sum_{a,b\in A}\|R\Z^{\tddeuxa{$a$}{$b$}}\ \,\|_{\alpha }=\sum_{a,b\in A}\sup_{s\ne t\in [ 0,T ] } \dfrac{|R\Z_{s,t}^{\tddeuxa{$a$}{$b$}}\ \,|}{|t-s|^{\alpha}} \le d^2\|F\|_{C_{b}^{3}}M(T, \sum_{a\in A}|\Y_{0}^{\bullet_a}|, \sum_{a,b\in A}|\Y_{0}^{\bullet_a\bullet_b}|, \sum_{a,b\in A}| \Y_{0}^{\tddeuxa{$a$}{$b$}}\ \,|, \fan{\Y} _{\X;\alpha }, \fan{\X} _{\alpha }).
\end{equation*}

\noindent{\bf Step 4:} For the fourth term $\|R\Z^{\bullet_a\bullet_b}\|_{\alpha}$,
we have
\begin{equation}
\begin{aligned}
|R\Z_{s,t}^{\bullet_a\bullet_b}|&=|\Z_{s,t}^{\bullet_a\bullet_b}| \hspace{3cm}(\text{by  ~(\ref{eq:cbrp3})})\\
&=|(F'(\Y_{t}^{\etree})\Y_{t}^{\bullet_a\bullet_b}+F''(\Y_{t}^{\etree})\Y_{t}^{\bullet_a}\Y_{t}^{\bullet_b})-(F'(\Y_{s}^{\etree})\Y_{s}^{\bullet_a\bullet_b}+F''(\Y_{s}^{\etree})\Y_{s}^{\bullet_a}\Y_{s}^{\bullet_b})| \\
&\hspace{8cm}(\text{by  ~(\ref{eq:zcbrp1})})\\
&\le|F'(\Y_{t}^{\etree})\Y_{t}^{\bullet_a\bullet_b}-F'(\Y_{s}^{\etree})\Y_{s}^{\bullet_a\bullet_b}|+|F''(\Y_{t}^{\etree})\Y_{t}^{\bullet_a}\Y_{t}^{\bullet_b}-F''(\Y_{s}^{\etree})\Y_{s}^{\bullet_a}\Y_{s}^{\bullet_b}|.
\end{aligned}
\mlabel{eq:rz4}
\end{equation}
For simplicity, denote by
$$I_1:=|F'(\Y_{t}^{\etree})\Y_{t}^{\bullet_a\bullet_b}-F'(\Y_{s}^{\etree})\Y_{s}^{\bullet_a\bullet_b}|\,\text{ and }\, I_2:=|F''(\Y_{t}^{\etree})\Y_{t}^{\bullet_a}\Y_{t}^{\bullet_b}-F''(\Y_{s}^{\etree})\Y_{s}^{\bullet_a}\Y_{s}^{\bullet_b}|.$$
We have
\begin{align*}
I_1 =&\ |(F'(\Y_{t}^{\etree})-F'(\Y_{s}^{\etree}))\Y_{t}^{\bullet_a\bullet_b}+F'( \Y_{s}^{\etree})\Y_{s,t}^{\bullet_a\bullet_b}| \\
\le&\ \|F''\|_{\infty}\Big(\sum_{a\in A}\Big(|\Y_{0}^{\bullet_a}|+T^{\alpha}\sum_{b\in A}\Big((|\Y_{0}^{\bullet_a\bullet_b}|+T^{\alpha }\| \Y^{\bullet_a\bullet_b}\|_{\alpha })\\
&+(|\Y_{0}^{\tddeuxa{$a$}{$b$}}\ \,|+T^{\alpha }\| \Y^{\tddeuxa{$a$}{$b$}}\ \,\|_{\alpha })\Big)\| \X^{\bullet_b}\|_{\alpha }+T^{2\alpha }\|R\Y^{\bullet_a}\|_{2\alpha }\Big)\,|\X_{s,t}^{\bullet_a}|+\sum_{a,b\in A}(|\Y_{0}^{\bullet_a\bullet_b}|+T^{\alpha }\| \Y^{\bullet_a\bullet_b}\|_{\alpha })\,|\X_{s,t} ^{\bullet_a\bullet_b}|\\
&+\sum_{a,b\in A}(|\Y_{0}^{\tddeuxa{$a$}{$b$}}\ \,|+T^{\alpha}\|\Y^{\tddeuxa{$a$}{$b$}}\ \,\|_{\alpha})\,|\X_{s,t}^{\tddeuxa{$a$}{$b$}}\ \,|+|R\Y^{\etree}_{s,t}|\Big)(|\Y_{0}^{\bullet_a\bullet_b}|+T^{\alpha}\|\Y^{\bullet_a\bullet_b}\|_{\alpha})\\
&\ +\|F'\|_{\infty}|\Y_{s,t}^{\bullet_a\bullet_b}|. \hspace{2cm} (\text{Similar to the process in Step 3 from the third equation})
\end{align*}
Moreover,
\begin{align*}
I_2
=&\ |(F''(\Y_{t}^{\etree})-F''(\Y_{s}^{\etree}))\Y_{t}^{\bullet_a}\Y_{t}^{\bullet_b}+F''(\Y_{s}^{\etree})(\Y_{t}^{\bullet_a}\Y_{t}^{\bullet_b}-\Y_{s}^{\bullet_a}\Y_{s}^{\bullet_b})|\\
=&\ |(F''(\Y_{t}^{\etree})-F''(\Y_{s}^{\etree}))\Y_{t}^{\bullet_a}\Y_{t}^{\bullet_b}+F''(\Y_{s}^{\etree})(\Y_{t}^{\bullet_a}\Y_{s,t}^{\bullet_b}+\Y_{s,t}^{\bullet_a}\Y_{s}^{\bullet_b})|\\
\le&\ \|F'''\|_{\infty}|\Y_{s,t}^{\etree}|\,|\Y_{t}^{\bullet_a}|\,|\Y_{t}^{\bullet_b}|
+\|F''\|_{\infty}(|\Y_{t}^{\bullet_a}|\,|\Y_{s,t}^{\bullet_b}|+|\Y_{s,t}^{\bullet_a}|\,|\Y_{s}^{\bullet_b}|)\\
&\   \hspace{1.5cm} (\text{by the differential mean value theorem of $F''$,   ~(\ref{eq:infF}) and the triangle inequality}) \\
=&\ \|F'''\|_{\infty}|\Y_{s,t}^{\etree}|\,|\Y_{t}^{\bullet_a}|\,|\Y_{t}^{\bullet_b}|+\|F''\|_{\infty}|\Y_{t}^{\bullet_a}|\,|\Y_{s,t}^{\bullet_b}|+\|F''\|_{\infty}|\Y_{s,t}^{\bullet_a}|\,|\Y_{s}^{\bullet_b}|\\
\le&\ \|F'''\|_{\infty }\Big|\sum_{a\in A}\Y_{s} ^{\bullet_a}\X_{s,t} ^{\bullet_a}+\sum_{a,b\in A}\Y_{s} ^{\bullet_a\bullet_b} \X_{s,t} ^{\bullet_a\bullet_b} +
\sum_{a,b\in A}\Y_{s} ^{\tddeuxa{$a$}{$b$}}\ \, \X_{s,t} ^{\tddeuxa{$a$}{$b$}}\ \,+R\Y^{\etree}_{s,t}\Big|\\
&\ \Bigg(|\Y_{0}^{\bullet_a}|+T^{\alpha}\sum_{b\in A}\Big((|\Y_{0}^{\bullet_a\bullet_b}|+T^{\alpha }\| \Y^{\bullet_a\bullet_b}\|_{\alpha })+(|\Y_{0}^{\tddeuxa{$a$}{$b$}}\ \,|+T^{\alpha }\| \Y^{\tddeuxa{$a$}{$b$}}\ \,\|_{\alpha })\Big)\| \X^{\bullet_a} \|_{\alpha }+T^{2\alpha }\|R\Y^{\bullet_a}\|_{2\alpha}\Bigg)\\
&\ \Bigg(|\Y_{0}^{\bullet_b}|+T^{\alpha}\sum_{c\in A}\Big((|\Y_{0}^{\bullet_b\bullet_c}|+T^{\alpha }\| \Y^{\bullet_b\bullet_c}\|_{\alpha })+(|\Y_{0}^{\tddeuxa{$b$}{$c$}}\ \,|+T^{\alpha }\| \Y^{\tddeuxa{$b$}{$c$}}\ \,\|_{\alpha })\Big)\| \X^{\bullet_b} \|_{\alpha }+T^{2\alpha }\|R\Y^{\bullet_b}\|_{2\alpha}\Bigg)\\
&\ \hspace{8cm} (\text{by  ~(\ref{eq:cbrp1}) and (\ref{eq:iequ4})})\\
&\ +\|F''\|_{\infty}\Bigg(|\Y_{0}^{\bullet_a}|+T^{\alpha}\sum_{b\in A}\Big((|\Y_{0}^{\bullet_a\bullet_b}|+T^{\alpha }\| \Y^{\bullet_a\bullet_b}\|_{\alpha })+(|\Y_{0}^{\tddeuxa{$a$}{$b$}}\ \,|+T^{\alpha }\| \Y^{\tddeuxa{$a$}{$b$}}\ \,\|_{\alpha })\Big)\| \X^{\bullet_a} \|_{\alpha }+T^{2\alpha }\|R\Y^{\bullet_a}\|_{2\alpha }\Bigg)\\
&\ \hspace{0.47cm}\Big|\sum_{c\in A}(\Y_{s}^{\bullet_b\bullet_c}+\Y_{s}^{\tddeuxa{$b$}{$c$}}\ \,)\X_{s,t}^{\bullet_c}+R\Y_{s,t}^{\bullet_b}\Big| \hspace{1cm}(\text{by  ~(\ref{eq:cbrp2}) and (\ref{eq:iequ4})})\\
&\ +\|F''\|_{\infty}\Big|\sum_{b\in A}(\Y_{s}^{\bullet_a\bullet_b}+\Y_{s}^{\tddeuxa{$a$}{$b$}}\ \,)\X_{s,t}^{\bullet_b}+R\Y_{s,t}^{\bullet_a}\Big| \Bigg(|\Y_{0}^{\bullet_b}|+T^{\alpha}\sum_{c\in A}\Big((|\Y_{0}^{\bullet_b\bullet_c}|+T^{\alpha }\| \Y^{\bullet_b\bullet_c}\|_{\alpha })+(|\Y_{0}^{\tddeuxa{$b$}{$c$}}\ \,|\\
&\ \hspace{0.47cm}+T^{\alpha}\|\Y^{\tddeuxa{$b$}{$c$}}\ \,\|_{\alpha })\Big)\| \X^{\bullet_b} \|_{\alpha }+T^{2\alpha }\|R\Y^{\bullet_b}\|_{2\alpha }\Bigg).
\hspace{1cm}(\text{by  ~(\ref{eq:cbrp2}) and (\ref{eq:iequ4})})
\end{align*}
Consequently,
\begin{align*}
\sum_{a,b\in A}\| R\Z^{\bullet_a\bullet_b}\|_{\alpha } =&\ \sum_{a,b\in A}\sup_{s\ne t\in [ 0,T ] } \dfrac{|R\Z_{s,t}^{\bullet_a\bullet_b}|}{|t-s|^{\alpha}}\le \sum_{a,b\in A}\sup_{s\ne t\in [0,T]} \dfrac{I_1}{|t-s|^{\alpha}}+\sum_{a,b\in A}\sup_{s\ne t\in [ 0,T ] } \dfrac{I_2}{|t-s|^{\alpha}}\\
\le&\ d^2\|F\|_{C_{b}^{3}}M(T, \sum_{a\in A}|\Y_{0}^{\bullet_a}|, \sum_{a,b\in A}|\Y_{0}^{\bullet_a\bullet_b}|, \sum_{a,b\in A}| \Y_{0}^{\tddeuxa{$a$}{$b$}}\ \,|, \fan{\Y} _{\X;\alpha }, \fan{\X} _{\alpha }).
\end{align*}
Here the last step employs  ~(\ref{eq:distx}), (\ref{eq:normy}) and substituting the above bounds about $I_1$ and $I_2$ into ~(\ref{eq:rz4}).

Summing up the above estimates of the four steps, we conclude
\begin{align*}
\fan{\Z}_{\X;\alpha }=&\ \|R\Z^{\etree }\|_{3\alpha}+\sum_{a\in A}\|R\Z^{\bullet_a}\|_{2\alpha }+\sum_{a,b\in A}\|R\Z^{\tddeuxa{$a$}{$b$}}\ \,\|_{\alpha }+\sum_{a,b\in A}\| R\Z^{\bullet_a\bullet_b}\|_{\alpha}\\
\le &\  4d^2\|F\|_{C_{b}^{3}}M(T, \sum_{a\in A}|\Y_{0}^{\bullet_a}|, \sum_{a,b\in A}|\Y_{0}^{\bullet_a\bullet_b}|, \sum_{a,b\in A}| \Y_{0}^{\tddeuxa{$a$}{$b$}}\ \,|, \fan{\Y} _{\X;\alpha }, \fan{\X} _{\alpha }). \qedhere
\end{align*}
\end{proof}
%%%
\subsection{Stability of the composition with regular functions}\label{sec2.3}
This subsection is dedicated to establishing the stability of controlled planarly branched rough paths, as detailed below.

\begin{theorem}
Let \X, $\tilde{\X} \in \brpt$, $\Y\in \cbrpxt$, $\tilde{\Y}\in \cbrptxt$ and $F\in C_{b}^{3}(\RR^n; \RR^n) $. Let $\Z=F(\Y) $ and $\tilde{\Z}=F(\tilde{\Y})$ obtained in  ~(\mref{eq:zcbrp1}).
Then
\begin{align*}
\fan {\Z,\tilde{\Z}}_{\X,\tilde{\X};\alpha} \le 4d^2\|F\|_{C_{b}^{3}} M\Big( & T, \sum_{a\in A}|\Y_{0}^{\bullet_a}|, \sum_{a\in A}|\tilde\Y_{0}^{\bullet_a}|, \sum_{a,b\in A}|\Y_{0}^{\bullet_a\bullet_b}|, \sum_{a,b\in A}|\tilde\Y_{0}^{\bullet_a\bullet_b}|, \sum_{a,b\in A}| \Y_{0}^{\tddeuxa{$a$}{$b$}}\ \,|, \sum_{a,b\in A}|\tilde\Y_{0}^{\tddeuxa{$a$}{$b$}}\ \,|, \\
&\ \hspace{-1cm}\fan{\Y} _{\X;\alpha }, \fan{\tilde\Y} _{\X;\alpha},\fan{\X} _{\alpha}, \fan{\tilde\X}_{\alpha}\Big)\times \Big(\fan{\Y,\tilde{\Y} }_{\X,\tilde{\X};\alpha  }+\fan{\X,\tilde{\X} }_{\alpha}+|\Y_0^\etree -\tilde{\Y}^\etree_0|\Big ).
\end{align*}
\mlabel{thm:stab2}
\end{theorem}

\begin{proof}
Invoking  ~(\mref{eq:dist}),
\begin{equation}
\fan {\Z,\tilde{\Z}}_{\X,\tilde{\X};\alpha} =\|R\Z^{\etree}-R\tilde\Z^{\etree}\| _{3\alpha} +\sum_{a\in A}\|R\Z^{\bullet_a}-R\tilde\Z^{\bullet_a}\|_{2\alpha}+\sum_{a,b\in A}\|R\Z^{\bullet_a\bullet_b}-R\tilde\Z^{\bullet_a\bullet_b}\|_{\alpha}+\sum_{a,b\in A}\|R\Z^{\tddeuxa{$a$}{$b$}}\ \,-R\tilde\Z^{\tddeuxa{$a$}{$b$}}\ \,\|_{\alpha}.
\mlabel{eq:ztz}
\end{equation}
To accomplish the proof, we estimate the four terms on the right hand side of the above equation.\\

\noindent{\bf Step 1:} For the first term $\| R\Z^{\etree }-R\tilde\Z^{\etree } \| _{3\alpha }$,
\begin{align*}
&\ |R\Z_{s,t}^{\etree}-R\tilde{\Z}_{s,t}^{\etree} | \\
%1st
=&\ \Big|\Big(F''(\Y_{s}^{\etree})\sum_{a,b,c\in A}\Y_{s}^{\bullet_a}\Y_{s}^{\bullet_b\bullet_c}(\X_{s,t} ^{\bullet_a\bullet_b\bullet_c}+\X_{s,t} ^{\bullet_b\bullet_c\bullet_a})+F''(\Y_{s}^{\etree})\sum_{a,b,c\in A}\Y_{s}^{\bullet_a}\Y_{s}^{\tddeuxa{$b$}{$c$}}\ \,(\X_{s,t} ^{\bullet_a\tddeuxa{$b$}{$c$}}\ \,+\X_{s,t} ^{\tddeuxa{$b$}{$c$}\ \,\bullet_a})+F'(\Y_{s}^{\etree})R\Y^{\etree}_{s,t}\\
&\ +\displaystyle\int_0^1\frac{(1-\theta )^2}{2} F'''(\Y_{s}^{\etree}+\theta \Y_{s,t}^{\etree})(\Y_{s,t}^{\etree}) ^3d\theta \Big)-\Big(F''(\tilde\Y_{s}^{\etree})\sum_{a,b,c\in A}\tilde\Y_{s}^{\bullet_a}\tilde\Y_{s}^{\bullet_b\bullet_c}(\tilde\X_{s,t} ^{\bullet_a\bullet_b\bullet_c}+\tilde\X_{s,t} ^{\bullet_b\bullet_c\bullet_a})\\
&\ +F''(\tilde\Y_{s}^{\etree})\sum_{a,b,c\in A}\tilde\Y_{s}^{\bullet_a}\tilde\Y_{s}^{\tddeuxa{$b$}{$c$}}\ \,(\tilde\X_{s,t} ^{\bullet_a\tddeuxa{$b$}{$c$}}\ \,+\tilde\X_{s,t} ^{\tddeuxa{$b$}{$c$}\ \,\bullet_a})+F'(\tilde\Y_{s}^{\etree})R\tilde\Y^{\etree}_{s,t}+\displaystyle\int_0^1\frac{(1-\theta )^2}{2} F'''(\tilde\Y_{s}^{\etree} +\theta\tilde\Y_{s,t}^{\etree})(\tilde\Y_{s,t}^{\etree})^3d\theta \Big)\Big|\\
&\hspace{7cm} (\text{by  ~(\mref{eq:simp1})})\\
%2nd
\le&\ |F''(\Y_{s}^{\etree})|\sum_{a,b,c\in A}|\Y_{s}^{\bullet_a}|\,|\Y_{s}^{\bullet_b\bullet_c}|(|\X_{s,t} ^{\bullet_a\bullet_b\bullet_c}|+|\X_{s,t} ^{\bullet_b\bullet_c\bullet_a}|)+|F''(\tilde\Y_{s}^{\etree})|\sum_{a,b,c\in A}|\tilde\Y_{s}^{\bullet_a}|\,|\tilde\Y_{s}^{\bullet_b\bullet_c}|(|\tilde\X_{s,t} ^{\bullet_a\bullet_b\bullet_c}|+|\tilde\X_{s,t} ^{\bullet_b\bullet_c\bullet_a}|)\\
&\ +|F''(\Y_{s}^{\etree})|\sum_{a,b,c\in A}|\Y_{s}^{\bullet_a}|\,|\Y_{s}^{\tddeuxa{$b$}{$c$}}\ \,|(|\X_{s,t} ^{\bullet_a\tddeuxa{$b$}{$c$}}\ \,|+|\X_{s,t} ^{\tddeuxa{$b$}{$c$}\ \,\bullet_a}|)+|F''(\tilde\Y_{s}^{\etree})|\sum_{a,b,c\in A}|\tilde\Y_{s}^{\bullet_a}|\,|\tilde\Y_{s}^{\tddeuxa{$b$}{$c$}}\ \,|(|\tilde\X_{s,t} ^{\bullet_a\tddeuxa{$b$}{$c$}}\ \,|+|\tilde\X_{s,t} ^{\tddeuxa{$b$}{$c$}\ \,\bullet_a}|)\\
&\ +|F'(\Y_{s}^{\etree})R\Y_{s,t}^{\etree}
-F'(\tilde\Y_{s}^{\etree})R\tilde\Y_{s,t}^{\etree}|+\dfrac{1}{6}\|F'''\|_{\infty}|\Y_{s,t}^{\etree}|^3+\dfrac{1}{6}\|F'''\|_{\infty} |\tilde\Y_{s,t}^{\etree}|^3 \\
& \hspace{5cm} (\text{by the triangle inequality,  ~(\ref{eq:infF}) and $\displaystyle\int_0^1 \frac{(1-\theta )^2}{2}d\theta =\frac{1}{6}$})\\
%3rd
\le&\ \|F''\|_{\infty}\sum_{a,b,c\in A}(|\Y_{0}^{\bullet_a}|+T^{\alpha}\|\Y^{\bullet_a}\|_{\alpha})(|\Y_{0}^{\bullet_b\bullet_c}|+T^{\alpha }\| \Y^{\bullet_b\bullet_c}\|_{\alpha })(|\X_{s,t} ^{\bullet_a\bullet_b\bullet_c}|+|\X_{s,t} ^{\bullet_b\bullet_c\bullet_a}|)\\
&\ +\|F''\|_{\infty}\sum_{a,b,c\in A}(|\tilde\Y_{0}^{\bullet_a}|+T^{\alpha}\|\tilde\Y^{\bullet_a}\|_{\alpha})(|\tilde\Y_{0}^{\bullet_b\bullet_c}|+T^{\alpha }\|\tilde\Y^{\bullet_b\bullet_c}\|_{\alpha })(|\tilde\X_{s,t} ^{\bullet_a\bullet_b\bullet_c}|+|\tilde\X_{s,t} ^{\bullet_b\bullet_c\bullet_a}|)\\
&\ \hspace{7cm}(\text{by  ~(\ref{eq:cbrp3}), (\ref{eq:cbrp4}) and~(\ref{eq:iequ4})})\\
&\ +\|F''\|_{\infty}\sum_{a,b,c\in A}(|\Y_{0}^{\bullet_a}|+T^{\alpha}\|\Y^{\bullet_a}\|_{\alpha})(|\Y_{0}^{\tddeuxa{$b$}{$c$}}\ \,|+T^{\alpha }\| \Y^{\tddeuxa{$b$}{$c$}}\ \,\|_{\alpha })(|\X_{s,t} ^{\bullet_a\tddeuxa{$b$}{$c$}}\ \,|+|\X_{s,t} ^{\tddeuxa{$b$}{$c$}\ \,\bullet_a}|)\\
&\ +\|F''\|_{\infty}\sum_{a,b,c\in A}(|\tilde\Y_{0}^{\bullet_a}|+T^{\alpha}\|\tilde\Y^{\bullet_a}\|_{\alpha})(|\tilde\Y_{0}^{\tddeuxa{$b$}{$c$}}\ \,|+T^{\alpha }\|\tilde\Y^{\tddeuxa{$b$}{$c$}}\ \,\|_{\alpha })(|\tilde\X_{s,t} ^{\bullet_a\tddeuxa{$b$}{$c$}}\ \,|+|\tilde\X_{s,t} ^{\tddeuxa{$b$}{$c$}\ \,\bullet_a}|)\\
&\ +|F'(\Y_{s}^{\etree})R\Y_{s,t}^{\etree}-F'(\tilde\Y_{s}^{\etree})R\tilde\Y_{s,t}^{\etree}|
+\dfrac{1}{6}\|F'''\|_{\infty}|\Y_{s,t}^{\etree}|^3+\dfrac{1}{6}\|F'''\|_{\infty} |\tilde\Y_{s,t}^{\etree}|^3.\\
&\ \hspace{8cm} (\text{by ~(\ref{eq:iequ3}), (\ref{eq:iequ2}) and (\ref{eq:iequ4})})
\end{align*}
By  ~(\ref{eq:distx}) and (\ref{eq:normy}), all the terms in the right hand side are in the required form except the last three summands, which will be handled as follows.
Firstly,
\begin{align}
&\ \dfrac{| F'( \Y_{s}^{\etree} ) R\Y_{s,t}^{\etree }-F'( \tilde \Y_{s}^{\etree} ) R\tilde \Y_{s,t}^{\etree } |}{|t-s|^{3\alpha}} \mlabel{eq:step1a}\\
\le&\  \dfrac{|F'(\Y_{s}^{\etree})- F'(\tilde\Y_{s}^{\etree})|\,|R\Y_{s,t}^{\etree}|+ F'(\tilde\Y_{s}^{\etree})|\,|R\Y_{s,t}^{\etree }-R\tilde \Y_{s,t}^{\etree}|}{|t-s|^{3\alpha}} \hspace{1cm} (\text{by the triangle inequality})\nonumber\\
\le&\  \dfrac{2\|F'\|_{\infty}|R\Y_{s,t}^{\etree } |+\|F'\|_{\infty}|R\Y_{s,t}^{\etree}-R\tilde \Y_{s,t}^{\etree}|}{|t-s|^{3\alpha}}  \hspace{1cm} (\text{by the triangle inequality and  ~(\ref{eq:infF})})\nonumber\\
=&\ 2\|F'\|_{\infty}\|R\Y^{\etree}\|_{3\alpha}+\|F'\|_{\infty}\|R\Y^{\etree}-R\tilde\Y^{\etree}\|_{3\alpha}\nonumber\\
\le&\ \|F\|_{C_{b}^{3}}M(\fan{\Y} _{\X;\alpha })+\|F\|_{C_{b}^{3}}\|R\Y^{\etree}-R\tilde\Y^{\etree}\|_{3\alpha}.\nonumber
\end{align}

Secondly, thanks to Theorem~\mref{thm:stab1},
\begin{align*}
&\ \dfrac{1}{6}\dfrac{\|F'''\|_{\infty}|\Y_{s,t}^{\etree}|^3}{|t-s|^{3\alpha}}\\
=&\ \dfrac{1}{6}\frac{\|F'''\|_{\infty}\Big|\sum_{a\in A}\Y_{s} ^{\bullet_a}\X_{s,t} ^{\bullet_a}+\sum_{a,b\in A}\Y_{s} ^{\bullet_a\bullet_b} \X_{s,t} ^{\bullet_a\bullet_b}+
\sum_{a,b\in A}\Y_{s} ^{\tddeuxa{$a$}{$b$}}\ \, \X_{s,t} ^{\tddeuxa{$a$}{$b$}}\ \,+R\Y^{\etree}_{s,t}\Big|^3}{|t-s|^{3\alpha}}\hspace{0.5cm} \text{(by  ~(\ref{eq:cbrp1}))}\\
\le&\ \dfrac{1}{6}\|F'''\|_{\infty}\Big(\sum_{a\in A}|\Y_{s}^{\bullet_a}|\,\|\X^{\bullet_a}\|_{\alpha}+T^{\alpha}\sum_{a,b\in A}|\Y_{s}^{\bullet_a\bullet_b}|\,\|\X^{\bullet_a\bullet_b} \|_{2\alpha}+T^{\alpha}\sum_{a,b\in A}|\Y_{s}^{\tddeuxa{$a$}{$b$}}\ \,|\,\|\X^{\tddeuxa{$a$}{$b$}}\ \,\|_{2\alpha}+T^{2\alpha}\|R\Y^{\etree}\|_{3\alpha}\Big)^3\\
&\hspace{6cm} (\text{by the triangle inequality,  ~(\ref{eq:normx}) and (\ref{eq:remainry})})\\
=&\ \|F\|_{C_{b}^{3}}M(T, \sum_{a\in A}|\Y_{0}^{\bullet_a}|, \sum_{a,b\in A}|\Y_{0}^{\bullet_a\bullet_b}|, \sum_{a,b\in A}| \Y_{0}^{\tddeuxa{$a$}{$b$}}\ \,|, \fan{\Y} _{\X;\alpha }, \fan{\X} _{\alpha }),\\
&\hspace{6cm} (\text{by  ~(\ref{eq:distx}), (\ref{eq:remain}), (\ref{eq:cbrp3}), (\ref{eq:cbrp4}}), (\ref{eq:fcbn})\,\text{ and }\, (\ref{eq:iequ4}))
\end{align*}
and by symmetry,
\begin{equation*}
\dfrac{1}{6}\dfrac{\|F'''\|_{\infty}|\tilde\Y_{s,t}^{\etree}|^3}{|t-s|^{3\alpha}}\le\|F\|_{C_{b}^{3}}M(T, \sum_{a\in A}|\tilde\Y_{0}^{\bullet_a}|, \sum_{a,b\in A}|\tilde\Y_{0}^{\bullet_a\bullet_b}|, \sum_{a,b\in A}|\tilde\Y_{0}^{\tddeuxa{$a$}{$b$}}\ \,|, \fan{\tilde\Y} _{\X;\alpha }, \fan{\tilde\X}_{\alpha }).
\end{equation*}
With the above bounds for the last three summands, we can carry the calculations further
\begin{equation}
\begin{aligned}
&\ \|R\Z^{\etree}-R\tilde{\Z}^{\etree}\|_{3\alpha } =\sup_{s\ne t\in [ 0,T]} \dfrac{|R\Z_{s,t}^{\etree}-R\tilde{\Z}_{s,t}^{\etree} |}{|t-s|^{3\alpha}} \\
\le&\
d^2\|F\|_{C_{b}^{3}} M\Big(T, \sum_{b\in A}|\Y_{0}^{\bullet_b}|, \sum_{a\in A}|\tilde\Y_{0}^{\bullet_a}|, \sum_{a,b\in A}|\Y_{0}^{\bullet_a\bullet_b}|, \sum_{a,b\in A}|\tilde\Y_{0}^{\bullet_a\bullet_b}|, \sum_{a,b\in A}| \Y_{0}^{\tddeuxa{$a$}{$b$}}\ \,|, \sum_{a,b\in A}|\tilde\Y_{0}^{\tddeuxa{$a$}{$b$}}\ \,|, \fan{\Y} _{\X;\alpha }, \fan{\tilde\Y} _{\X;\alpha},\\
&\ \hspace{2cm}\fan{\X} _{\alpha}, \fan{\tilde\X}_{\alpha}\Big)+d^2\|F\|_{C_{b}^{3}}\|R\Y^{\etree }-R\tilde\Y^{\etree }\|_{3\alpha}.
\end{aligned}
\mlabel{eq:step1}
\end{equation}

\noindent{\bf Step 2:} For $\|R\Z^{\bullet_a}-R\tilde\Z^{\bullet_a}\|_{2\alpha}$, we have
\begin{align*}
|R\Z_{s,t}^{\bullet_a}-R\tilde{\Z}_{s,t}^{\bullet_a}|
=&\ \Big|\Bigg(F''(\Y^{\etree}_{s})\Big(\sum_{a\in A}\Y_{s} ^{\bullet_a}\X_{s,t} ^{\bullet_a}+\sum_{a,b\in A}\Y_{s} ^{\bullet_a\bullet_b}\X_{s,t} ^{\bullet_a\bullet_b}+\sum_{a,b\in A}\Y_{s} ^{\tddeuxa{$a$}{$b$}}\ \, \X_{s,t}^{\tddeuxa{$a$}{$b$}}\ \,+R\Y^{\etree}_{s,t}\Big)\Y^{\bullet_{a}}_{s}\\
&\ +\Y^{\bullet_{a}}_{s}\displaystyle\int_0^1(1-\theta)F'''(\Y_{s}^{\etree}+\theta\Y_{s,t}^{\etree})(\Y_{s,t}^{\etree})^2d\theta\\
&\ +\Big(F''(\Y^{\etree}_{s})\Y^{\etree}_{s,t}
+\displaystyle\int_0^1(1-\theta)F'''(\Y_{s}^{\etree}+\theta\Y_{s,t}^{\etree})(\Y_{s,t}^{\etree})^2d\theta\Big)\Y^{\bullet_{a}}_{s,t}+F'(\Y^{\etree}_{s})R\Y_{s,t}^{\bullet_a}\Bigg)\\
&\ -\Bigg(F''(\tilde\Y^{\etree}_{s})\Big(\sum_{a\in A}\tilde\Y_{s} ^{\bullet_a}\tilde\X_{s,t} ^{\bullet_a}+\sum_{a,b\in A}\tilde\Y_{s} ^{\bullet_a\bullet_b}\tilde\X_{s,t} ^{\bullet_a\bullet_b}+\sum_{a,b\in A}\tilde\Y_{s} ^{\tddeuxa{$a$}{$b$}}\ \, \tilde\X_{s,t}^{\tddeuxa{$a$}{$b$}}\ \, +R\tilde\Y^{\etree}_{s,t}\Big)\tilde\Y^{\bullet_{a}}_{s}\\
&\ +\tilde\Y^{\bullet_{a}}_{s}\displaystyle\int_0^1(1-\theta)F'''(\tilde\Y_{s}^{\etree}+\theta\tilde\Y_{s,t}^{\etree})(\tilde\Y_{s,t}^{\etree})^2d\theta\\
&\ +\Big(F''(\tilde\Y^{\etree}_{s})\tilde\Y^{\etree}_{s,t}
+\displaystyle\int_0^1(1-\theta)F'''(\tilde\Y_{s}^{\etree}+\theta\tilde\Y_{s,t}^{\etree})(\tilde\Y_{s,t}^{\etree})^2d\theta\Big)\tilde\Y^{\bullet_{a}}_{s,t}+F'(\tilde\Y^{\etree}_{s})R\tilde\Y_{s,t}^{\bullet_a}\Bigg)\Big|\\
&\hspace{7cm} (\text{by  ~(\mref{eq:simp2})})\\
\le&\ \|F''\|_{\infty}\Big(\sum_{a\in A}|\Y_{s}^{\bullet_a}|\,|\X_{s,t} ^{\bullet_a}|+\sum_{a,b\in A}|\Y_{s} ^{\bullet_a\bullet_b}|\,|\X_{s,t} ^{\bullet_a\bullet_b}|+\sum_{a,b\in A}|\Y_{s} ^{\tddeuxa{$a$}{$b$}}\ \,| \,|\X_{s,t}^{\tddeuxa{$a$}{$b$}}\ \,|+|R\Y_{s,t}^{\etree}|\Big)|\Y^{\bullet_{a}}_{s}|\\
&\ +\|F''\|_{\infty}\Big(\sum_{a\in A}|\tilde\Y_{s}^{\bullet_a}|\,|\tilde\X_{s,t} ^{\bullet_a}|+\sum_{a,b\in A}|\tilde\Y_{s} ^{\bullet_a\bullet_b}|\,|\tilde\X_{s,t} ^{\bullet_a\bullet_b}|+\sum_{a,b\in A}|\tilde\Y_{s} ^{\tddeuxa{$a$}{$b$}}\ \,| \,|\tilde\X_{s,t}^{\tddeuxa{$a$}{$b$}}\ \,|+|R\tilde\Y_{s,t}^{\etree}|\Big)|\tilde\Y^{\bullet_{a}}_{s}|\\
&\ +\dfrac{1}{2}\|F'''\|_{\infty}|\Y^{\bullet_{a}}_{s}|\,|\Y_{s,t}^{\etree}|^2+\dfrac{1}{2}\|F'''\|_{\infty}|\tilde\Y^{\bullet_{a}}_{s}|\,|\tilde\Y_{s,t}^{\etree}|^2+\|F''\|_{\infty}|\Y_{s,t}^{\etree}|\,|\Y_{s,t}^{\bullet_{a}}|+\|F''\|_{\infty}|\tilde\Y_{s,t}^{\etree}|\,|\tilde\Y_{s,t}^{\bullet_{a}}| \\
&\ +\dfrac{1}{2}\|F'''\|_{\infty}|\Y_{s,t}^{\etree}|^2|\Y^{\bullet_{a}}_{s,t}|
+\dfrac{1}{2}\|F'''\|_{\infty}|\tilde\Y_{s,t}^{\etree}|^2|\tilde\Y^{\bullet_{a}}_{s,t}|+|F'(\Y_{s}^{\etree})R\Y_{s,t}^{\bullet_a}-F'(\tilde \Y_{s}^{\etree})R\tilde \Y_{s,t}^{\bullet_a}|.\\	
&\hspace{0.2cm} (\text{by the triangle inequality,  ~(\ref{eq:infF}), combining like terms and $\displaystyle\int_0^1 (1-\theta ) d\theta=\frac{1}{2}$})
\end{align*}
By~(\ref{eq:iequ3}), (\ref{eq:iequ2}) and (\ref{eq:iequ4}), all the terms in the right hand side are in the required form except the summands in the last two lines. With the same argument in Step 1 dealing with the last three summands, we conclude
\begin{align*}
&\ \|R\Z^{\bullet_a}-R\tilde{\Z}^{\bullet_a}\|_{2\alpha} =\sup_{s\ne t\in [ 0,T]} \dfrac{|R\Z_{s,t}^{\bullet_a}-R\tilde{\Z}_{s,t}^{\bullet_a} |}{|t-s|^{2\alpha}} \\
\le&\
\|F\|_{C_{b}^{3}} M\Big(T, \sum_{a\in A}|\Y_{0}^{\bullet_a}|, \sum_{a\in A}|\tilde\Y_{0}^{\bullet_a}|, \sum_{a,b\in A}|\Y_{0}^{\bullet_a\bullet_b}|, \sum_{a,b\in A}|\tilde\Y_{0}^{\bullet_a\bullet_b}|, \sum_{a,b\in A}| \Y_{0}^{\tddeuxa{$a$}{$b$}}\ \,|, \sum_{a,b\in A}|\tilde\Y_{0}^{\tddeuxa{$a$}{$b$}}\ \,|, \fan{\Y} _{\X;\alpha }, \fan{\tilde\Y} _{\X;\alpha},\\
&\ \hspace{1.7cm} \fan{\X} _{\alpha}, \fan{\tilde\X}_{\alpha}\Big)+\|F\|_{C_{b}^{3} }\|R\Y^{\bullet_a}-R\tilde\Y^{\bullet_a}\|_{2\alpha}.
\end{align*}
Summing up and using $d\le d^2$, we obtain
\begin{equation}
\begin{aligned}
&\ \sum_{a\in A}\|R\Z^{\bullet_a}-R\tilde{\Z}^{\bullet_a}\|_{2\alpha}\\
\le&\
d^2\|F\|_{C_{b}^{3}} M\Big(T, \sum_{a\in A}|\Y_{0}^{\bullet_a}|, \sum_{a\in A}|\tilde\Y_{0}^{\bullet_a}|, \sum_{a,b\in A}|\Y_{0}^{\bullet_a\bullet_b}|, \sum_{a,b\in A}|\tilde\Y_{0}^{\bullet_a\bullet_b}|, \sum_{a,b\in A}| \Y_{0}^{\tddeuxa{$a$}{$b$}}\ \,|, \sum_{a,b\in A}|\tilde\Y_{0}^{\tddeuxa{$a$}{$b$}}\ \,|, \fan{\Y} _{\X;\alpha },\\
&\ \hspace{1.9cm} \fan{\tilde\Y} _{\X;\alpha}, \fan{\X} _{\alpha}, \fan{\tilde\X}_{\alpha}\Big)+d^2\|F\|_{C_{b}^{3} }\sum_{a\in A}\|R\Y^{\bullet_a}-R\tilde\Y^{\bullet_a}\|_{2\alpha}.
\end{aligned}
\mlabel{eq:step2}
\end{equation}
	
\noindent{\bf Step 3:} For the $\|R\Z^{\tddeuxa{$a$}{$b$}}\ \,-R\tilde\Z^{\tddeuxa{$a$}{$b$}}\ \,\|_{\alpha}$, we have
\begin{align*}
&\ |R\Z_{s,t}^{\tddeuxa{$a$}{$b$}}\ \,-R\tilde\Z_{s,t}^{\tddeuxa{$a$}{$b$}}\ \,| = |\Z_{s,t}^{\tddeuxa{$a$}{$b$}}\ \,-\tilde\Z_{s,t}^{\tddeuxa{$a$}{$b$}}\ \,| \hspace{1cm}(\text{by  ~(\ref{eq:cbrp3})})\\
=&\ \Big|\Big((F'(\Y_{t}^{\etree})-F'(\Y_{s}^{\etree}))\Y_{t}^{\tddeuxa{$a$}{$b$}}\ \,+F'(\Y_{s}^{\etree})\Y_{s,t}^{\tddeuxa{$a$}{$b$}}\ \,\Big)-\Big((F'(\tilde\Y_{t}^{\etree})-F'(\tilde \Y_{s}^{\etree}))\tilde\Y_{t}^{\tddeuxa{$a$}{$b$}}\ \,+F'(\tilde\Y_{s}^{\etree})\tilde\Y_{s,t}^{\tddeuxa{$a$}{$b$}}\ \,\Big)\Big| \\
&\hspace{7cm} (\text{by  ~(\mref{eq:zcbrp1}}))\\
\le&\ |(F'(\Y_{t}^{\etree})-F'(\Y_{s}^{\etree}))\Y_{t}^{\tddeuxa{$a$}{$b$}}\ \,|+|(F'(\tilde\Y_{t}^{\etree})-F'(\tilde \Y_{s}^{\etree}))\tilde \Y_{t}^{\tddeuxa{$a$}{$b$}}\ \,|+|F'(\Y_{s}^{\etree})\Y_{s,t}^{\tddeuxa{$a$}{$b$}}\ \,-F'(\tilde\Y_{s}^{\etree})\tilde\Y_{s,t}^{\tddeuxa{$a$}{$b$}}\ \,| \\
&\ \hspace{7cm} (\text{by the triangle inequality}) \\
\le&\ \|F''\|_{\infty}|\Y^{\etree}_{s,t}|(|\Y_{0}^{\tddeuxa{$a$}{$b$}}\ \, |+T^{\alpha}\|\Y^{\tddeuxa{$a$}{$b$}}\ \,\|_{\alpha})+\|F''\|_{\infty}|\tilde \Y_{s,t}^{\etree}|(| \tilde \Y_{0}^{\tddeuxa{$a$}{$b$}}\ \,|+T^{\alpha}\|\tilde\Y^{\tddeuxa{$a$}{$b$}}\ \,\|_{\alpha})\\
&\ +|F'(\Y_{s}^{\etree})\Y_{s,t}^{\tddeuxa{$a$}{$b$}}\ \,-F'(\tilde\Y_{s}^{\etree})\tilde\Y_{s,t}^{\tddeuxa{$a$}{$b$}}\ \,|\\
&\hspace{2cm} (\text{by  ~(\ref{eq:infF}), (\ref{eq:iequ3}) and the differential mean value theorem of $F'$})\\
=&\ \|F''\|_{\infty}\Big|\sum_{a\in A}\Y_{s} ^{\bullet_a}\X_{s,t} ^{\bullet_a}+\sum_{a,b\in A}\Y_{s} ^{\bullet_a\bullet_b} \X_{s,t} ^{\bullet_a\bullet_b} +
\sum_{a,b\in A}\Y_{s} ^{\tddeuxa{$a$}{$b$}}\ \, \X_{s,t} ^{\tddeuxa{$a$}{$b$}}\ \,+R\Y^{\etree}_{s,t}\Big|(|\Y_{0}^{\tddeuxa{$a$}{$b$}}\ \, |+T^{\alpha}\|\Y^{\tddeuxa{$a$}{$b$}}\ \,\|_{\alpha})\\
&\ +\|F''\|_{\infty}\Big|\sum_{a\in A}\tilde\Y_{s} ^{\bullet_a}\X_{s,t} ^{\bullet_a}+\sum_{a,b\in A}\tilde\Y_{s} ^{\bullet_a\bullet_b} \X_{s,t} ^{\bullet_a\bullet_b} +
\sum_{a,b\in A}\tilde\Y_{s} ^{\tddeuxa{$a$}{$b$}}\ \, \X_{s,t} ^{\tddeuxa{$a$}{$b$}}\ \,+R\tilde\Y^{\etree}_{s,t}\Big|(|\tilde\Y_{0}^{\tddeuxa{$a$}{$b$}}\ \, |+T^{\alpha}\|\tilde\Y^{\tddeuxa{$a$}{$b$}}\ \,\|_{\alpha})\\
&\ +|F'(\Y_{s}^{\etree})\Y_{s,t}^{\tddeuxa{$a$}{$b$}}\ \,-F'(\tilde\Y_{s}^{\etree})\tilde\Y_{s,t}^{\tddeuxa{$a$}{$b$}}\ \,|.
\hspace{2cm} (\text{by  ~(\ref{eq:cbrp1})})
\end{align*}
All the summands are in the required form except the last one,
which can be handled as  ~\meqref{eq:step1a}. Then
\begin{align*}
&\ \|R\Z^{\tddeuxa{$a$}{$b$}}\ \,-R\tilde\Z^{\tddeuxa{$a$}{$b$}}\ \,\|_{\alpha} = \sup_{s\ne t\in [ 0,T]} \dfrac{|R\Z_{s,t}^{\tddeuxa{$a$}{$b$}}\ \,-R\tilde{\Z}_{s,t}^{\tddeuxa{$a$}{$b$}}\ \,|}{|t-s|^{\alpha}} \\
\le&\
\|F\|_{C_{b}^{3}} M\Big(T, \sum_{a\in A}|\Y_{0}^{\bullet_a}|, \sum_{a\in A}|\tilde\Y_{0}^{\bullet_a}|, \sum_{a,b\in A}|\Y_{0}^{\bullet_a\bullet_b}|, \sum_{a,b\in A}|\tilde\Y_{0}^{\bullet_a\bullet_b}|, \sum_{a,b\in A}| \Y_{0}^{\tddeuxa{$a$}{$b$}}\ \,|, \sum_{a,b\in A}|\tilde\Y_{0}^{\tddeuxa{$a$}{$b$}}\ \,|,\\
&\ \hspace{1.7cm} \fan{\Y} _{\X;\alpha }, \fan{\tilde\Y} _{\X;\alpha},
 \fan{\X} _{\alpha}, \fan{\tilde\X}_{\alpha}\Big)+\|F\|_{C_{b}^{3} }\|R\Y^{\tddeuxa{$a$}{$b$}}\ \,-R\tilde\Y^{\tddeuxa{$a$}{$b$}}\ \,\|_{\alpha}
\end{align*}
and so
\begin{equation}
\begin{aligned}
&\ \sum_{a,b\in A}\|R\Z^{\tddeuxa{$a$}{$b$}}\ \,-R\tilde\Z^{\tddeuxa{$a$}{$b$}}\ \,\|_{\alpha}\\
 \le&\
d^2\|F\|_{C_{b}^{3}} M\Big(T, \sum_{a\in A}|\Y_{0}^{\bullet_a}|, \sum_{a\in A}|\tilde\Y_{0}^{\bullet_a}|, \sum_{a,b\in A}|\Y_{0}^{\bullet_a\bullet_b}|, \sum_{a,b\in A}|\tilde\Y_{0}^{\bullet_a\bullet_b}|, \sum_{a,b\in A}| \Y_{0}^{\tddeuxa{$a$}{$b$}}\ \,|, \sum_{a,b\in A}|\tilde\Y_{0}^{\tddeuxa{$a$}{$b$}}\ \,|,  \\
&\ \hspace{2cm}\fan{\Y} _{\X;\alpha }, \fan{\tilde\Y} _{\X;\alpha}, \fan{\X} _{\alpha}, \fan{\tilde\X}_{\alpha}\Big)+d^2\|F\|_{C_{b}^{3}}\sum_{a,b\in A}\|R\Y^{\tddeuxa{$a$}{$b$}}\ \,-R\tilde\Y^{\tddeuxa{$a$}{$b$}}\ \,\|_{\alpha}.
\end{aligned}
\mlabel{eq:step3}
\end{equation}

\noindent{\bf Step 4:}	
Estimating $\| R\Z^{\bullet_a\bullet_b}-R\tilde\Z^{\bullet_a\bullet_b}\|_{\alpha}$ directly is difficult, so let us do some simplifications:
\begin{equation}
\begin{aligned}
&\ |R\Z_{s,t}^{\bullet_a\bullet_b}-\tilde R\Z_{s,t}^{\bullet_a\bullet_b}| = |\Z_{s,t}^{\bullet_a\bullet_b}-\tilde \Z_{s,t}^{\bullet_a\bullet_b}|\hspace{1cm}(\text{by  ~(\ref{eq:cbrp4})}) \\
=&\ \Big|\Big((F'(\Y_{t}^{\etree})\Y_{t}^{\bullet_a\bullet_b}+F''(\Y_{t}^{\etree})
\Y_{t}^{\bullet_a}\Y_{t}^{\bullet_b})-(F'(\Y_{s}^{\etree})
\Y_{s}^{\bullet_a\bullet_b}+F''(\Y_{s}^{\etree})\Y_{s}^{\bullet_a}\Y_{s}^{\bullet_b})\Big)\\
&\ -\Big((F'(\tilde\Y_{t}^{\etree})\tilde\Y_{t}^{\bullet_a\bullet_b}+F''(\tilde\Y_{t}^{\etree})\tilde\Y_{t}^{\bullet_a}\tilde\Y_{t}^{\bullet_b})-(F'(\tilde\Y_{s}^{\etree})\tilde\Y_{s}^{\bullet_a\bullet_b}+F''(\tilde\Y_{s}^{\etree})\tilde\Y_{s}^{\bullet_a}\tilde\Y_{s}^{\bullet_b})\Big)\Big| \\
&\ \hspace{7cm} (\text{by  ~(\mref{eq:zcbrp1}}))\\
\le&\  \Big|\Big((F'(\Y_{t}^{\etree})-F'(\Y_{s}^{\etree}))\Y_{t}^{\bullet_a\bullet_b}+F'(\Y_{s}^{\etree})\Y_{s,t}^{\bullet_a\bullet_b}\Big)-\Big((F'(\tilde\Y_{t}^{\etree})-F'( \tilde\Y_{s}^{\etree}))\tilde\Y_{t}^{\bullet_a\bullet_b}+F'(\tilde\Y_{s}^{\etree})\tilde\Y_{s,t}^{\bullet_a\bullet_b}\Big)\Big| \\
&\ +|F''(\Y_{t}^{\etree})\Y_{t}^{\bullet_a}\Y_{t}^{\bullet_b}
-F''(\Y_{s}^{\etree})\Y_{s}^{\bullet_a}\Y_{s}^{\bullet_b}|
+|F''(\tilde\Y_{t}^{\etree})\tilde\Y_{t}^{\bullet_a}\tilde\Y_{t}^{\bullet_b}
-F''(\tilde\Y_{s}^{\etree})\tilde\Y_{t}^{\bullet_a}\tilde\Y_{t}^{\bullet_b}|.
\end{aligned}
\mlabel{eq:I123}
\end{equation}
For convenience, let us inroduce
\begin{align*}
&I_1:=\Big|\Big((F'(\Y_{t}^{\etree})-F'(\Y_{s}^{\etree}))\Y_{t}^{\bullet_a\bullet_b}+F'(\Y_{s}^{\etree})\Y_{s,t}^{\bullet_a\bullet_b}\Big)-\Big((F'(\tilde\Y_{t}^{\etree})-F'( \tilde\Y_{s}^{\etree}))\tilde\Y_{t}^{\bullet_a\bullet_b}+F'(\tilde\Y_{s}^{\etree})\tilde\Y_{s,t}^{\bullet_a\bullet_b}\Big)\Big|,\\
&I_2:=|F''(\Y_{t}^{\etree})\Y_{t}^{\bullet_a}\Y_{t}^{\bullet_b}-F''(\Y_{s}^{\etree})\Y_{s}^{\bullet_a}\Y_{s}^{\bullet_b}|,\\
&I_3:=|F''(\tilde\Y_{t}^{\etree})\tilde\Y_{t}^{\bullet_a}\tilde\Y_{t}^{\bullet_b}-F''(\tilde\Y_{s}^{\etree})\tilde\Y_{t}^{\bullet_a}\tilde\Y_{t}^{\bullet_b}|.
\end{align*}
The estimation of $I_1$ can be written directly as
\begin{align*}
I_1
\le&\  |(F'(\Y_{t}^{\etree})-F'(\Y_{s}^{\etree}))\Y_{t}^{\bullet_a\bullet_b}|+|(F'(\tilde\Y_{t}^{\etree})-F'( \tilde\Y_{s}^{\etree}))\tilde\Y_{t}^{\bullet_a\bullet_b}|+|F'( \Y_{s}^{\etree})\Y_{s,t}^{\bullet_a\bullet_b}-F'(\tilde\Y_{s}^{\etree})\tilde\Y_{s,t}^{\bullet_a\bullet_b}| \\
&\hspace{7cm} (\text{by the triangle inequality})\\
\le&\  \|F''\|_{\infty}|\Y_{s,t}^{\etree}|(|\Y_{0}^{\bullet_a\bullet_b}|+T^{\alpha}\|\Y^{\bullet_a\bullet_b}\|_{\alpha})+\|F''\|_{\infty}|\tilde\Y_{s,t}^{\etree}|(|\tilde \Y_{0}^{\bullet_a\bullet_b}|+T^{\alpha}\|\tilde\Y^{\bullet_a\bullet_b}\|_{\alpha})\\
&\ +|F'(\Y_{s}^{\etree})\Y_{s,t}^{\bullet_a\bullet_b}-F'(\tilde\Y_{s}^{\etree})
\tilde\Y_{s,t}^{\bullet_a\bullet_b}|,\\
&\hspace{1.5cm} (\text{by the differential mean value theorem of $F'$ and  ~(\ref{eq:infF}) and (\ref{eq:iequ2})})
\end{align*}
whose last summand can be estimated similarly to  ~\meqref{eq:step1a}. Then	
\begin{align}
\frac{I_1}{|t-s|^{\alpha}}
\le&\ \|F\|_{C_{b}^{3}} M\Big(T, \sum_{a\in A}|\Y_{0}^{\bullet_a}|, \sum_{a\in A}|\tilde\Y_{0}^{\bullet_a}|, \sum_{a,b\in A}|\Y_{0}^{\bullet_a\bullet_b}|, \sum_{a,b\in A}|\tilde\Y_{0}^{\bullet_a\bullet_b}|, \sum_{a,b\in A}| \Y_{0}^{\tddeuxa{$a$}{$b$}}\ \,|, \sum_{a,b\in A}|\tilde\Y_{0}^{\tddeuxa{$a$}{$b$}}\ \,|, \mlabel{eq:eI1}\\
&\ \hspace{1.7cm} \fan{\Y} _{\X;\alpha }, \fan{\tilde\Y} _{\X;\alpha}, \fan{\X} _{\alpha}, \fan{\tilde\X}_{\alpha}\Big)+\|F\|_{C_{b}^{3} }\|R\Y^{\bullet_a\bullet_b}-R\tilde\Y^{\bullet_a\bullet_b}\|_{\alpha}.\nonumber
\end{align}
Turning to $I_2$,
\begin{align*}
I_2
\le&\ |(F''(\Y_{t}^{\etree})-F''(\Y_{s}^{\etree}))
\Y_{t}^{\bullet_a}\Y_{t}^{\bullet_b}|+ |F''(\Y_{s}^{\etree})
(\Y_{t}^{\bullet_a}\Y_{s,t}^{\bullet_b}+\Y_{s,t}^{\bullet_a}\Y_{s}^{\bullet_b})|\\
\le&\ \|F'''\|_{\infty}|\Y_{s,t}^{\etree}|\,|\Y_{t}^{\bullet_a}|\,|\Y_{t}^{\bullet_b}|+\|F''\|_{\infty}|\Y_{t}^{\bullet_a}|\,|\Y_{s,t}^{\bullet_b}|+\|F''\|_{\infty}|\Y_{s,t}^{\bullet_a}|\,|\Y_{s}^{\bullet_b}|\\
&\ \hspace{5cm} (\text{by the differential mean value theorem of $F''$and  ~(\ref{eq:infF})}) \\
\le&\  \|F'''\|_{\infty}|\Y_{s,t}^{\etree}|(|\Y_{0}^{\bullet_a}|+\|\Y^{\bullet_a}\|_{\alpha})(|\Y_{0}^{\bullet_b}|+\|\Y^{\bullet_b}\|_{\alpha})\\
&\ +\|F''\|_{\infty}\Bigg(|\Y_{0}^{\bullet_a}|+T^{\alpha}\sum_{b\in A}\Big((|\Y_{0}^{\bullet_a\bullet_b}|+T^{\alpha }\| \Y^{\bullet_a\bullet_b}\|_{\alpha })+(|\Y_{0}^{\tddeuxa{$a$}{$b$}}\ \,|+T^{\alpha }\| \Y^{\tddeuxa{$a$}{$b$}}\ \,\|_{\alpha })\Big)\| \X^{\bullet_b} \|_{\alpha }+T^{2\alpha }\|R\Y^{\bullet_a}\|_{2\alpha }\Bigg)|\Y_{s,t}^{\bullet_b}|\\
&\ +\|F''\|_{\infty}|\Y_{s,t}^{\bullet_a}|\Bigg(|\Y_{0}^{\bullet_b}|+T^{\alpha}\sum_{c\in A}\Big((|\Y_{0}^{\bullet_b\bullet_c}|+T^{\alpha }\| \Y^{\bullet_b\bullet_c}\|_{\alpha })+(|\Y_{0}^{\tddeuxa{$b$}{$c$}}\ \,|+T^{\alpha }\| \Y^{\tddeuxa{$b$}{$c$}}\ \,\|_{\alpha })\Big)\| \X^{\bullet_c} \|_{\alpha }+T^{2\alpha }\|R\Y^{\bullet_b}\|_{2\alpha }\Bigg)\\
&\hspace{8cm} (\text{by  ~(\ref{eq:iequ1}) and (\ref{eq:iequ4})})
\end{align*}
and so
\begin{align}
\frac{I_2}{|t-s|^{\alpha}}
\le&\ \|F'''\|_{\infty}(|\Y_{0}^{\bullet_a}|+\|\Y^{\bullet_a}\|_{\alpha})(|\Y_{0}^{\bullet_b}|+\|\Y^{\bullet_b}\|_{\alpha})T^{2\alpha}\|\Y^{\etree}\|_{3\alpha}\hspace{3cm}\mlabel{eq:eI2}\\
&\ +\|F''\|_{\infty}\Bigg(|\Y_{0}^{\bullet_a}|+T^{\alpha}\sum_{b\in A}\Big((|\Y_{0}^{\bullet_a\bullet_b}|+T^{\alpha }\| \Y^{\bullet_a\bullet_b}\|_{\alpha })+(|\Y_{0}^{\tddeuxa{$a$}{$b$}}\ \,|+T^{\alpha }\| \Y^{\tddeuxa{$a$}{$b$}}\ \,\|_{\alpha })\Big)\| \X^{\bullet_b} \|_{\alpha }\nonumber\\
&\ +T^{2\alpha }\|R\Y^{\bullet_a}\|_{2\alpha }\Bigg)T^{\alpha}\|\Y^{\bullet_b}\|_{2\alpha}\nonumber\\
&\ +\|F''\|_{\infty}\Bigg(|\Y_{0}^{\bullet_b}|+T^{\alpha}\sum_{c\in A}\Big((|\Y_{0}^{\bullet_b\bullet_c}|+T^{\alpha }\| \Y^{\bullet_b\bullet_c}\|_{\alpha })+(|\Y_{0}^{\tddeuxa{$b$}{$c$}}\ \,|+T^{\alpha }\| \Y^{\tddeuxa{$b$}{$c$}}\ \,\|_{\alpha })\Big)\| \X^{\bullet_c} \|_{\alpha }\nonumber\\
&\ +T^{2\alpha }\|R\Y^{\bullet_b}\|_{2\alpha }\Bigg)T^{\alpha}\|\Y^{\bullet_a}\|_{2\alpha} \nonumber\\
\le&\ \|F\|_{C_{b}^{3}} M\Big(T, \sum_{a\in A}|\Y_{0}^{\bullet_a}|, \sum_{a\in A}|\tilde\Y_{0}^{\bullet_a}|, \sum_{a,b\in A}|\Y_{0}^{\bullet_a\bullet_b}|, \sum_{a,b\in A}|\tilde\Y_{0}^{\bullet_a\bullet_b}|, \sum_{a,b\in A}| \Y_{0}^{\tddeuxa{$a$}{$b$}}\ \,|, \sum_{a,b\in A}|\tilde\Y_{0}^{\tddeuxa{$a$}{$b$}}\ \,|, \nonumber\\
&\ \hspace{1.6cm} \fan{\Y} _{\X;\alpha }, \fan{\tilde\Y} _{\X;\alpha}, \fan{\X} _{\alpha}, \fan{\tilde\X}_{\alpha}\Big).\nonumber
\end{align}
Similarly,
\begin{align}
\frac{I_3}{|t-s|^{\alpha}}
\le&\ \|F'''\|_{\infty}(|\tilde\Y_{0}^{\bullet_a}|+\|\tilde\Y^{\bullet_a}\|_{\alpha})(|\tilde\Y_{0}^{\bullet_b}|+\|\tilde\Y^{\bullet_b}\|_{\alpha})T^{2\alpha}\|\tilde\Y^{\etree}\|_{3\alpha} \hspace{3cm}\mlabel{eq:eI3}\\
&\ +\|F''\|_{\infty}\Bigg(|\tilde\Y_{0}^{\bullet_a}|+T^{\alpha}\sum_{b\in A}\Big((|\tilde\Y_{0}^{\bullet_a\bullet_b}|+T^{\alpha }\|\tilde\Y^{\bullet_a\bullet_b}\|_{\alpha })+(|\tilde\Y_{0}^{\tddeuxa{$a$}{$b$}}\ \,|+T^{\alpha }\|\tilde\Y^{\tddeuxa{$a$}{$b$}}\ \,\|_{\alpha })\Big)\|\tilde\X^{\bullet_b} \|_{\alpha }\nonumber\\
&\ +T^{2\alpha }\|R\tilde\Y^{\bullet_a}\|_{2\alpha }\Bigg)T^{\alpha}\|\tilde\Y^{\bullet_b}\|_{2\alpha}\nonumber\\
&\ +\|F''\|_{\infty}\Bigg(|\tilde\Y_{0}^{\bullet_b}|+T^{\alpha}\sum_{c\in A}\Big((|\tilde\Y_{0}^{\bullet_b\bullet_c}|+T^{\alpha }\|\tilde\Y^{\bullet_b\bullet_c}\|_{\alpha })+(|\tilde\Y_{0}^{\tddeuxa{$b$}{$c$}}\ \,|+T^{\alpha }\|\tilde\Y^{\tddeuxa{$b$}{$c$}}\ \,\|_{\alpha })\Big)\|\tilde\X^{\bullet_c} \|_{\alpha }\nonumber\\
&\ +T^{2\alpha }\|R\tilde\Y^{\bullet_b}\|_{2\alpha }\Bigg)T^{\alpha}\|\tilde\Y^{\bullet_a}\|_{2\alpha}\nonumber\\
\le&\ \|F\|_{C_{b}^{3}} M\Big(T, \sum_{a\in A}|\Y_{0}^{\bullet_a}|, \sum_{a\in A}|\tilde\Y_{0}^{\bullet_a}|, \sum_{a,b\in A}|\Y_{0}^{\bullet_a\bullet_b}|, \sum_{a,b\in A}|\tilde\Y_{0}^{\bullet_a\bullet_b}|, \sum_{a,b\in A}| \Y_{0}^{\tddeuxa{$a$}{$b$}}\ \,|, \sum_{a,b\in A}|\tilde\Y_{0}^{\tddeuxa{$a$}{$b$}}\ \,|, \nonumber\\
&\ \hspace{1.7cm} \fan{\Y} _{\X;\alpha }, \fan{\tilde\Y} _{\X;\alpha}, \fan{\X} _{\alpha}, \fan{\tilde\X}_{\alpha}\Big).\nonumber
\end{align}
Hence
\begin{align*}
&\ \|R\Z^{\bullet_a\bullet_b}-R\tilde{\Z}^{\bullet_a\bullet_b}\|_{\alpha} = \sup_{s\ne t\in [ 0,T]} \dfrac{|R\Z_{s,t}^{\bullet_a\bullet_b}-R\tilde{\Z}_{s,t}^{\bullet_a\bullet_b}|}{|t-s|^{\alpha}} \le  \sup_{s\ne t\in [0,T]} \frac{I_1+I_2+I_3}{|t-s|^{\alpha}}  \quad (\text{by  ~\eqref{eq:I123}})\\
\le&\ \|F\|_{C_{b}^{3}} M\Big(T, \sum_{a\in A}|\Y_{0}^{\bullet_a}|, \sum_{a\in A}|\tilde\Y_{0}^{\bullet_a}|, \sum_{a,b\in A}|\Y_{0}^{\bullet_a\bullet_b}|, \sum_{a,b\in A}|\tilde\Y_{0}^{\bullet_a\bullet_b}|, \sum_{a,b\in A}| \Y_{0}^{\tddeuxa{$a$}{$b$}}\ \,|, \sum_{a,b\in A}|\tilde\Y_{0}^{\tddeuxa{$a$}{$b$}}\ \,|, \fan{\Y} _{\X;\alpha }, \fan{\tilde\Y} _{\X;\alpha},\\
&\ \hspace{1.7cm} \fan{\X} _{\alpha}, \fan{\tilde\X}_{\alpha}\Big)+\|F\|_{C_{b}^{3} }\|R\Y^{\bullet_a\bullet_b}-R\tilde\Y^{\bullet_a\bullet_b}\|_{\alpha}
\hspace{0.5cm} (\text{by  ~\eqref{eq:eI1}, \eqref{eq:eI2}, \eqref{eq:eI3} and adjusting $M$})
%
%=&\ \|F\|_{C_{b}^{3}} M\Big(T, \sum_{a\in A}|\Y_{0}^{\bullet_a}|, \sum_{a\in A}|\tilde\Y_{0}^{\bullet_a}|, \sum_{a,b\in A}|\Y_{0}^{\bullet_a\bullet_b}|, \sum_{a,b\in A}|\tilde\Y_{0}^{\bullet_a\bullet_b}|, \sum_{a,b\in A}| \Y_{0}^{\tddeuxa{$a$}{$b$}}\ \,|, \sum_{a,b\in A}|\tilde\Y_{0}^{\tddeuxa{$a$}{$b$}}\ \,|, \fan{\Y} _{\X;\alpha }, \fan{\tilde\Y} _{\X;\alpha},\\
%%
%&\ \hspace{1.7cm} \fan{\X} _{\alpha}, \fan{\tilde\X}_{\alpha}\Big)+\|F\|_{C_{b}^{3} }\|R\Y^{\bullet_a\bullet_b}-R\tilde\Y^{\bullet_a\bullet_b}\|_{\alpha}.
\end{align*}
and so
\begin{align}
\sum_{a,b\in A}\|R\Z^{\bullet_a\bullet_b}-R\tilde{\Z}^{\bullet_a\bullet_b}\|_{\alpha}
\le&\ d^2\|F\|_{C_{b}^{3}} M\Big(T, \sum_{a\in A}|\Y_{0}^{\bullet_a}|, \sum_{a\in A}|\tilde\Y_{0}^{\bullet_a}|, \sum_{a,b\in A}|\Y_{0}^{\bullet_a\bullet_b}|, \sum_{a,b\in A}|\tilde\Y_{0}^{\bullet_a\bullet_b}|, \sum_{a,b\in A}| \Y_{0}^{\tddeuxa{$a$}{$b$}}\ \,|, \sum_{a,b\in A}|\tilde\Y_{0}^{\tddeuxa{$a$}{$b$}}\ \,|, \mlabel{eq:step4}\\
&\ \hspace{2cm} \fan{\Y} _{\X;\alpha },\fan{\tilde\Y} _{\X;\alpha}, \fan{\X} _{\alpha}, \fan{\tilde\X}_{\alpha}\Big)+d^2\|F\|_{C_{b}^{3}}\sum_{a,b\in A}\|R\Y^{\bullet_a\bullet_b}-R\tilde\Y^{\bullet_a\bullet_b}\|_{\alpha}.\nonumber
\end{align}
	
Finally, gather the previous bounds in Steps 1-4 and conclude that
\begin{align*}
&\ \fan{\Z,\tilde{\Z}}_{\X,\tilde{\X};\alpha}\\
%1
=&\ \|R\Z^{\etree}-R\tilde\Z^{\etree}\| _{3\alpha} +\sum_{a\in A}\|R\Z^{\bullet_a}-R\tilde\Z^{\bullet_a}\|_{2\alpha}+\sum_{a,b\in A}\|R\Z^{\bullet_a \bullet_b}-R\tilde\Z^{\bullet_a\bullet_a}\|_{\alpha}+\sum_{a,b\in A}\|R\Z^{\tddeuxa{$a$}{$b$}}\ \,-R\tilde\Z^{\tddeuxa{$a$}{$b$}}\ \,\|_{\alpha}\\
&\ \hspace{7cm} (\text{by  ~\meqref{eq:ztz}})\\
%2
\le&\  4d^2\|F\|_{C_{b}^{3}} M\Big(T, \sum_{a\in A}|\Y_{0}^{\bullet_a}|, \sum_{a\in A}|\tilde\Y_{0}^{\bullet_a}|, \sum_{a,b\in A}|\Y_{0}^{\bullet_a\bullet_b}|, \sum_{a,b\in A}|\tilde\Y_{0}^{\bullet_a\bullet_b}|, \sum_{a,b\in A}| \Y_{0}^{\tddeuxa{$a$}{$b$}}\ \,|, \sum_{a,b\in A}|\tilde\Y_{0}^{\tddeuxa{$a$}{$b$}}\ \,|, \fan{\Y} _{\X;\alpha }, \\
&\ \hspace{2.1cm}\fan{\tilde\Y} _{\X;\alpha}, \fan{\X} _{\alpha}, \fan{\tilde\X}_{\alpha}\Big)\\
&\ +d^2\|F\|_{C_{b}^{3}}\Big(\|R\Y^{\etree }-R\tilde\Y^{\etree }\|_{3\alpha}+\sum_{a\in A}\|R\Y^{\bullet_a}-R\tilde\Y^{\bullet_a}\|_{2\alpha}+\sum_{a,b\in A}\|R\Y^{\tddeuxa{$a$}{$b$}}\ \,-R\tilde\Y^{\tddeuxa{$a$}{$b$}}\ \,\|_{\alpha}+\sum_{a,b\in A}\|R\Y^{\bullet_a\bullet_b}-R\tilde\Y^{\bullet_a\bullet_b}\|_{\alpha}\Big)\\
&\ \hspace{5cm} (\text{by summing up  ~(\ref{eq:step1}), (\ref{eq:step2}), (\ref{eq:step3}) and (\ref{eq:step4})})\\
%3
=&\  4d^2\|F\|_{C_{b}^{3}} M\Big(T, \sum_{a\in A}|\Y_{0}^{\bullet_a}|, \sum_{a\in A}|\tilde\Y_{0}^{\bullet_a}|, \sum_{a,b\in A}|\Y_{0}^{\bullet_a\bullet_b}|, \sum_{a,b\in A}|\tilde\Y_{0}^{\bullet_a\bullet_b}|, \sum_{a,b\in A}| \Y_{0}^{\tddeuxa{$a$}{$b$}}\ \,|, \sum_{a,b\in A}|\tilde\Y_{0}^{\tddeuxa{$a$}{$b$}}\ \,|, \fan{\Y} _{\X;\alpha }, \\
&\ \hspace{2.1cm}\fan{\tilde\Y} _{\X;\alpha}, \fan{\X} _{\alpha}, \fan{\tilde\X}_{\alpha}\Big)+d^2\|F\|_{C_{b}^{3}}\fan {\Y,\tilde{\Y}}_{\X,\tilde{\X};\alpha}  \hspace{1cm} (\text{by  ~\meqref{eq:dist}})\\
\le&\  4d^2\|F\|_{C_{b}^{3}}\Bigg( M\Big(T, \sum_{a\in A}|\Y_{0}^{\bullet_a}|, \sum_{a\in A}|\tilde\Y_{0}^{\bullet_a}|, \sum_{a,b\in A}|\Y_{0}^{\bullet_a\bullet_b}|, \sum_{a,b\in A}|\tilde\Y_{0}^{\bullet_a\bullet_b}|, \sum_{a,b\in A}| \Y_{0}^{\tddeuxa{$a$}{$b$}}\ \,|, \sum_{a,b\in A}|\tilde\Y_{0}^{\tddeuxa{$a$}{$b$}}\ \,|, \fan{\Y} _{\X;\alpha }, \\
&\ \hspace{2.3cm}\fan{\tilde\Y} _{\X;\alpha}, \fan{\X} _{\alpha}, \fan{\tilde\X}_{\alpha}\Big)+\fan {\Y,\tilde{\Y}}_{\X,\tilde{\X};\alpha}\Bigg)  \hspace{1cm} \\
%4
\le&\ 4d^2\|F\|_{C_{b}^{3}}\Bigg(M\Big(T, \sum_{a\in A}|\Y_{0}^{\bullet_a}|, \sum_{a\in A}|\tilde\Y_{0}^{\bullet_a}|, \sum_{a,b\in A}|\Y_{0}^{\bullet_a\bullet_b}|, \sum_{a,b\in A}|\tilde\Y_{0}^{\bullet_a\bullet_b}|, \sum_{a,b\in A}| \Y_{0}^{\tddeuxa{$a$}{$b$}}\ \,|, \sum_{a,b\in A}|\tilde\Y_{0}^{\tddeuxa{$a$}{$b$}}\ \,|, \fan{\Y} _{\X;\alpha }, \\
&\ \hspace{2.2cm}\fan{\tilde\Y} _{\X;\alpha}, \fan{\X} _{\alpha}, \fan{\tilde\X}_{\alpha}\Big)+1\Bigg)\times \Big(\fan{\Y,\tilde{\Y}}_{\X,\tilde{\X};\alpha} +\fan{\X,\tilde{\X}}_{\alpha}+|\Y_0^\etree -\tilde{\Y}^\etree_0|\Big)\\
&\ \hspace{3cm}(\text{each above item appears in below and adjust the below $M$})\\
=&\ 4d^2\|F\|_{C_{b}^{3}}M\Big(T, \sum_{a\in A}|\Y_{0}^{\bullet_a}|, \sum_{a\in A}|\tilde\Y_{0}^{\bullet_a}|, \sum_{a,b\in A}|\Y_{0}^{\bullet_a\bullet_b}|, \sum_{a,b\in A}|\tilde\Y_{0}^{\bullet_a\bullet_b}|, \sum_{a,b\in A}| \Y_{0}^{\tddeuxa{$a$}{$b$}}\ \,|, \sum_{a,b\in A}|\tilde\Y_{0}^{\tddeuxa{$a$}{$b$}}\ \,|, \fan{\Y} _{\X;\alpha }, \\
&\ \hspace{2.1cm}\fan{\tilde\Y} _{\X;\alpha}, \fan{\X} _{\alpha}, \fan{\tilde\X}_{\alpha}\Big)\times \Big(\fan{\Y,\tilde{\Y}}_{\X,\tilde{\X};\alpha} +\fan{\X,\tilde{\X}}_{\alpha}+|\Y_0^\etree -\tilde{\Y}^\etree_0|\Big).
\hspace{0.5cm}(\text{by adjusting $M$})
\end{align*}
This completes the proof.
\end{proof}

\section{Planarly branched universal limit theorem}
In this section, we first review the concept of rough integral of controlled planarly branched rough paths against planarly branched rough paths as a tool to solve RODEs. Building on this, we then prove the universal limit theorem for planarly branched rough paths.

\subsection{Integration of planarly branched rough paths}\label{sec3.1}
In the remainder of this section, we always assume
\begin{equation}
X=(X^1,\ldots,X^d):[0,T]\rightarrow \mathbb{R}^d,\quad Y:[0,T]\rightarrow \mathbb{R}^n,\quad F=(F^1,\ldots,F^d):\mathbb{R}^n \rightarrow  (\mathbb{R}^n)^d,
\mlabel{eq:XYF}
\end{equation}
with each $F^i:\mathbb{R}^n \rightarrow\mathbb{R}^n $ smooth. We want to define the planar rough integral by the Sewing lemma:
\begin{lemma}~\cite{FP06, Gub04}
For any continuous function $A: [0, T]\times[0, T] \to \mathbb{R}$ and some constant C, if
\begin{equation}
|A_{s,t}-A_{s,u}-A_{u,t}|\le C|t-s|^p,
\mlabel{eq:seweq}
\end{equation}
for some $p>1$, then there exists a unique function $I: [0, T] \to \mathbb{R}$ such that $I_0=0$ and
$
|I_t-I_s-A_{s,t}|= O(| t-s | ^{p} ),
$
uniformly over $0\le s\le t\le T$. Moreover I is the limit of Riemann-type sums
$
I_t=\lim\limits_{|\pi | \to 0} \sum\limits_{[t_i, t_{i+1}]\in \pi}A_{t_i, t_{i+1}},
$
where $\pi$ is an arbitrary partition of $[0, T]$.
\mlabel{lem:sew}
\end{lemma}

Let $\Y$ be an $\X$-controlled planarly branched rough path above $Y$. Then we can use $\Y$ to define the integral $\int_s^t YdX^a$ for any $1 \le a \le d$.
\begin{theorem}
Let $\alpha \in  ( \frac{1}{4}, \frac{1}{3}]$, $\X\in \brpt$ and $\Y\in \cbrpxd$ above Y. Define	
\begin{equation}
\tilde{Y}_{s,t} :=\sum_{\tau\in \mathcal{F}^{\leq 2}}\Y_{s}^\tau\X_{s,t}^{[\tau]_a}.
\mlabel{eq:inte1}
\end{equation}
Then $\tilde{Y}$ satisfies (\ref{eq:seweq}). By Lemma~\ref{lem:sew}, there is a unique function $I:[0, T]\rightarrow \RR^d$  such that
\begin{equation}
\Big|\int_s^t Y_rdX_{r}^{a}-\sum_{\tau\in \mathcal{F}^{\leq 2}}\Y_{s}^\tau\X_{s,t}^{[\tau]_a}\Big|= O(| t-s | ^{4\alpha} )
\mlabel{eq:ieq1}
\end{equation}
and
\begin{equation*}
I_t=\int_0^t Y_rdX_{r}^{a}=\lim_{|\pi | \to 0} \sum_{[t_i, t_{i+1}]\in \pi}\sum_{\tau\in \mathcal{F}^{\leq 2}}\Y_{t_i}^\tau\X_{t_i,t_{i+1}}^{[\tau]_a},
\end{equation*}
where $\pi$ is an arbitrary partition of $[0, T]$.
\mlabel{thm:inte1}
\end{theorem}

\begin{proof}
It suffices to prove that $\tilde{Y}$ satisfies (\ref{eq:seweq}). We compute
\begin{equation}
\begin{aligned}
|\tilde{Y}_{s,t}-\tilde{Y}_{s,u}-\tilde{Y}_{u,t}|
=&\ \Big|\sum_{\tau\in \mathcal{F}^{\leq 2}}\Y_{s}^\tau\X_{s,t}^{[\tau]_a}-\sum_{\tau\in \mathcal{F}^{\leq 2}}\Y_{s}^\tau\X_{s,u}^{[\tau]_a}-\sum_{\tau\in \mathcal{F}^{\leq 2}}\Y_{u}^\tau\X_{u,t}^{[\tau]_a}\Big|\hspace{1cm}(\text{by (\ref{eq:inte1})})\\
=&\ \Big|\Big(\Y^{\etree}_s\X^{\bullet_a}_{s,t}+\sum_{b\in A}\Y^{\bullet_b}_s\X^{\tddeuxa{$a$}{$b$}}_{s,t}\ \,+\sum_{b,c\in A}\Y^{\bullet_b\bullet_c}_s\X^{\tdtroisuna{$a$}{$c$}{$b$}}_{s,t}\ \,+\sum_{b,c\in A}\Y^{\tddeuxa{$b$}{$c$}}_s\ \,\X^{\tdddeuxa{$a$}{$b$}{$c$}}_{s,t}\, \Big) \\
&\ -\Big(\Y^{\etree}_s\X^{\bullet_a}_{s,u}+\sum_{b\in A}\Y^{\bullet_b}_s\X^{\tddeuxa{$a$}{$b$}}_{s,u}\ \,+\sum_{b,c\in A}\Y^{\bullet_b\bullet_c}_s\X^{\tdtroisuna{$a$}{$c$}{$b$}}_{s,u}\ \,+\sum_{b,c\in A}\Y^{\tddeuxa{$b$}{$c$}}_s\ \,\X^{\tdddeuxa{$a$}{$b$}{$c$}}_{s,u}\, \Big)\\
&\ -\Big(\Y^{\etree}_u\X^{\bullet_a}_{u,t}+\sum_{b\in A}\Y^{\bullet_b}_u\X^{\tddeuxa{$a$}{$b$}}_{u,t}\ \,+\sum_{b,c\in A}\Y^{\bullet_b\bullet_c}_u\X^{\tdtroisuna{$a$}{$c$}{$b$}}_{u,t}\ \,+\sum_{b,c\in A}\Y^{\tddeuxa{$b$}{$c$}}_u\ \,\X^{\tdddeuxa{$a$}{$b$}{$c$}}_{u,t}\, \Big)\Big|\\
=&\ \Big|\Big(\Y^{\etree}_s\X^{\bullet_a}_{s,t}+\sum_{b\in A}\Y^{\bullet_b}_s\X^{\tddeuxa{$a$}{$b$}}_{s,t}\ \,+\sum_{b,c\in A}\Y^{\bullet_b\bullet_c}_s\X^{\tdtroisuna{$a$}{$c$}{$b$}}_{s,t}\ \,+\sum_{b,c\in A}\Y^{\tddeuxa{$b$}{$c$}}_s\ \,\X^{\tdddeuxa{$a$}{$b$}{$c$}}_{s,t}\, \Big) \\
&\ -\Big(\Y^{\etree}_s\X^{\bullet_a}_{s,u}+\sum_{b\in A}\Y^{\bullet_b}_s\X^{\tddeuxa{$a$}{$b$}}_{s,u}\ \,+\sum_{b,c\in A}\Y^{\bullet_b\bullet_c}_s\X^{\tdtroisuna{$a$}{$c$}{$b$}}_{s,u}\ \,+\sum_{b,c\in A}\Y^{\tddeuxa{$b$}{$c$}}_s\ \,\X^{\tdddeuxa{$a$}{$b$}{$c$}}_{s,u}\, \Big)\\
&\ -\Big(\Y^{\etree}_u\X^{\bullet_a}_{u,t}+\sum_{b\in A}\Y^{\bullet_b}_u\X^{\tddeuxa{$a$}{$b$}}_{u,t}\ \,+\sum_{b,c\in A}\Y^{\bullet_b\bullet_c}_s\X^{\tdtroisuna{$a$}{$c$}{$b$}}_{u,t}\ \,+\sum_{b,c\in A}\Y^{\tddeuxa{$b$}{$c$}}_s\ \,\X^{\tdddeuxa{$a$}{$b$}{$c$}}_{u,t}\, \\
&\hspace{0.5cm}+\sum_{b,c\in A}\Y^{\bullet_b\bullet_c}_{s,u}\X^{\tdtroisuna{$a$}{$c$}{$b$}}_{u,t}\ \,+\sum_{b,c\in A}\Y^{\tddeuxa{$b$}{$c$}}_{s,u}\ \,\X^{\tdddeuxa{$a$}{$b$}{$c$}}_{u,t}\, \Big)\Big|.
\end{aligned}
\mlabel{eq:chen}
\end{equation}
It follows from Definition~\ref{def:pbrp}~(a) that
\begin{align*}
\langle \X_{s,t}, \bullet_a\rangle
=&\ \langle \X_{s,u}\star\X_{u,t}, \bullet_a\rangle  = \langle \X_{s,u}\otimes\X_{u,t}, \Delta_{\text {MKW}}(\bullet_a)\rangle \\
=&\ \langle \X_{s,u}\otimes\X_{u,t}, \bullet_a \otimes \etree+\etree \otimes \bullet_a\rangle
=  \langle \X_{s,u}, \bullet_a \rangle+\langle \X_{u,t}, \bullet_a \rangle,
\end{align*}
whence
\begin{equation}
\Y^{\etree}_s\X^{\bullet_a}_{s,t}-\Y^{\etree}_s\X^{\bullet_a}_{s,u}=\Y^{\etree}_s\X^{\bullet_a}_{u,t}.
\mlabel{eq:chen1}
\end{equation}
Similarly,
\begin{align}
\sum_{b\in A}\Y^{\bullet_b}_s\X^{\tddeuxa{$a$}{$b$}}_{s,t}\ \,-\sum_{b\in A}\Y^{\bullet_b}_s\X^{\tddeuxa{$a$}{$b$}}_{s,u}\ \,=&\ \sum_{b\in A}\Y^{\bullet_b}_s\X^{\tddeuxa{$a$}{$b$}}_{u,t}\ \,+\sum_{b\in A}\Y^{\bullet_b}_s\X^{\bullet_b}_{s,u}\X^{\bullet_a}_{u,t},\mlabel{eq:chen2}\\
\sum_{b,c\in A}\Y^{\bullet_b\bullet_c}_s\X^{\tdtroisuna{$a$}{$c$}{$b$}}_{s,t}\ \,-\sum_{b,c\in A}\Y^{\bullet_b\bullet_c}_s\X^{\tdtroisuna{$a$}{$c$}{$b$}}_{s,u}\ \,=&\ \sum_{b,c\in A}\Y^{\bullet_b\bullet_c}_s\Big(\X^{\tdtroisuna{$a$}{$c$}{$b$}}_{u,t}\ \,+\X^{\bullet_b}_{s,u}\X^{\tddeuxa{$a$}{$c$}}_{u,t}\ +\X^{\bullet_b\bullet_c}_{s,u}\X^{\bullet_a}_{u,t}\Big),
\mlabel{eq:chen3}\\
\sum_{b,c\in A}\Y^{\tddeuxa{$b$}{$c$}}_s\ \,\X^{\tdddeuxa{$a$}{$b$}{$c$}}_{s,t}\ \,-\sum_{b,c\in A}\Y^{\tddeuxa{$b$}{$c$}}_s\ \,\X^{\tdddeuxa{$a$}{$b$}{$c$}}_{s,u}\ \,=&\ \sum_{b,c\in A}\Y^{\tddeuxa{$b$}{$c$}}_s\ \,\Big(\X^{\tdddeuxa{$a$}{$b$}{$c$}}_{u,t}\ \,+\X^{\bullet_c}_{s,u}\X^{\tddeuxa{$a$}{$b$}}_{u,t}\ +\X^{\tddeuxa{$b$}{$c$}}_{s,u}\X^{\bullet_a}_{u,t}\Big).
\mlabel{eq:chen4}
\end{align}
Plugging (\ref{eq:chen1}), (\ref{eq:chen2}), (\ref{eq:chen3}) and (\ref{eq:chen4}) into the last equation of (\ref{eq:chen}), we obtain
\begin{align*}
|\tilde{Y}_{s,t}-\tilde{Y}_{s,u}-\tilde{Y}_{u,t}|
=&\ \Big|\Y^{\etree}_s\X^{\bullet_a}_{u,t}+\Big(\sum_{b\in A}\Y^{\bullet_b}_s\X^{\tddeuxa{$a$}{$b$}}_{u,t}\ \,+\sum_{b\in A}\Y^{\bullet_b}_s\X^{\bullet_b}_{s,u}\X^{\bullet_a}_{u,t} \Big)+\Big(\sum_{b,c\in A}\Y^{\bullet_b\bullet_c}_s\X^{\tdtroisuna{$a$}{$c$}{$b$}}_{u,t}\ \,+\sum_{b,c\in A}\Y^{\bullet_b\bullet_c}_s\X^{\bullet_b}_{s,u}\X^{\tddeuxa{$a$}{$c$}}_{u,t}\ \\
&\ +\sum_{b,c\in A}\Y^{\bullet_b\bullet_c}_s\X^{\bullet_b\bullet_c}_{s,u}\X^{\bullet_a}_{u,t} \Big)+\Big(\sum_{b,c\in A}\Y^{\tddeuxa{$b$}{$c$}}_s\ \,\X^{\tdddeuxa{$a$}{$b$}{$c$}}_{u,t}\ \,+\sum_{b,c\in A}\Y^{\tddeuxa{$b$}{$c$}}_s\ \,\X^{\bullet_c}_{s,u}\X^{\tddeuxa{$a$}{$b$}}_{u,t}\ +\sum_{b,c\in A}\Y^{\tddeuxa{$b$}{$c$}}_s\ \,\X^{\tddeuxa{$b$}{$c$}}_{s,u}\X^{\bullet_a}_{u,t}\Big)\\
&\ -\Big(\Y^{\etree}_u\X^{\bullet_a}_{u,t}+\sum_{b\in A}\Y^{\bullet_b}_u\X^{\tddeuxa{$a$}{$b$}}_{u,t}\ \,+\sum_{b,c\in A}\Y^{\bullet_b\bullet_c}_s\X^{\tdtroisuna{$a$}{$c$}{$b$}}_{u,t}\ \,+\sum_{b,c\in A}\Y^{\tddeuxa{$b$}{$c$}}_s\ \,\X^{\tdddeuxa{$a$}{$b$}{$c$}}_{u,t}\, \\
&\hspace{0.5cm}+\sum_{b,c\in A}\Y^{\bullet_b\bullet_c}_{s,u}\X^{\tdtroisuna{$a$}{$c$}{$b$}}_{u,t}\ \,+\sum_{b,c\in A}\Y^{\tddeuxa{$b$}{$c$}}_{s,u}\ \,\X^{\tdddeuxa{$a$}{$b$}{$c$}}_{u,t}\, \Big)\Big|\\
=&\ \Big|-\Big(\Y^{\etree}_{s,u}-\sum_{b\in A}\Y^{\bullet_b}_s\X^{\bullet_b}_{s,u}-\sum_{b,c\in A}\Y^{\bullet_b\bullet_c}_s\X^{\bullet_b\bullet_c}_{s,u}-\sum_{b,c\in A}\Y^{\tddeuxa{$b$}{$c$}}_s\ \,\X^{\tddeuxa{$b$}{$c$}}_{s,u} \Big)\X^{\bullet_a}_{u,t}\\
&\ -\sum_{b\in A}\Big(\Y^{\bullet_b}_{s,u}-\sum_{c\in A}\Y^{\bullet_b\bullet_c}_s\X^{\bullet_c}_{s,u}-\sum_{c\in A}\Y^{\tddeuxa{$b$}{$c$}}_s\ \,\X^{\bullet_c}_{s,u}\Big)\X^{\tddeuxa{$a$}{$b$}}_{u,t}\ \,-\sum_{b,c\in A}\Y^{\bullet_b\bullet_c}_{s,u}\X^{\tdtroisuna{$a$}{$c$}{$b$}}_{u,t}\ \, \\
&\ -\sum_{b,c\in A}\Y^{\tddeuxa{$b$}{$c$}}_{s,u}\ \,\X^{\tdddeuxa{$a$}{$b$}{$c$}}_{u,t}\,\Big|\\
=&\ \Big|-R\Y^{\etree}_{s,u}\X^{\bullet_a}_{u,t}-\sum_{b\in A}R\Y^{\bullet_b}_{s,u}\X^{\tddeuxa{$a$}{$b$}}_{u,t}\ \,-\sum_{b,c\in A}\Y^{\bullet_b\bullet_c}_{s,u}\X^{\tdtroisuna{$a$}{$c$}{$b$}}_{u,t}\ \,-\sum_{b,c\in A}\Y^{\tddeuxa{$b$}{$c$}}_{s,u}\ \,\X^{\tdddeuxa{$a$}{$b$}{$c$}}_{u,t}\,\Big|\\
&\ \hspace{8cm}(\text{by Definition~\mref{defn:cbrp1}})\\
\le&\ \Big|R\Y^{\etree}_{s,u}\Big|\,\Big|\X^{\bullet_a}_{u,t}\Big|+\sum_{b\in A}\Big|R\Y^{\bullet_b}_{s,u} \Big|\,\Big|\X^{\tddeuxa{$a$}{$b$}}_{u,t}\ \, \Big|+\sum_{b,c\in A}\Big|\Y^{\bullet_b\bullet_c}_{s,u} \Big|\,\Big|\X^{\tdtroisuna{$a$}{$c$}{$b$}}_{u,t}\ \, \Big|+\sum_{b,c\in A}\Big|\Y^{\tddeuxa{$b$}{$c$}}_{s,u}\ \, \Big|\,\Big|\X^{\tdddeuxa{$a$}{$b$}{$c$}}_{u,t}\, \Big|\\
&\ \hspace{6cm}(\text{by the triangle inequality})\\
\le&\ C_{1,1}|u-s|^{3\alpha}|t-u|^{\alpha}+\sum_{i=1}^{d}C_{2,i}|u-s|^{2\alpha}|t-u|^{2\alpha}+\sum_{i=1}^{d^2}C_{3,i}|u-s|^{\alpha}|t-u|^{3\alpha}\\
&\ +\sum_{i=1}^{d^2}C_{4,i}|u-s|^{\alpha}|t-u|^{3\alpha} \hspace{0.5cm}(\text{by Definitions~\ref{def:pbrp}, \ref{defn:cbrp1} and Remark~\ref{re:ieq}})\\
\le&\ C_{1,1}|t-s|^{4\alpha}+\sum_{i=1}^{d}C_{2,i}|t-s|^{4\alpha}+\sum_{i=1}^{d^2}C_{3,i}|t-s|^{4\alpha}+\sum_{i=1}^{d^2}C_{4,i}|t-s|^{4\alpha}\\
=:&\ C|t-s|^{4\alpha},
\end{align*}
where $C_{i,j}, C\in \RR$.
% and
%$$C:=C_{1,1}+\sum_{i=1}^{d}C_{2,i}+\sum_{i=1}^{d^2}C_{3,i}+\sum_{i=1}^{d^2}C_{4,i}.$$
Since $4\alpha >4\times \frac{1}{4}=1 $, $\tilde{Y}$ satisfies (\ref{eq:seweq}), as required.
\end{proof}

\begin{remark}
With the setting in Theorem~\ref{thm:inte1}, (\ref{eq:ieq1}) can be rewritten as
\[
\Big|\displaystyle\int_s^t Y_rdX_{r}^{a}-\sum_{\tau\in \mathcal{F}^{\leq 2}}\Y_{s}^\tau\X_{s,t}^{[\tau]_a}\Big|\leq C | t-s | ^{4\alpha}
\]
for some constant $C\in \RR$. Since $\fan{\Y} _{\X;\alpha }$ and $\fan{\X} _{\alpha }$ are constants in $\RR$, there is a constant $C_\alpha\in \RR$ such that
\begin{equation}
\Big|\displaystyle\int_0^t Y_rdX_{r}^{a}-\sum_{\tau\in \mathcal{F}^{\leq 2}}\Y^\tau_s\X^{[\tau]_{\bullet_a}}_{s,t}\Big|\le C_\alpha \fan{\Y} _{\X;\alpha }\fan{\X} _{\alpha }|t-s|^{4\alpha}.
\mlabel{eq:nt}
\end{equation}
\end{remark}

Similarly to the case of branched rough paths~\cite{Kel}, the path $ \R{$\Int_0^\bullet$} Y_rdX_{r}^{a}$ can be lifted to an $\X$-controlled planarly branched rough path
\begin{equation*}
\displaystyle\int_0^\bullet \Y_rdX_{r}^{a}:[0,T]\longrightarrow (\hmsNd)^n, \quad t\mapsto \displaystyle\int_0^t \Y_rdX_{r}^{a}
\end{equation*}
by setting, for $\tau\in \f^{\le 2}$ and $\tau _i\in \calt^{\le 3}$,
\begin{equation}
\langle\etree,\displaystyle\int_0^t\Y_{r}dX_{r}^{a}\rangle:=\int_0^tY_{r}dX_{r}^{a}, \quad
\langle [\tau]_{a},\displaystyle\int_0^t\Y_{r} dX_{r}\rangle:=\ \langle \tau,\Y_{t}\rangle, \quad
\langle \tau _1\cdots \tau _n,\displaystyle\int_0^t\Y_{r} dX_{r}\rangle:= 0.
\mlabel{eq:inte2}
\end{equation}
Now we turn to the integral $\R{$\Int_0^t$}{F(\Y_{r})\cdot dX_{r}}$. Let $\X:[0,T]^2\rightarrow \hesnt$ be an $\alpha$-H\"{o}lder planarly branched path above $X$.
Let
$$Z=(Z^1,\ldots,Z^d):[0,T]\rightarrow (\mathbb{R}^n)^d,\quad \Z=(\Z^{1}, \ldots, \Z^{d}):[0,T]\rightarrow \Big((\hmsNd)^n \Big)^d,$$
where each $\Z^{i}:[0,T]\rightarrow (\hmsNd)^n$ is an $\X$-controlled planarly branched rough path above $Z^i$ for $i=1, \ldots, d$. Define a path
\[
\displaystyle\int_0^\bullet Z_r\cdot dX_r:[0, T] \rightarrow \RR^n, \quad t\mapsto \displaystyle\int_0^t Z_r\cdot dX_r:=\sum_{i=1}^{d}\displaystyle\int_0^t Z_r^{i}dX_r^{i}.
\]
Here each $\R{$\Int_0^t$} Z_r^{i}dX_r^{i}$ in $\RR^n$ is given in Theorem~\ref{thm:inte1}.
In analogy to the case of branched rough paths~\cite{Kel}, the path $\R{$\Int_0^\bullet$} Z_r\cdot dX_r$ can be lifted to an $\X$-controlled planarly branched rough path
$$\displaystyle\int_0^\bullet\Z_r\cdot dX_r:[0,T]\rightarrow (\hmsNd)^n,\quad t\mapsto \displaystyle\int_0^t\Z_r \cdot dX_r,$$
where
\begin{equation}
\begin{aligned}
\langle\etree,\displaystyle\int_0^t\Z_{r}\cdot dX_{r}\rangle:=&
 \sum_{i=1}^{d}\langle\etree,\displaystyle\int_0^t\Z_{r}^{i}dX_{r}^{i}\rangle=\sum_{i=1}^{d}
 \displaystyle\int_0^tZ_{r}^{i}dX_{r}^{i} \in\RR^n,\\
\langle [\tau]_{a},\displaystyle\int_0^t\Z_{r}\cdot dX_{r}\rangle:=&\ \langle \tau,\Z_{t}^{a}\rangle\in\RR^n, \quad
\langle \tau _1\cdots \tau _n,\displaystyle\int_0^t\Z_{r}\cdot dX_{r}\rangle:= 0\in\RR^n, \quad \forall \tau\in \f^{\le 2}, \tau _i\in \calt^{\le 3}.
\end{aligned}
\mlabel{eq:inte2}
\end{equation}
Here the pairing in given in  ~\meqref{eq:pairy}.
%%%
\subsection{Universal limit theorem of planarly branched rough paths}\label{sec3.2}
%%%
Consider the RODE
\begin{equation*}
\left\{
\begin{array}{rll}
dY_{t}&={F}(Y_{t})\cdot dX_t=\sum\limits_{i=1}^{d}{F^i}(Y_{t}) dX_{t}^{i}, \quad \forall t\in [0, T],\\
Y_0 &= \xi  ,
   \end{array}\right.
%\mlabel{eq:irde}
\end{equation*}
which is understood by interpreting it as the integral equation
\begin{equation}
\left\{
\begin{array}{rll}
Y_{t} - Y_s& =\displaystyle\int_s^t{F(Y_{r})\cdot dX_{r}} ,\quad \forall s,t\in [0, T] ,\\
Y_0 &= \xi.
\end{array}\right.
\mlabel{eq:RDE}
\end{equation}
Solutions to~(\mref{eq:RDE}) are defined by lifting the problem to the space of $\X$-controlled planarly branched rough paths.

\begin{definition}
Let $\X:[0,T]^2\rightarrow \hesnt$ be an $\alpha$-H\"{o}lder planarly branched rough path above $X$. A path $Y:[0,T]\rightarrow \mathbb{R}^n$ with $Y_0 = \xi$ is called {\bf a solution} to  ~(\mref{eq:RDE}) if there is an $\X$-controlled planarly branched rough path $\Y:[0,T]\rightarrow (\hmsNd)^n $ above $Y$ such that
\begin{equation}
\Y_{t} - \Y_s =\displaystyle\int_s^t{F(\Y_{r})\cdot dX_{r}} , \quad \forall s,t\in [0,T].
\mlabel{eq:BRDE}
\end{equation}
Here
$$F(\Y)=(F^1(\Y), \ldots, F^d(\Y)):[0,T]\rightarrow \Big((\hmsNd)^n\Big)^d  $$ and each $F^i(\Y): [0,T]\rightarrow (\hmsNd)^n$ above $F^i(Y)$ is given in Proposition~\mref{pp:regu} for $i=1, \ldots, d$.
\mlabel{defn:bRDE}
\end{definition}

\begin{remark}
In  ~(\ref{eq:BRDE}), take s=0 and then $\Y_{t} - \Y_0 =\R{$\Int_0^t$}{F(\Y_{r})\cdot dX_{r}}$. Set $ \W:= \Y- \R{$\Int_0^\bullet$}{F(\Y_{r})\cdot dX_{r}}$, that is,
$$\W:[0,T]\rightarrow (\hmsNd)^n,\quad t\mapsto \Y_{t} - \displaystyle\int_0^t{F(\Y_{r})\cdot dX_{r}}
=\Y_0.$$
Since $\Y$ and $\R{$\Int_0^\bullet$}{F(\Y_{r})\cdot dX_{r}}$ are $\X$-controlled planarly branched rough paths in $\cbrpxd$ and $\cbrpxd$ is a Banach space, we have $\W \in \cbrpxd$.
\mlabel{rk:conty}
\end{remark}
According to  ~(\ref{eq:BRDE}), we defined a map
\begin{equation}
\mathcal{M}:\cbrpxd \rightarrow \cbrpxd, \quad \Y\mapsto \Y_{0}+\displaystyle\int_0^\bullet{F(\Y_{r})\cdot dX_{r}}.
\mlabel{eq:contram}
\end{equation}

We will apply the fixed point theorem in the Banach space
$(\cbrpxd, \fan{\cdot}_{\X;\alpha})$.
Assume that all paths discussed below are defined on $[0, \delta]$, where
$\delta < T$ is a sufficiently small positive number. As a preparation, let us establish an upper bound for $\mathcal{M}(Y)$.

\begin{lemma}
For $\alpha\in (\frac{1}{4},\frac{1}{3}]$, $\X\in \brpt$, $\Y\in \cbrpxt$ and  $\Y_0$ satisfied Definition ~\ref{defn:bRDE}, we have
$$\fan{\mathcal{M}\Y} _{\X;\alpha }\le C_{\alpha}\Big(\sum_{i=1}^{d}\|F^i\|_{C_{b}^{3}}\Big)\Bigg(\delta^{\alpha}M\Big(\delta, \sum_{a\in A}|\Y_{0}^{\bullet_a}|, \sum_{a,b\in A}|\Y_{0}^{\bullet_a\bullet_b}|, \sum_{a,b\in A}| \Y_{0}^{\tddeuxa{$a$}{$b$}}\ \,|, \fan{\Y} _{\X;\alpha }, \fan{\X} _{\alpha } \Big)+\fan{\X}_{\alpha } \Bigg)+4\fan{\W} _{\X;\alpha },$$
where $C_{\alpha}, \fan{\W} _{\X;\alpha }\in \RR$.
\mlabel{lem:jstab1}
\end{lemma}

\begin{proof}
It follows from  ~(\ref{eq:normy}) that
$$\fan{\mathcal{M}\Y}_{\X;\alpha }= \|R\mathcal{M}\Y^{\etree }\|_{3\alpha}+\sum_{a\in A}\|R\mathcal{M}\Y^{\bullet_a}\|_{2\alpha }+\sum_{a,b\in A}\|R\mathcal{M}\Y^{\tddeuxa{$a$}{$b$}}\ \,\|_{\alpha }+\sum_{a,b\in A}\| R\mathcal{M}\Y^{\bullet_a\bullet_b}\|_{\alpha }. $$
We only estimate the first term $\|R\mathcal{M}\Y^{\etree }\|_{3\alpha}$ explicitly, as the others are similar and simpler. We have
\begin{align*}
|\mathcal{M}\Y_{s,t}^{\etree }|
=&\ \Big|\Big(\langle\etree, \Y_{0}\rangle+\langle\etree, \displaystyle\int_0^t{F(\Y_{r})\cdot dX_{r}}\rangle\Big)-\Big(\langle\etree, \Y_{0}\rangle+\langle\etree, \displaystyle\int_0^s{F(\Y_{r})\cdot dX_{r}}\rangle\Big)\Big|\hspace{1cm}(\text{by  ~(\ref{eq:contram})})\\
=&\ \Big|\langle\etree, \displaystyle\int_s^t{F(\Y_{r})\cdot dX_{r}}\rangle\Big|= \Big|\sum_{i=1}^{d} \langle\etree, \displaystyle\int_s^t{F^i(\Y_{r}) dX_{r}^i}\rangle\Big|\hspace{1cm}(\text{by  ~(\ref{eq:inte2})})\\
=&\ \Big|\sum_{i=1}^{d} \displaystyle\int_s^t{F^i(Y_{r}) dX_{r}^i}\Big|.\hspace{5cm}(\text{by  ~(\ref{eq:inte1})})
\end{align*}
This implies that
\begin{align}
|R\mathcal{M}\Y_{s,t}^{\etree }|
%1
=&\ \Big|\mathcal{M}\Y^{\etree}_{s,t}-\Big(\sum_{a\in A}\mathcal{M}\Y_{s} ^{\bullet_a}\X_{s,t} ^{\bullet_a}+\sum_{a,b\in A}\mathcal{M}\Y_{s} ^{\bullet_a\bullet_b} \X_{s,t} ^{\bullet_a\bullet_b}+\sum_{a,b\in A}\mathcal{M}\Y_{s} ^{\tddeuxa{$a$}{$b$}}\ \, \X_{s,t} ^{\tddeuxa{$a$}{$b$}}\ \Big)\Big| \hspace{0.5cm} (\text{by  ~(\ref{eq:cbrp1})})\mlabel{eq:wq1}\\
%2
=&\ \Big|\sum_{i=1}^{d} \displaystyle\int_s^t{F^i(Y_{r}) dX_{r}^i}-\Bigg(\sum_{a\in A}\langle\bullet_a, \Y_0+\displaystyle\int_0^s{F(\Y_{r})\cdot dX_{r}}\rangle\X_{s,t} ^{\bullet_a}\nonumber\\
&\ +\sum_{a,b\in A}\langle\bullet_a\bullet_b, \Y_0+\displaystyle\int_0^t{F(\Y_{r})\cdot dX_{r}}\rangle\X_{s,t} ^{\bullet_a\bullet_b}+\sum_{a,b\in A}\langle\tddeuxa{$a$}{$b$}\ \,, \Y_0+\displaystyle\int_0^s{F(\Y_{r})\cdot dX_{r}}\rangle\X_{s,t} ^{\tddeuxa{$a$}{$b$}}\ \Bigg)\Big|\nonumber\\
&\ \hspace{8cm} (\text{by  ~(\ref{eq:contram})})\nonumber\\
%3
=&\ \Bigg|\sum_{i=1}^{d} \displaystyle\int_s^t{F^i(Y_{r}) dX_{r}^i}-\Bigg(\sum_{a\in A}(\Y_0^{\bullet_a}+F^a(\Y_s)^{\etree})\X_{s,t} ^{\bullet_a}+\sum_{a,b\in A}(\Y_0^{\bullet_a\bullet_b}+0)\X_{s,t} ^{\bullet_a\bullet_b}\nonumber\\
&\ \hspace{3.8cm}+\sum_{a,b\in A}(\Y_0^{\tddeuxa{$a$}{$b$}}\ \,+F^a(\Y_s)^{\bullet_b})\X_{s,t} ^{\tddeuxa{$a$}{$b$}}\ \Bigg)\Bigg|\hspace{1cm} (\text{by  ~(\ref{eq:inte2})})\nonumber\\
=&\ \Bigg|\sum_{a=1}^{d} \displaystyle\int_s^t{F^a(Y_{r}) dX_{r}^a}-\Big(\sum_{a\in A}F^a(\Y_s)^{\etree}\X_{s,t} ^{\bullet_a}+\sum_{a,b\in A}F^a(\Y_s)^{\bullet_b}\X_{s,t} ^{\tddeuxa{$a$}{$b$}}\ \Big)\nonumber\\
&\ -\Big(\sum_{a\in A}\Y_0^{\bullet_a}\X_{s,t} ^{\bullet_a}+\sum_{a,b\in A}\Y_0^{\bullet_a\bullet_b}\X_{s,t} ^{\bullet_a\bullet_b}+\sum_{a,b\in A}\Y_0^{\tddeuxa{$a$}{$b$}}\ \,\X_{s,t} ^{\tddeuxa{$a$}{$b$}}\ \Big)\Bigg|\nonumber\\
\le&\ \sum_{a=1}^{d}\Big| \displaystyle\int_s^t{F^a(Y_{r}) dX_{r}^a}-\Big(F^a(\Y_s)^{\etree}\X_{s,t} ^{\bullet_a}+\sum_{b\in A}F^a(\Y_s)^{\bullet_b}\X_{s,t} ^{\tddeuxa{$a$}{$b$}}\ \Big)\Big|\nonumber\\
&\ +\Big|\Big(\sum_{a\in A}\Y_0^{\bullet_a}\X_{s,t} ^{\bullet_a}+\sum_{a,b\in A}\Y_0^{\bullet_a\bullet_b}\X_{s,t} ^{\bullet_a\bullet_b}+\sum_{a,b\in A}\Y_0^{\tddeuxa{$a$}{$b$}}\ \,\X_{s,t} ^{\tddeuxa{$a$}{$b$}}\ \Big)\Big|.\hspace{0.5cm}(\text{by the triangle inequality})\nonumber
\end{align}
Introducing the two following expressions:
\begin{align*}
I_1:=&\ \sum_{a=1}^{d}\Big| \displaystyle\int_s^t{F^a(Y_{r}) dX_{r}^a}-\Big(F^a(\Y_s)^{\etree}\X_{s,t} ^{\bullet_a}+\sum_{b\in A}F^a(\Y_s)^{\bullet_b}\X_{s,t} ^{\tddeuxa{$a$}{$b$}}\ \Big)\Big|,\\
I_2:=&\ \Big|\Big(\sum_{a\in A}\Y_0^{\bullet_a}\X_{s,t} ^{\bullet_a}+\sum_{a,b\in A}\Y_0^{\bullet_a\bullet_b}\X_{s,t} ^{\bullet_a\bullet_b}+\sum_{a,b\in A}\Y_0^{\tddeuxa{$a$}{$b$}}\ \,\X_{s,t} ^{\tddeuxa{$a$}{$b$}}\ \Big)\Big|.
\end{align*}
We have
\begin{align*}
I_1=&\ \sum_{a=1}^{d}\Big|\displaystyle\int_s^t{F^a(Y_{r}) dX_{r}^a}-\Bigg(F^a(\Y_s)^{\etree}\X_{s,t} ^{\bullet_a}+\sum_{b\in A}F^a(\Y_s)^{\bullet_b}\X_{s,t} ^{\tddeuxa{$a$}{$b$}}\ +\sum_{b,c\in A}F^a(\Y_s)^{\bullet_b\bullet_c}\X_{s,t}^{\tdtroisuna{$a$}{$c$}{$b$}}\\
&\ +\sum_{b,c\in A}F^a(\Y_s)^{\tddeuxa{$b$}{$c$}}\ \,\X_{s,t} ^{\tdddeuxa{$a$}{$b$}{$c$}}\ \,\Bigg)+\Bigg(\sum_{b,c\in A}F^a(\Y_s)^{\bullet_b\bullet_c}\X_{s,t}^{\tdtroisuna{$a$}{$c$}{$b$}}+\sum_{b,c\in A}F^a(\Y_s)^{\tddeuxa{$b$}{$c$}}\ \,\X_{s,t} ^{\tdddeuxa{$a$}{$b$}{$c$}}\ \,\Bigg)\Big|\\
%
%&\ \hspace{7cm}(\text{by adding and subtracting})\\
%
\le&\ \sum_{a=1}^{d}\Big|\R{$\Int_s^t$}{F^a(Y_{r}) dX_{r}^a}-\Big(F^a(\Y_s)^{\etree}\X_{s,t} ^{\bullet_a}+\sum_{b\in A}F^a(\Y_s)^{\bullet_b}\X_{s,t} ^{\tddeuxa{$a$}{$b$}}\ +\sum_{b,c\in A}F^a(\Y_s)^{\bullet_b\bullet_c}\X_{s,t}^{\tdtroisuna{$a$}{$c$}{$b$}}\\
&\ +\sum_{b,c\in A}F^a(\Y_s)^{\tddeuxa{$b$}{$c$}}\ \,\X_{s,t} ^{\tdddeuxa{$a$}{$b$}{$c$}}\ \,\Big)\Big|+\sum_{a=1}^{d}\Big|\Big(\sum_{b,c\in A}F^a(\Y_s)^{\bullet_b\bullet_c}\X_{s,t}^{\tdtroisuna{$a$}{$c$}{$b$}}+\sum_{b,c\in A}F^a(\Y_s)^{\tddeuxa{$b$}{$c$}}\ \,\X_{s,t} ^{\tdddeuxa{$a$}{$b$}{$c$}}\ \,\Big)\Big|\\
&\ \hspace{7cm}(\text{by the triangle inequality})\\
\le&\ \sum_{a=1}^{d}C_\alpha\fan{F^a(\Y)}_{\X;\alpha }\fan{\X}_{\alpha }|t-s|^{4\alpha}\\
&\ +\sum_{a=1}^{d}\sum_{b,c\in A}\Big(|F^a(\Y_s)^{\bullet_b\bullet_c}|\,\|\X^{\tdtroisuna{$a$}{$c$}{$b$}\ }\|_{3\alpha}|t-s|^{3\alpha}+|F^a(\Y_s)^{\tddeuxa{$b$}{$c$}}\ \,|\,\|\X^{\tdddeuxa{$a$}{$b$}{$c$}}\ \,\|_{3\alpha}|t-s|^{3\alpha} \Big).\\
&\ \hspace{7cm}(\text{by  ~(\ref{eq:normx}) and ~(\ref{eq:nt})})
\end{align*}
Let us pause for a moment to handle the second summand. It follows from~\meqref{eq:zcbrp1} that
\begin{equation}
F^a(\Y)_{s}^{\tddeuxa{$b$}{$c$}}\ \,= (F^a)'(\Y^{\etree}_{s})\Y^{\tddeuxa{$b$}{$c$}}_{s}\ \,,\quad F^a(\Y)_{s}^{\bullet_b\bullet_c}= (F^a)'(\Y^{\etree}_{s})\Y^{\bullet_b\bullet_c}_{s}+(F^a)''(\Y^{\etree}_{s})
\Y^{\bullet_b}_{s}\Y^{\bullet_c}_{s},
\mlabel{eq:myw1}
\end{equation}
Here
$(F^a)'(\Y^{\etree}_{s}):\mathbb{R}^n \longrightarrow  \mathbb{R}^n$ and
$(F^a)''(\Y^{\etree}_{s}):\mathbb{R}^n \otimes \mathbb{R}^n \longrightarrow  \mathbb{R}^n
$
are Fr$\acute{\text{e}}$chet derivatives.
We have
\begin{align*}
&\ \sum_{a=1}^{d}\sum_{b,c\in A}\Big(|F^a(\Y_s)^{\bullet_b\bullet_c}|\,\|\X^{\tdtroisuna{$a$}{$c$}{$b$}\ }\|_{3\alpha}|t-s|^{3\alpha}+|F^a(\Y_s)^{\tddeuxa{$b$}{$c$}}\ \,|\,\|\X^{\tdddeuxa{$a$}{$b$}{$c$}}\ \,\|_{3\alpha}|t-s|^{3\alpha} \Big)\hspace{2cm}\\
=&\  \sum_{a=1}^{d}\sum_{b,c\in A}\Big(|(F^a)'(\Y^{\etree}_{s})\Y^{\bullet_b\bullet_c}_{s}+(F^a)''(\Y^{\etree}_{s})\Y^{\bullet_b}_{s}\Y^{\bullet_c}_{s}|\,\|\X^{\tdtroisuna{$a$}{$c$}{$b$}\ }\|_{3\alpha}|t-s|^{3\alpha}+|(F^a)'(\Y^{\etree}_{s})\Y^{\tddeuxa{$b$}{$c$}}_{s}\ \,|\,\|\X^{\tdddeuxa{$a$}{$b$}{$c$}}\ \,\|_{3\alpha}|t-s|^{3\alpha} \Big)\nonumber\\
&\ \hspace{11cm} (\text{by  ~(\ref{eq:myw1})})\nonumber\\
\le&\  \sum_{a=1}^{d}\sum_{b,c\in A}\bigg(\Big(|(F^a)'(\Y^{\etree}_{s})\Y^{\bullet_b\bullet_c}_{s}|+|(F^a)''(\Y^{\etree}_{s})\Y^{\bullet_b}_{s}\Y^{\bullet_c}_{s}|\Big)\|\X^{\tdtroisuna{$a$}{$c$}{$b$}\ }\|_{3\alpha}|t-s|^{3\alpha}+|(F^a)'(\Y^{\etree}_{s})\Y^{\tddeuxa{$b$}{$c$}}_{s}\ \,|\,\|\X^{\tdddeuxa{$a$}{$b$}{$c$}}\ \,\|_{3\alpha}|t-s|^{3\alpha} \bigg)\nonumber\\
\le&\ \sum_{a=1}^{d}\sum_{b,c\in A}\Big(2\|F^a\|_{C_{b}^{3}}\|\X^{\tdtroisuna{$a$}{$c$}{$b$}\ }\|_{3\alpha}|t-s|^{3\alpha}+\|F^a\|_{C_{b}^{3}}\|\X^{\tdddeuxa{$a$}{$b$}{$c$}}\ \,\|_{3\alpha}|t-s|^{3\alpha} \Big)\hspace{1cm} (\text{by ~\eqref{eq:infF} and~\eqref{eq:fcbn}})\nonumber\\
=&\ \sum_{a=1}^{d}\|F^a\|_{C_{b}^{3}}\sum_{b,c\in A}\bigg(2\|\X^{\tdtroisuna{$a$}{$c$}{$b$}\ }\|_{3\alpha}+\|\X^{\tdddeuxa{$a$}{$b$}{$c$}}\ \,\|_{3\alpha}\bigg)|t-s|^{3\alpha}\nonumber\\
\le&\ 2\sum_{a=1}^{d}\|F^a\|_{C_{b}^{3}}\sum_{b,c\in A}\bigg(\|\X^{\tdtroisuna{$a$}{$c$}{$b$}\ }\|_{3\alpha}+\|\X^{\tdddeuxa{$a$}{$b$}{$c$}}\ \,\|_{3\alpha}\bigg)|t-s|^{3\alpha}\nonumber\\
\le&\ 2\sum_{a=1}^{d}\Big(\|F^a\|_{C_{b}^{3}}\Big)\fan{\X}_{\alpha }|t-s|^{3\alpha},\hspace{0.5cm}(\text{by ~(\ref{eq:distx})})\nonumber
\end{align*}
which together with the above bound about $I_1$ implies that
\begin{align}
&\dfrac{I_1}{|t-s|^{3\alpha}}\hspace{10cm}\mlabel{eq:jstab0eq}\\
%2
\le&\ \sum_{a=1}^{d}C_{\alpha}\delta ^{\alpha}\fan{F^a(\Y)}_{\X;\alpha }\fan{\X}_{\alpha }+2\sum_{a=1}^{d}\Big(\|F^a\|_{C_{b}^{3}}\Big)\fan{\X}_{\alpha }\hspace{0.5cm}(\text{by $0<s<t<\delta$})\nonumber\\
=&\ \sum_{i=1}^{d}C_{\alpha}\delta ^{\alpha}\fan{F^i(\Y)}_{\X;\alpha }\fan{\X}_{\alpha }+2\sum_{i=1}^{d}\Big(\|F^i\|_{C_{b}^{3}}\Big)\fan{\X}_{\alpha }\hspace{0.5cm}(\text{by replacing $a$ with $i$\ })\nonumber\\
%3
\le&\  \sum_{i=1}^{d}C_{\alpha}\delta^{\alpha}4d^2\|F^i\|_{C_{b}^{3}} M\Big(\delta, \sum_{a\in A}|\Y_{0}^{\bullet_a}|, \sum_{a,b\in A}|\Y_{0}^{\bullet_a\bullet_b}|, \sum_{a,b\in A}| \Y_{0}^{\tddeuxa{$a$}{$b$}}\ \,|, \fan{\Y} _{\X;\alpha }, \fan{\X} _{\alpha } \Big)+2\sum_{a=1}^{d}\Big(\|F^a\|_{C_{b}^{3}}\Big)\fan{\X}_{\alpha }\nonumber\\
&\ \hspace{8cm}(\text{by Theorem~\ref{thm:stab1}})\nonumber\\
%4
\le&\ (4C_{\alpha}d^2+2)\Big(\sum_{i=1}^{d}\|F^i\|_{C_{b}^{3}}\Big)\Bigg(\delta^{\alpha}M\Big(\delta, \sum_{a\in A}|\Y_{0}^{\bullet_a}|, \sum_{a,b\in A}|\Y_{0}^{\bullet_a\bullet_b}|, \sum_{a,b\in A}| \Y_{0}^{\tddeuxa{$a$}{$b$}}\ \,|, \fan{\Y} _{\X;\alpha }, \fan{\X} _{\alpha } \Big)+\fan{\X}_{\alpha } \Bigg)\nonumber\\
&\ \hspace{7cm}(\text{each above item appears in below})\nonumber\\
=:&\ C_{\alpha}\Big(\sum_{i=1}^{d}\|F^i\|_{C_{b}^{3}}\Big)\Bigg(\delta^{\alpha}M\Big(\delta, \sum_{a\in A}|\Y_{0}^{\bullet_a}|, \sum_{a,b\in A}|\Y_{0}^{\bullet_a\bullet_b}|, \sum_{a,b\in A}| \Y_{0}^{\tddeuxa{$a$}{$b$}}\ \,|, \fan{\Y} _{\X;\alpha }, \fan{\X} _{\alpha } \Big)+\fan{\X}_{\alpha } \Bigg).\nonumber
\end{align}

Moreover,
\begin{equation}
\begin{aligned}
\dfrac{I_2}{|t-s|^{3\alpha}}=&\ \dfrac{\Big|\W_{s,t}^\etree - \Big(\sum_{a\in A}\W_s^{\bullet_a}\X_{s,t} ^{\bullet_a}+\sum_{a,b\in A}\W_s^{\bullet_a\bullet_b}\X_{s,t} ^{\bullet_a\bullet_b}+\sum_{a,b\in A}\W_s^{\tddeuxa{$a$}{$b$}}\ \,\X_{s,t} ^{\tddeuxa{$a$}{$b$}}\ \Big)\Big|}{|t-s|^{3\alpha}} \hspace{1cm}(\text{by Remark~\ref{rk:conty} })\\
=&\ \dfrac{|R\W_{s,t}^\etree|}{|t-s|^{3\alpha}}\hspace{1cm}(\text{by  ~(\ref{eq:cbrp1})})\\
\le&\ \|R\W^\etree\|_{3\alpha}\hspace{1cm}(\text{by  ~(\ref{eq:remainry})})\\
\le&\ \fan{\W} _{\X;\alpha }\hspace{1cm}(\text{by  ~(\ref{eq:normy})})\\
<&\ \infty.  \hspace{1cm}(\text{by $\W$ being an $\X$-controlled planarly branched rough path})
\end{aligned}
\mlabel{eq:jstab1eq}
\end{equation}
Combining  ~\meqref{eq:jstab0eq} and~\meqref{eq:jstab1eq}, we conclude
\begin{align*}
&\ \|R\mathcal{M}\Y^{\etree}\|_{3\alpha } =\sup_{s\ne t\in [0,T]} \dfrac{|R\mathcal{M}\Y_{s,t}^{\etree }|}{|t-s|^{3\alpha}}\le \sup_{s\ne t\in [ 0,T]}\frac{I_1+I_2}{|t-s|^{3\alpha} }\le \sup_{s\ne t\in [ 0,T]}\frac{I_1}{|t-s|^{3\alpha} }+\sup_{s\ne t\in [ 0,T]}\frac{I_2}{|t-s|^{3\alpha} }\\
\le&\
C_{\alpha}\Big(\sum_{i=1}^{d}\|F^i\|_{C_{b}^{3}}\Big)\Bigg(\delta^{\alpha}M\Big(\delta, \sum_{a\in A}|\Y_{0}^{\bullet_a}|, \sum_{a,b\in A}|\Y_{0}^{\bullet_a\bullet_b}|, \sum_{a,b\in A}| \Y_{0}^{\tddeuxa{$a$}{$b$}}\ \,|, \fan{\Y} _{\X;\alpha }, \fan{\X} _{\alpha } \Big)+\fan{\X}_{\alpha } \Bigg)+\fan{\W} _{\X;\alpha }. \qedhere
\end{align*}
\end{proof}

The following result gives the continuity estimation of $\mathcal{M}\Y $.

\begin{lemma}
Let $\alpha\in (\frac{1}{4},\frac{1}{3}]$, $\X \in \brpt $, $\Y, \tilde{\Y}\in \cbrpxt$ with $\Y_{0} =\tilde\Y_{0}$. Then
\begin{align*}
&\fan{\mathcal{M}\Y,\mathcal{M}\tilde\Y}_{\X, \X; \alpha}\\
\le&\  C_{\alpha}\delta^{\alpha}\Big(\sum_{i=1}^{d}\|F^i\|_{C_{b}^{3}}\Big)M\Big(\delta, \sum_{a\in A}|\Y_{0}^{\bullet_a}|, \sum_{a,b\in A}|\Y_{0}^{\bullet_a\bullet_b}|, \sum_{a,b\in A}| \Y_{0}^{\tddeuxa{$a$}{$b$}}\ \,|, \fan{\Y} _{\X;\alpha }, \fan{\tilde\Y} _{\X;\alpha},\fan{\X} _{\alpha}\Big)\fan{\Y,\tilde{\Y} }_{\X,\tilde{\X};\alpha},
\end{align*}
where $C_{\alpha}\in \RR$.
\mlabel{lem:jstab2}
\end{lemma}

\begin{proof}
Invoking~(\ref{eq:dist}),
\begin{align}
\fan {\mathcal{M}\Y,\mathcal{M}\tilde{\Y}}_{\X, \X;\alpha}
=&\ \|R\mathcal{M}\Y^{\etree}-R\mathcal{M}\tilde\Y^{\etree}\| _{3\alpha} +\sum_{a\in A}\|R\mathcal{M}\Y^{\bullet_a}-R\mathcal{M}\tilde\Y^{\bullet_a}\|_{2\alpha}\hspace{3cm}\mlabel{eq:ztz}\\
&\ +\sum_{a,b\in A}\|R\mathcal{M}\Y^{\bullet_a\bullet_b}-R\mathcal{M}\tilde\Y^{\bullet_a\bullet_b}\|_{\alpha}+\sum_{a,b\in A}\|R\mathcal{M}\Y^{\tddeuxa{$a$}{$b$}}\ \,-R\mathcal{M}\tilde\Y^{\tddeuxa{$a$}{$b$}}\ \,\|_{\alpha}.\notag
\end{align}
We just give the bound of the first term on the right hand side of the above equation in detail, as the others are treated analogously. We have
\begin{align*}
&\ |R\mathcal{M}\Y_{s,t}^{\etree }-R\mathcal{M}\tilde\Y_{s,t}^{\etree }|\\
%1
=&\ \Bigg|\Bigg(\sum_{a=1}^{d} \R{$\Int_s^t$}{F^a(Y_{r}) dX_{r}^a}-\Big(\sum_{a\in A}F^a(\Y_s)^{\etree}\X_{s,t} ^{\bullet_a}+\sum_{a,b\in A}F^a(\Y_s)^{\bullet_b}\X_{s,t} ^{\tddeuxa{$a$}{$b$}}\ \Big)\\
&\ \hspace{0.15cm}-\Big(\sum_{a\in A}\Y_0^{\bullet_a}\X_{s,t} ^{\bullet_a}+\sum_{a,b\in A}\Y_0^{\bullet_a\bullet_b}\X_{s,t} ^{\bullet_a\bullet_b}+\sum_{a,b\in A}\Y_0^{\tddeuxa{$a$}{$b$}}\ \,\X_{s,t} ^{\tddeuxa{$a$}{$b$}}\ \Big)\Bigg)\\
&\ -\Bigg(\sum_{a=1}^{d} \R{$\Int_s^t$}{F^a(\tilde Y_{r}) dX_{r}^a}-\Big(\sum_{a\in A}F^a(\tilde\Y_s)^{\etree}\X_{s,t} ^{\bullet_a}+\sum_{a,b\in A}F^a(\tilde\Y_s)^{\bullet_b}\X_{s,t} ^{\tddeuxa{$a$}{$b$}}\ \Big)\\
&\ \hspace{0.5cm}-\Big(\sum_{a\in A}\tilde\Y_0^{\bullet_a}\X_{s,t} ^{\bullet_a}+\sum_{a,b\in A}\tilde\Y_0^{\bullet_a\bullet_b}\X_{s,t} ^{\bullet_a\bullet_b}+\sum_{a,b\in A}\tilde\Y_0^{\tddeuxa{$a$}{$b$}}\ \,\X_{s,t} ^{\tddeuxa{$a$}{$b$}}\ \Big)\Bigg)\Bigg|\hspace{0.5cm}(\text{by  ~(\ref{eq:wq1})})\\
=&\ \Bigg| \sum_{a=1}^{d}\R{$\Int_s^t$}{F^a(Y_{r})-F^a(\tilde Y_{r}) dX_{r}^a}-\Big(\sum_{a\in A}(F^a(\Y_s)^{\etree}-F^a(\tilde\Y_s)^{\etree})\X_{s,t} ^{\bullet_a} +\sum_{a,b\in A}(F^a(\Y_s)^{\bullet_b}-F^a(\tilde\Y_s)^{\bullet_b})\X_{s,t} ^{\tddeuxa{$a$}{$b$}}\ \Big)\Bigg|\\
&\ \hspace{8cm}(\text{by $\Y_{0}^\tau =\tilde\Y_{0}^\tau$})\\
\le&\ \sum_{a=1}^{d}\Bigg|\R{$\Int_s^t$}{F^a(Y_{r})-F^a(\tilde Y_{r}) dX_{r}^a}-\Big((F^a(\Y_s)^{\etree}-F^a(\tilde\Y_s)^{\etree})\X_{s,t} ^{\bullet_a} +\sum_{b\in A}(F^a(\Y_s)^{\bullet_b}-F^a(\tilde\Y_s)^{\bullet_b})\X_{s,t} ^{\tddeuxa{$a$}{$b$}}\ \Big)\Bigg|\\
&\ \hspace{7cm}(\text{by the triangle inequality})\\
\le &\ \sum_{a=1}^{d}\Bigg|\R{$\Int_s^t$}{F^a(Y_{r})-F^a(\tilde Y_{r}) dX_{r}^a}-\Big((F^a(\Y_s)^{\etree}-F^a(\tilde\Y_s)^{\etree})\X_{s,t} ^{\bullet_a}\\
&\ \hspace{0.8cm}+\sum_{b\in A}(F^a(\Y_s)^{\bullet_b}-F^a(\tilde\Y_s)^{\bullet_b})\X_{s,t} ^{\tddeuxa{$a$}{$b$}}\ \, +\sum_{b,c\in A}(F^a(\Y_s)^{\bullet_b\bullet_c}-F^a(\tilde\Y_s)^{\bullet_b\bullet_c})\X_{s,t}^{\tdtroisuna{$a$}{$c$}{$b$}}\\
&\ \hspace{0.8cm}+\sum_{b,c\in A}(F^a(\Y_s)^{\tddeuxa{$b$}{$c$}}\ \,-F^a(\tilde\Y_s)^{\tddeuxa{$b$}{$c$}}\ \,)\X_{s,t} ^{\tdddeuxa{$a$}{$b$}{$c$}}\ \Big)\Bigg|\\
&\ +\sum_{a=1}^{d}\Bigg|-\sum_{b,c\in A}(F^a(\Y_s)^{\bullet_b\bullet_c}-F^a(\tilde\Y_s)^{\bullet_b\bullet_c})
\X_{s,t}^{\tdtroisuna{$a$}{$c$}{$b$}} \, - \sum_{b,c\in A}(F^a(\Y_s)^{\tddeuxa{$b$}{$c$}}\ \,-F^a(\tilde\Y_s)^{\tddeuxa{$b$}{$c$}}\ \,)\X_{s,t} ^{\tdddeuxa{$a$}{$b$}{$c$}}\ \,\Bigg|\\
&\ \hspace{6cm}(\text{by the triangle inequality})\\
\le&\ \sum_{a=1}^{d}C_\alpha\fan{F^a(\Y)-F^a(\tilde\Y)}_{\X;\alpha }\fan{\X}_{\alpha }|t-s|^{4\alpha}\\
&\ +\sum_{a=1}^{d}\Bigg|\sum_{b,c\in A}(F^a(\Y_s)^{\bullet_b\bullet_c}-F^a(\tilde\Y_s)^{\bullet_b\bullet_c})\X_{s,t}^{\tdtroisuna{$a$}{$c$}{$b$}}+\sum_{b,c\in
A}(F^a(\Y_s)^{\tddeuxa{$b$}{$c$}}\ \,-F^a(\tilde\Y_s)^{\tddeuxa{$b$}{$c$}}\ \,)\X_{s,t} ^{\tdddeuxa{$a$}{$b$}{$c$}}\ \,\Bigg|. \quad (\text{by (\ref{eq:nt})})
\end{align*}
For simplicity, denote by
\begin{align*}
I_1:=&\ \sum_{a=1}^{d}C_\alpha\fan{F^a(\Y)-F^a(\tilde\Y)}_{\X;\alpha }\fan{\X}_{\alpha }|t-s|^{4\alpha},\\
I_2:=&\ \sum_{a=1}^{d}\Bigg|\sum_{b,c\in A}(F^a(\Y_s)^{\bullet_b\bullet_c}-F^a(\tilde\Y_s)^{\bullet_b\bullet_c})\X_{s,t}^{\tdtroisuna{$a$}{$c$}{$b$}}+\sum_{b,c\in
A}(F^a(\Y_s)^{\tddeuxa{$b$}{$c$}}\ \,-F^a(\tilde\Y_s)^{\tddeuxa{$b$}{$c$}}\ \,)\X_{s,t} ^{\tdddeuxa{$a$}{$b$}{$c$}}\ \,\Bigg|.
\end{align*}
We have
\begin{align}
&\ \frac{I_1}{|t-s|^{3\alpha} }\mlabel{eq:lastI1}\\
\le&\ \sum_{a=1}^{d}C_\alpha\delta^{\alpha}\fan{F^a(\Y)-F^a(\tilde\Y)}_{\X;\alpha }\fan{\X}_{\alpha } \quad (\text{by $0<s<t<\delta$})\notag\\
=&\ \sum_{i=1}^{d}C_\alpha\delta^{\alpha}\fan{F^i(\Y), F^i(\tilde\Y)}_{\X, \X;\alpha }\fan{\X}_{\alpha }\quad (\text{by  ~(\ref{eq:dist}) and replacing $a$ by $i$} )\notag\\
\le&\ \sum_{i=1}^{d}C_\alpha\delta^{\alpha}4d^2\|F^i\|_{C_{b}^{3}} M\Big(\delta, \sum_{a\in A}|\Y_{0}^{\bullet_a}|, \sum_{a,b\in A}|\Y_{0}^{\bullet_a\bullet_b}|, \sum_{a,b\in A}| \Y_{0}^{\tddeuxa{$a$}{$b$}}\ \,|, \fan{\Y} _{\X;\alpha }, \fan{\tilde\Y} _{\X;\alpha},\fan{\X} _{\alpha}\Big)\fan{\Y,\tilde{\Y} }_{\X,\X;\alpha}.\notag
\end{align}
Here the last step employs Theorem~\ref{thm:stab2} with $\Y,\tilde{\Y}\in \cbrpxt$ and  $\Y_0=\tilde\Y_0$ by the hypothesis.
For $I_2$, we have
\begin{align}
&\ \frac{I_2}{|t-s|^{3\alpha} } \mlabel{eq:last0}
\\
\le&\ \sum_{a=1}^{d}\sum_{b,c\in A}\Big|F^a(\Y_s)^{\bullet_b\bullet_c}-F^a(\tilde\Y_s)^{\bullet_b\bullet_c}\Big|\frac{\Big|\X_{s,t}^{\tdtroisuna{$a$}{$c$}{$b$}}\ \,\Big|}{|t-s|^{3\alpha}}+ \sum_{a=1}^{d}\sum_{b,c\in
A}\Big|F^a(\Y_s)^{\tddeuxa{$b$}{$c$}}\ \,-F^a(\tilde\Y_s)^{\tddeuxa{$b$}{$c$}}\ \,\Big|\frac{\Big|\X_{s,t} ^{\tdddeuxa{$a$}{$b$}{$c$}}\ \,\Big|}{|t-s|^{3\alpha}} \notag\\
&\ \hspace{7cm}(\text{by the triangle inequality})\notag\\
%2
\le&\ \bigg(\sum_{i=1}^{d}\sum_{b,c\in A}\Big|F^i(\Y_s)^{\bullet_b\bullet_c}-F^i(\tilde\Y_s)^{\bullet_b\bullet_c}\Big|\bigg)\bigg(\sum_{a,b,c\in A} \frac{\Big|\X_{s,t}^{\tdtroisuna{$a$}{$c$}{$b$}}\ \,\Big|}{|t-s|^{3\alpha}}\bigg)\notag\\
&\ +\bigg(\sum_{i=1}^{d}\sum_{b,c\in
A}\Big|F^i(\Y_s)^{\tddeuxa{$b$}{$c$}}\ \,-F^i(\tilde\Y_s)^{\tddeuxa{$b$}{$c$}}\ \,\Big| \bigg)\bigg(\sum_{a,b,c\in A}\frac{\Big|\X_{s,t} ^{\tdddeuxa{$a$}{$b$}{$c$}}\ \,\Big|}{|t-s|^{3\alpha}}\bigg) \quad(\text{each above item appears in below})\notag\\
%3
\le&\ \sum_{i=1}^{d}\sum_{b,c\in A}\Big|F^i(\Y_s)^{\bullet_b\bullet_c}-F^i(\tilde\Y_s)^{\bullet_b\bullet_c}\Big|\,\fan{\X}_{\alpha }+\sum_{i=1}^{d}\sum_{b,c\in
A}\Big|F^i(\Y_s)^{\tddeuxa{$b$}{$c$}}\ \,-F^i(\tilde\Y_s)^{\tddeuxa{$b$}{$c$}}\ \,\Big|\,\fan{\X}_{\alpha }\notag\\
&\ \hspace{4cm}(\text{by the triangle inequality and  ~\eqref{eq:normx}, (\ref{eq:distx})})\notag\\
%4
=&\ \sum_{i=1}^{d}\sum_{b,c\in A}\Big|\Big((F^i)'(\Y^{\etree}_{s})\Y^{\bullet_b\bullet_c}_{s}+(F^i)''(\Y^{\etree}_{s})\Y^{\bullet_b}_{s}\Y^{\bullet_c}_{s}\Big)-\Big((F^i)'(\tilde\Y^{\etree}_{s})\tilde\Y^{\bullet_b\bullet_c}_{s}+(F^i)''(\tilde\Y^{\etree}_{s})\tilde\Y^{\bullet_b}_{s}\tilde\Y^{\bullet_c}_{s}\Big)\Big|\,\fan{\X}_{\alpha }\notag\\
&+\sum_{i=1}^{d}\sum_{b,c\in A}\Big|(F^i)'(\Y^{\etree}_{s})\Y^{\tddeuxa{$b$}{$c$}}_{s}\ \,-(F^i)'(\tilde\Y^{\etree}_{s})\tilde\Y^{\tddeuxa{$b$}{$c$}}_{s}\ \,\Big|\,\fan{\X}_{\alpha }\hspace{3cm} (\text{by ~(\ref{eq:myw1})})\notag\\
%5
\le&\ \sum_{i=1}^{d}\sum_{b,c\in A}\Bigg(\Big|(F^i)'(\Y^{\etree}_{s})\Y^{\bullet_b\bullet_c}_{s}-(F^i)'(\tilde\Y^{\etree}_{s})\tilde\Y^{\bullet_b\bullet_c}_{s}\Big|+
\Big|(F^i)''(\Y^{\etree}_{s})\Y^{\bullet_b}_{s}\Y^{\bullet_c}_{s}-(F^i)''(\tilde\Y^{\etree}_{s})\tilde\Y^{\bullet_b}_{s}\tilde\Y^{\bullet_c}_{s}\Big|\notag\\
&\ \hspace{1.5cm}+\Big|(F^i)'(\Y^{\etree}_{s})\Y^{\tddeuxa{$b$}{$c$}}_{s}\ \,-(F^i)'(\tilde\Y^{\etree}_{s})\tilde\Y^{\tddeuxa{$b$}{$c$}}_{s}\ \,\Big|\Bigg)\fan{\X}_{\alpha }\hspace{2cm} (\text{by the triangle inequality})\notag\\
=&\ \sum_{i=1}^{d}\sum_{b,c\in A}\Bigg(\Big|\Big((F^i)'(\Y^{\etree}_{s})\Y^{\bullet_b\bullet_c}_{s}-(F^i)'(\tilde\Y^{\etree}_{s})\Y^{\bullet_b\bullet_c}_{s}\Big)+\Big((F^i)'(\tilde\Y^{\etree}_{s})\Y^{\bullet_b\bullet_c}_{s}-(F^i)'(\tilde\Y^{\etree}_{s})\tilde\Y^{\bullet_b\bullet_c}_{s}\Big)\Big|\notag\\
&\ +\Big|\Big((F^i)''(\Y^{\etree}_{s})\Y^{\bullet_b}_{s}\Y^{\bullet_c}_{s}-(F^i)''(\tilde\Y^{\etree}_{s})\Y^{\bullet_b}_{s}\Y^{\bullet_c}_{s}\Big)+\Big((F^i)''(\tilde\Y^{\etree}_{s})\Y^{\bullet_b}_{s}\Y^{\bullet_c}_{s}-(F^i)''(\tilde\Y^{\etree}_{s})\tilde\Y^{\bullet_b}_{s}\tilde\Y^{\bullet_c}_{s}\Big)\Big|\notag\\
&\ +\Big|\Big((F^i)'(\Y^{\etree}_{s})\Y^{\tddeuxa{$b$}{$c$}}_{s}\ \,-(F^i)'(\tilde\Y^{\etree}_{s})\Y^{\tddeuxa{$b$}{$c$}}_{s}\ \,\Big)+\Big((F^i)'(\tilde\Y^{\etree}_{s})\Y^{\tddeuxa{$b$}{$c$}}_{s}\ \,-(F^i)'(\tilde\Y^{\etree}_{s})\tilde\Y^{\tddeuxa{$b$}{$c$}}_{s}\ \,\Big)\Big|\Bigg)\fan{\X}_{\alpha }\notag\\
\le&\ \sum_{i=1}^{d}\sum_{b,c\in A}\Bigg(\Big|(F^i)'(\Y^{\etree}_{s})-(F^i)'(\tilde\Y^{\etree}_{s})\Big|\,\Big|\Y^{\bullet_b\bullet_c}_{s}\Big|+\Big|(F^i)'(\tilde\Y^{\etree}_{s})\Big|\,\Big|\Y^{\bullet_b\bullet_c}_{s}-\tilde\Y^{\bullet_b\bullet_c}_{s}\Big|\notag\\
&\ +\Big|(F^i)''(\Y^{\etree}_{s})-(F^i)''(\tilde\Y^{\etree}_{s})\Big|\,\Big|\Y^{\bullet_b}_{s}\Y^{\bullet_c}_{s}\Big|+\Big|(F^i)''(\tilde\Y^{\etree}_{s})\Big|\,\Big|\Y^{\bullet_b}_{s}\Y^{\bullet_c}_{s}-\tilde\Y^{\bullet_b}_{s}\tilde\Y^{\bullet_c}_{s}\Big|\notag\\
&\ +\Big|(F^i)'(\Y^{\etree}_{s})-(F^i)'(\tilde\Y^{\etree}_{s})\Big|\,\Big|\Y^{\tddeuxa{$b$}{$c$}}_{s}\ \,\Big|+\Big|(F^i)'(\tilde\Y^{\etree}_{s})\Big|\,\Big|\Y^{\tddeuxa{$b$}{$c$}}_{s}\ \,-\tilde\Y^{\tddeuxa{$b$}{$c$}}_{s}\ \,\Big|\Bigg)\fan{\X}_{\alpha }\notag\\
&\ \hspace{6cm} (\text{by the triangle inequality})\notag\\
%\|F''\|_{\infty}
\le&\  \sum_{i=1}^{d}\sum_{b,c\in A}\Big(\|(F^i)''\|_{\infty}|\Y^{\etree}_{s}-\tilde\Y^{\etree}_{s}|\,|\Y^{\bullet_b\bullet_c}_{s}|+\|(F^i)'\|_{\infty}|\Y^{\bullet_b\bullet_c}_{s}-\tilde\Y^{\bullet_b\bullet_c}_{s}|\notag\\
&\ +\|(F^i)'''\|_{\infty}|\Y^{\etree}_{s}-\tilde\Y^{\etree}_{s}|\,|\Y^{\bullet_b}_{s}|\,|\Y^{\bullet_c}_{s}|+\|(F^i)''\|_{\infty}|\Y^{\bullet_b}_{s}\Y^{\bullet_c}_{s}-\tilde\Y^{\bullet_b}_{s}\tilde\Y^{\bullet_c}_{s}|\notag\\
&\ +\|(F^i)''\|_{\infty}|\Y^{\etree}_{s}-\tilde\Y^{\etree}_{s}|\,|\Y^{\tddeuxa{$b$}{$c$}}_{s}\ \,|+\|(F^i)'\|_{\infty}|\Y^{\tddeuxa{$b$}{$c$}}_{s}\ \,-\tilde\Y^{\tddeuxa{$b$}{$c$}}_{s}\ \,|\Big)\fan{\X}_{\alpha }\notag\\
&\ \hspace{2cm} (\text{by the differential mean value theorem of $(F^i)''$ and  ~\eqref{eq:infF}}) \notag\\
\le&\ \sum_{i=1}^{d}\sum_{b,c\in A}\|F^i\|_{C_{b}^{3}}\Big(|\Y^{\etree}_{s}-\tilde\Y^{\etree}_{s}|\,|\Y^{\bullet_b\bullet_c}_{s}| +|\Y^{\etree}_{s}-\tilde\Y^{\etree}_{s}|\,|\Y^{\bullet_b}_{s}|\,|\Y^{\bullet_c}_{s}|+|\Y^{\etree}_{s}-\tilde\Y^{\etree}_{s}|\,|\Y^{\tddeuxa{$b$}{$c$}}_{s}\ \,|\Big)\fan{\X}_{\alpha }\notag\\
&\ +\sum_{i=1}^{d}\sum_{b,c\in A}\|F^i\|_{C_{b}^{3}}\Big(|\Y^{\bullet_b\bullet_c}_{s}-\tilde\Y^{\bullet_b\bullet_c}_{s}|+|\Y^{\bullet_b}_{s}\Y^{\bullet_c}_{s}-\tilde\Y^{\bullet_b}_{s}\tilde\Y^{\bullet_c}_{s}|+|\Y^{\tddeuxa{$b$}{$c$}}_{s}\ \,-\tilde\Y^{\tddeuxa{$b$}{$c$}}_{s}\ \,|  \Big)\fan{\X}_{\alpha }.\notag
\end{align}
Now we estimate the two terms on the right hand side of the above equation. For the first term,
\begin{align}
&\sum_{i=1}^{d}\sum_{b,c\in A}\|F^i\|_{C_{b}^{3}}\Big(|\Y^{\etree}_{s}-\tilde\Y^{\etree}_{s}|\,|\Y^{\bullet_b\bullet_c}_{s}| +|\Y^{\etree}_{s}-\tilde\Y^{\etree}_{s}|\,|\Y^{\bullet_b}_{s}|\,|\Y^{\bullet_c}_{s}|+|\Y^{\etree}_{s}-\tilde\Y^{\etree}_{s}|\,|\Y^{\tddeuxa{$b$}{$c$}}_{s}\ \,|\Big)\fan{\X}_{\alpha }\mlabel{eq:last1}\\
=&\ \sum_{i=1}^{d}\sum_{b,c\in A}\|F^i\|_{C_{b}^{3}}\Big(|\Y^{\etree}_{0,s}-\tilde\Y^{\etree}_{0,s}|\,|\Y^{\bullet_b\bullet_c}_{s}| +|\Y^{\etree}_{0,s}-\tilde\Y^{\etree}_{0,s}|\,|\Y^{\bullet_b}_{s}|\,|\Y^{\bullet_c}_{s}|+|\Y^{\etree}_{0,s}-\tilde\Y^{\etree}_{0,s}|\,|\Y^{\tddeuxa{$b$}{$c$}}_{s}\ \,|\Big)\fan{\X}_{\alpha }\notag\\
&\hspace{8cm} (\text{by $\Y^{\etree}_{0}=\tilde\Y^{\etree}_{0}$})\notag\\
=&\ \sum_{i=1}^{d}\sum_{b,c\in A}\|F^i\|_{C_{b}^{3}}\Big(|R\Y^{\etree}_{0,s}-R\tilde\Y^{\etree}_{0,s}|\,|\Y^{\bullet_b\bullet_c}_{s}| +|R\Y^{\etree}_{0,s}-R\tilde\Y^{\etree}_{0,s}|\,|\Y^{\bullet_b}_{s}|\,|\Y^{\bullet_c}_{s}|+|R\Y^{\etree}_{0,s}-R\tilde\Y^{\etree}_{0,s}|\,|\Y^{\tddeuxa{$b$}{$c$}}_{s}\ \,|\Big)\fan{\X}_{\alpha }\notag\\
&\hspace{7cm} (\text{by $\Y_{0}=\tilde\Y_{0}$ and   ~(\ref{eq:cbrp1})})\notag\\
\le&\ \sum_{i=1}^{d}\sum_{b,c\in A}\|F^i\|_{C_{b}^{3}}\delta^{3\alpha}\frac{|R\Y^{\etree}_{0,s}-R\tilde\Y^{\etree}_{0,s}|}{|s|^{3\alpha}}\Big(|\Y^{\bullet_b\bullet_c}_{s}| +|\Y^{\bullet_b}_{s}|\,|\Y^{\bullet_c}_{s}|+|\Y^{\tddeuxa{$b$}{$c$}}_{s}\ \,|\Big)\fan{\X}_{\alpha }\quad (\text{by $0<s<\delta$})\notag\\
\le&\  \sum_{i=1}^{d}\sum_{b,c\in A}\|F^i\|_{C_{b}^{3}}\delta^{3\alpha}\|R\Y^{\etree}-R\tilde\Y^{\etree}\|_{3\alpha}\Big(|\Y^{\bullet_b\bullet_c}_{s}| +|\Y^{\bullet_b}_{s}|\,|\Y^{\bullet_c}_{s}|+|\Y^{\tddeuxa{$b$}{$c$}}_{s}\ \,|\Big)\fan{\X}_{\alpha }\quad (\text{by   ~(\ref{eq:remainry})})\notag\\
\le&\ \sum_{i=1}^{d}\delta^{\alpha}\|F^i\|_{C_{b}^{3}} \fan{\Y,\tilde{\Y} }_{\X,\X;\alpha} M\Big(\delta, \sum_{a\in A}|\Y_{0}^{\bullet_a}|, \sum_{a,b\in A}|\Y_{0}^{\bullet_a\bullet_b}|, \sum_{a,b\in A}| \Y_{0}^{\tddeuxa{$a$}{$b$}}\ \,|, \fan{\Y} _{\X;\alpha }, \fan{\tilde\Y} _{\X;\alpha},\fan{\X} _{\alpha}\Big).\notag\\
&\hspace{7cm} (\text{by   ~(\ref{eq:dist}), ~(\ref{eq:iequ3}), ~(\ref{eq:iequ2}) and ~(\ref{eq:iequ4})})\notag
\end{align}
For the second term,
\begin{align}
&\sum_{i=1}^{d}\sum_{b,c\in A}\|F^i\|_{C_{b}^{3}}\Big(|\Y^{\bullet_b\bullet_c}_{s}-\tilde\Y^{\bullet_b\bullet_c}_{s}|+|\Y^{\bullet_b}_{s}\Y^{\bullet_c}_{s}-\tilde\Y^{\bullet_b}_{s}\tilde\Y^{\bullet_c}_{s}|+|\Y^{\tddeuxa{$b$}{$c$}}_{s}\ \,-\tilde\Y^{\tddeuxa{$b$}{$c$}}_{s}\ \,|  \Big)\fan{\X}_{\alpha }\mlabel{eq:last2}\\
=&\ \sum_{i=1}^{d}\sum_{b,c\in A}\|F^i\|_{C_{b}^{3}}\Big(|\Y^{\bullet_b\bullet_c}_{0,s}-\tilde\Y^{\bullet_b\bullet_c}_{0,s}|+|\Y^{\bullet_b}_{s}\Y^{\bullet_c}_{s}-\tilde\Y^{\bullet_b}_{s}\tilde\Y^{\bullet_c}_{s}|+|\Y^{\tddeuxa{$b$}{$c$}}_{0, s}\ \,-\tilde\Y^{\tddeuxa{$b$}{$c$}}_{0, s}\ \,|\Big)\fan{\X}_{\alpha }\notag\\
&\hspace{3cm} (\text{by $\Y^{\bullet_b\bullet_c}_{0}=\tilde\Y^{\bullet_b\bullet_c}_{0}$, $\Y^{\tddeuxa{$b$}{$c$}}_{0}\ \,=\tilde\Y^{\tddeuxa{$b$}{$c$}}_{0}$\ \, and the triangle inequality})\notag\\
\le&\ \sum_{i=1}^{d}\sum_{b,c\in A}\|F^i\|_{C_{b}^{3}}\Big(\delta^{\alpha}\|R\Y^{\bullet_b\bullet_c}-R\tilde\Y^{\bullet_b\bullet_c}\|_{\alpha}+\delta^{\alpha}\|R\Y^{\tddeuxa{$b$}{$c$}}\ \,-R\tilde\Y^{\tddeuxa{$b$}{$c$}}\ \,\|_{\alpha}+|\Y^{\bullet_b}_{s}\Y^{\bullet_c}_{s}-\tilde\Y^{\bullet_b}_{s}\tilde\Y^{\bullet_c}_{s}|\Big)\fan{\X}_{\alpha }\notag\\
&\hspace{7cm} (\text{by   ~(\ref{eq:remainry}), (\ref{eq:cbrp3}) and (\ref{eq:cbrp4})})\notag\\
\le&\ \sum_{i=1}^{d}\delta^{\alpha}\|F^i\|_{C_{b}^{3}}\fan{\Y,\tilde{\Y} }_{\X,\X;\alpha}\fan{\X}_{\alpha}+\sum_{i=1}^{d}\sum_{b,c\in A}\|F^i\|_{C_{b}^{3}}|\Y^{\bullet_b}_{s}\Y^{\bullet_c}_{s}
-\tilde\Y^{\bullet_b}_{s}\tilde\Y^{\bullet_c}_{s}|\,\fan{\X}_{\alpha }. \quad (\text{by  ~(\ref{eq:dist})})\notag
\end{align}
The $|\Y^{\bullet_b}_{s}\Y^{\bullet_c}_{s}
-\tilde\Y^{\bullet_b}_{s}\tilde\Y^{\bullet_c}_{s}|$ is bounded by
\begin{align}
|\Y^{\bullet_b}_{s}\Y^{\bullet_c}_{s}
-\tilde\Y^{\bullet_b}_{s}\tilde\Y^{\bullet_c}_{s}|
=&\ |(\Y^{\bullet_b}_{s}-\tilde\Y^{\bullet_b}_{s})\Y^{\bullet_c}_{s}
+\tilde\Y^{\bullet_b}_{s}(\Y^{\bullet_c}_{s}-\tilde\Y^{\bullet_c}_{s})|
 \hspace{5cm}\mlabel{eq:last3}\\
\le&\ |\Y^{\bullet_b}_{s}-\tilde\Y^{\bullet_b}_{s}|\,|\Y^{\bullet_c}_{s}|+|\tilde\Y^{\bullet_b}_{s}|\,|\Y^{\bullet_c}_{s}-\tilde\Y^{\bullet_c}_{s}|\hspace{2cm}(\text{by the triangle inequality})\notag\\
=&\ |\Y^{\bullet_b}_{0, s}-\tilde\Y^{\bullet_b}_{0, s}|\,|\Y^{\bullet_c}_{s}|+|\tilde\Y^{\bullet_b}_{s}|\,|\Y^{\bullet_c}_{0, s}-\tilde\Y^{\bullet_c}_{0, s}|\hspace{2cm} (\text{by $\Y^{\bullet_b}_{0}=\tilde\Y^{\bullet_b}_{0}$ and $\Y^{\bullet_c}_{0}=\tilde\Y^{\bullet_c}_{0}$})\notag\\
%4
%=&\ \Big|\Big(\sum_{e\in A}(\Y_{0}^{\bullet_b\bullet_e}+\Y_{0}^{\tddeuxa{$b$}{$e$}}\ \,)\X_{0,s}^{\bullet_e}+R\Y_{0,s}^{\bullet_b}\Big)-\Big(\sum_{e\in A}(\tilde\Y_{0}^{\bullet_b\bullet_e}+\tilde\Y_{0}^{\tddeuxa{$b$}{$e$}}\ \,)\X_{0, s}^{\bullet_e}+R\tilde\Y_{0,s}^{\bullet_b}\Big)\Big|\,|\Y^{\bullet_c}_{s}|\\
%&\ +|\tilde\Y^{\bullet_b}_{s}|\,\Big|\Big(\sum_{e\in A}(\Y_{0}^{\bullet_c\bullet_e}+\Y_{0}^{\tddeuxa{$c$}{$e$}}\ \,)\X_{0,s}^{\bullet_e}+R\Y_{0,s}^{\bullet_c}\Big)-\Big(\sum_{e\in A}(\tilde\Y_{0}^{\bullet_c\bullet_e}+\tilde\Y_{0}^{\tddeuxa{$c$}{$e$}}\ \,)\X_{0,s}^{\bullet_e}+R\tilde\Y_{0,s}^{\bullet_c}\Big)\Big|\\
%%
%&\ \hspace{8cm}(\text{by  ~(\ref{eq:cbrp2})})\\
%
=&\ |R\Y_{0,s}^{\bullet_b}-R\tilde\Y_{0,s}^{\bullet_b}|\,|\Y^{\bullet_c}_{s}|
+|\tilde\Y^{\bullet_b}_{s}|\,|R\Y_{0,s}^{\bullet_c}-R\tilde\Y_{0,s}^{\bullet_c}|
\hspace{2cm}(\text{by  ~(\ref{eq:cbrp2}) and $\Y_{0}=\tilde\Y_{0}$})\notag\\
\le&\ \delta^{2\alpha}\|R\Y^{\bullet_b}-R\tilde\Y^{\bullet_b}\|_{2\alpha}|\Y^{\bullet_c}_{s}|
+\delta^{2\alpha}|\tilde\Y^{\bullet_b}_{s}|\,\|R\Y^{\bullet_c}
-R\tilde\Y^{\bullet_c}\|_{2\alpha}.
\quad (\text{by   ~(\ref{eq:cbrp3}) and (\ref{eq:cbrp4})})\notag
\end{align}
Substitute   ~(\ref{eq:last3}) into   ~(\ref{eq:last2}) and then substitute   ~(\ref{eq:last1}) and   ~(\ref{eq:last2}) into   ~(\ref{eq:last0}),
\begin{align}
&\ \frac{I_2}{|t-s|^{3\alpha} }\hspace{12.3cm} \mlabel{eq:lastI2}\\
\le&\ \sum_{i=1}^{d}\delta^{\alpha}\|F^i\|_{C_{b}^{3}} M\Big(\delta, \sum_{a\in A}|\Y_{0}^{\bullet_a}|, \sum_{a,b\in A}|\Y_{0}^{\bullet_a\bullet_b}|, \sum_{a,b\in A}| \Y_{0}^{\tddeuxa{$a$}{$b$}}\ \,|, \fan{\Y} _{\X;\alpha }, \fan{\tilde\Y} _{\X;\alpha},\fan{\X} _{\alpha}\Big)\fan{\Y,\tilde{\Y} }_{\X,\X;\alpha}\notag\\
&\ +\sum_{i=1}^{d}\delta^{\alpha}\|F^i\|_{C_{b}^{3}}\fan{\Y,\tilde{\Y} }_{\X,\X;\alpha}\fan{\X}_{\alpha}\notag\\
&\ +\sum_{i=1}^{d}\sum_{b,c\in A}\|F^i\|_{C_{b}^{3}}\bigg(\delta^{2\alpha}|											\Y^{\bullet_c}_{s}|\,\|R\Y^{\bullet_b}-R\tilde\Y^{\bullet_b}\|_{2\alpha}+\delta^{2\alpha}|\tilde\Y^{\bullet_b}_{s}|\,\|R\Y^{\bullet_c}-R\tilde\Y^{\bullet_c}\|_{2\alpha} \bigg)\fan{\X}_{\alpha}\notag\\
\le&\ \sum_{i=1}^{d}\delta^{\alpha}\|F^i\|_{C_{b}^{3}} M\Big(\delta, \sum_{a\in A}|\Y_{0}^{\bullet_a}|, \sum_{a,b\in A}|\Y_{0}^{\bullet_a\bullet_b}|, \sum_{a,b\in A}| \Y_{0}^{\tddeuxa{$a$}{$b$}}\ \,|, \fan{\Y} _{\X;\alpha }, \fan{\tilde\Y} _{\X;\alpha},\fan{\X} _{\alpha}\Big)\fan{\Y,\tilde{\Y} }_{\X,\X;\alpha}\notag\\
&\ +\sum_{i=1}^{d}\delta^{\alpha}\|F^i\|_{C_{b}^{3}}\fan{\Y,\tilde{\Y} }_{\X,\X;\alpha}\fan{\X}_{\alpha}\notag\\
&\ +\sum_{i=1}^{d}\delta^{2\alpha}\|F^i\|_{C_{b}^{3}}\bigg(\sum_{c\in A}|\Y^{\bullet_c}_{s}|+\sum_{b\in A}|\tilde\Y^{\bullet_b}_{s}| \bigg)\bigg(\sum_{b\in A}  \|R\Y^{\bullet_b}-R\tilde\Y^{\bullet_b}\|_{2\alpha}+\sum_{c\in A}\|R\Y^{\bullet_c}-R\tilde\Y^{\bullet_c}\|_{2\alpha} \bigg)\fan{\X}_{\alpha}\notag\\
&\ \hspace{6cm}(\text{each above item appears in below })\notag\\
\le&\ \sum_{i=1}^{d}\delta^{\alpha}\|F^i\|_{C_{b}^{3}} M\Big(\delta, \sum_{a\in A}|\Y_{0}^{\bullet_a}|, \sum_{a,b\in A}|\Y_{0}^{\bullet_a\bullet_b}|, \sum_{a,b\in A}| \Y_{0}^{\tddeuxa{$a$}{$b$}}\ \,|, \fan{\Y} _{\X;\alpha }, \fan{\tilde\Y} _{\X;\alpha},\fan{\X} _{\alpha}\Big)\fan{\Y,\tilde{\Y} }_{\X,\X;\alpha}\notag\\
&\ +\sum_{i=1}^{d}\delta^{\alpha}\|F^i\|_{C_{b}^{3}}\fan{\Y,\tilde{\Y} }_{\X,\X;\alpha}\fan{\X}_{\alpha}\notag\\
&\ +\sum_{i=1}^{d}\delta^{2\alpha}\|F^i\|_{C_{b}^{3}}\bigg(\sum_{c\in A}|\Y^{\bullet_c}_{s}|+\sum_{b\in A}|\tilde\Y^{\bullet_b}_{s}| \bigg)\fan{\Y,\tilde{\Y} }_{\X,\X;\alpha}\fan{\X}_{\alpha}\hspace{2cm} (\text{by   ~(\ref{eq:dist})})\notag\\
\le&\ \sum_{i=1}^{d}\delta^{\alpha}\|F^i\|_{C_{b}^{3}} M\Big(\delta, \sum_{a\in A}|\Y_{0}^{\bullet_a}|, \sum_{a,b\in A}|\Y_{0}^{\bullet_a\bullet_b}|, \sum_{a,b\in A}| \Y_{0}^{\tddeuxa{$a$}{$b$}}\ \,|, \fan{\Y} _{\X;\alpha }, \fan{\tilde\Y} _{\X;\alpha},\fan{\X} _{\alpha}\Big)\fan{\Y,\tilde{\Y} }_{\X,\X;\alpha}.\notag\\
&\ \hspace{6cm}(\text{by   ~(\ref{eq:iequ4}), $\Y_0=\tilde\Y_0$ and adjusting $M$})\notag
\end{align}
Hence we conclude
\begin{align*}
&\ \|R\mathcal{M}\Y^{\etree}-R\mathcal{M}\tilde\Y^{\etree}\|_{3\alpha } =\sup_{s\ne t\in [ 0,T]} \dfrac{|R\mathcal{M}\Y_{s,t}^{\etree }-R\mathcal{M}\tilde\Y_{s,t}^{\etree }|}{|t-s|^{3\alpha}}\le \sup_{s\ne t\in [ 0,T]}\frac{I_1+I_2}{|t-s|^{3\alpha} }\\
\le&\
(C_\alpha4d^2+1)\delta^{\alpha}\Big(\sum_{i=1}^{d}\|F^i\|_{C_{b}^{3}}\Big)M\Big(\delta, \sum_{a\in A}|\Y_{0}^{\bullet_a}|, \sum_{a,b\in A}|\Y_{0}^{\bullet_a\bullet_b}|, \sum_{a,b\in A}| \Y_{0}^{\tddeuxa{$a$}{$b$}}\ \,|, \fan{\Y} _{\X;\alpha }, \fan{\tilde\Y} _{\X;\alpha},\fan{\X} _{\alpha}\Big)\fan{\Y,\tilde{\Y} }_{\X,\X;\alpha}\\
&\ \hspace{5cm} (\text{by~(\mref{eq:lastI1}) and~(\mref{eq:lastI2})})\\
=:&\ C_\alpha\delta^{\alpha}\Big(\sum_{i=1}^{d}\|F^i\|_{C_{b}^{3}}\Big)M\Big(\delta, \sum_{a\in A}|\Y_{0}^{\bullet_a}|, \sum_{a,b\in A}|\Y_{0}^{\bullet_a\bullet_b}|, \sum_{a,b\in A}| \Y_{0}^{\tddeuxa{$a$}{$b$}}\ \,|, \fan{\Y} _{\X;\alpha }, \fan{\tilde\Y} _{\X;\alpha},\fan{\X} _{\alpha}\Big)\fan{\Y,\tilde{\Y} }_{\X,\X;\alpha}. \qedhere
\end{align*}
\end{proof}

By virtue of Lemma~\mref{lem:jstab1}, the finiteness of the function $M$ is required to ensure the finiteness of $\|\mathcal{M}\Y\|_{\X; \alpha}$. Specifically, since the parameters of $M$ are fixed except for $\|\Y\|_{\X; \alpha}$, the finiteness of $M$ is guaranteed by restricting $\Y$ to a bounded subset of $\cbrpxt$, such as a metric ball:
\begin{equation}
B_{\delta}(\W, R):=\{\Y\in \cbrpxt \mid \fan {\Y-\W}_{\X;\alpha}\le R, \Y_0=\W_0\},
\mlabel{eq:wyzero}
\end{equation}
which will serve as the working domain for applying the Banach fixed point theorem in what follows. Here, the subscript $\delta$ indicates that all relevant paths are restricted to $[0, \delta]$, and the radius $R$ is chosen as needed for subsequent applications.

\begin{theorem}(Local existence and uniqueness)
Let $\alpha\in (\frac{1}{4},\frac{1}{3}]$. For sufficiently small $\delta$, there is a unique $\Y\in \cbrpxt$ such that
\begin{equation*}
\Y_{t}=\Y_{0}+\displaystyle\int_0^t{F(\Y_{r})\cdot dX_{r}},
\end{equation*}	
where $t\in [0, \delta]$, that is,  ~\meqref{eq:RDE} has a unique (local) solution on $[0, \delta]$.
\mlabel{thm:leau}
\end{theorem}

\begin{proof}
We first show
\begin{equation}
\mathcal{M}(B_{\delta}(\W, R))\subseteq B_{\delta}(\W, R).
\mlabel{eq:mycont}
\end{equation}
Let $\Y\in B_{\delta}(\W, R)$. It follows from  ~\meqref{eq:contram} and \meqref{eq:wyzero} that
$\mathcal{M}\Y_{0}=\Y_0=\W_0 .$
In addition, by  ~(\ref{eq:jstab1eq}) we have $\fan{\W}_{\X;\alpha}=\fan{\W} _{\X;\alpha }$. Notice that $\mathcal{M}\Y-\W \in \cbrpxt$. Then
\begin{align*}
&\ \fan{\mathcal{M}\Y,\W}_{\X, \X; \alpha}\\
=&\ \fan{\mathcal{M}\Y-\W}_{\X; \alpha} \hspace{2cm} (\text{by~\meqref{eq:normy} and \meqref{eq:dist}})\\
\le&\ \fan{\mathcal{M}\Y}_{\X, \X; \alpha} +\fan{\W}_{\X; \alpha}\\
\le&\ \Bigg(C_{\alpha}\Big(\sum_{a=1}^{d}\|F^a\|_{C_{b}^{3}}\Big)\bigg(\delta^{\alpha}M\Big(\delta, \sum_{a\in A}|\Y_{0}^{\bullet_a}|, \sum_{a,b\in A}|\Y_{0}^{\bullet_a\bullet_b}|, \sum_{a,b\in A}| \Y_{0}^{\tddeuxa{$a$}{$b$}}\ \,|, \fan{\Y} _{\X;\alpha }, \fan{\X} _{\alpha } \Big)+\fan{\X}_{\alpha } \bigg)+4\fan{\W} _{\X;\alpha }\Bigg)\\
&\ +\fan{\W} _{\X;\alpha }. \hspace{9cm} (\text{by Lemma~\mref{lem:jstab1}})
\end{align*}
Hence
$$\fan{\mathcal{M}\Y, \W}_{\X;\alpha} \le\ C_{\alpha}\Big(\sum_{i=1}^{d}\|F^i\|_{C_{b}^{3}}\Big)\delta^{\alpha}M\Big(\delta, \sum_{a\in A}|\Y_{0}^{\bullet_a}|, \sum_{a,b\in A}|\Y_{0}^{\bullet_a\bullet_b}|, \sum_{a,b\in A}| \Y_{0}^{\tddeuxa{$a$}{$b$}}\ \,|, \fan{\Y} _{\X;\alpha }, \fan{\X} _{\alpha } \Big)+\dfrac{R}{2},$$
where
$$ R:=2\bigg(C_{\alpha}\Big(\sum_{i=1}^{d}\|F^i\|_{C_{b}^{3}}\Big)\fan{\X} _{\alpha }+4C_{\alpha}\Big(\sum_{i=1}^{d}\|F^i\|_{C_{b}^{3}}\Big)\fan{\W} _{\X;\alpha }+\fan{\W} _{\X;\alpha }\bigg).$$
Taking $\delta$ small enough such that	
$$C_{\alpha}\Big(\sum_{i=1}^{d}\|F^i\|_{C_{b}^{3}}\Big)\delta^{\alpha}M\Big(\delta, \sum_{a\in A}|\Y_{0}^{\bullet_a}|, \sum_{a,b\in A}|\Y_{0}^{\bullet_a\bullet_b}|, \sum_{a,b\in A}| \Y_{0}^{\tddeuxa{$a$}{$b$}}\ \,|, \fan{\Y} _{\X;\alpha }, \fan{\X} _{\alpha } \Big)\le \dfrac{R}{2},$$
we obtain $\mathcal{M}\Y\in B_{\delta}(\W, R)$.
%\begin{equation}
%\mathcal{M}\Y\in B_{\delta}(\W, R), \quad \mathcal{M}(B_{\delta}(\W, R))\subseteq B_{\delta}(\W, R).
%\mlabel{eq:mycont}
%\end{equation}

Next, for $\Y, \tilde\Y\in B_{\delta}(\W, R)$, since $\Y_0^\tau -\tilde{\Y}^\tau_0=W_0^\tau -W_0^\tau =0$ and $\Y$, $\tilde\Y$ are controlled planarly branched rough paths by the same $\X$, it follows from Lemma~\mref{lem:jstab2} that
\begin{align*}
&\ \fan{\mathcal{M}\Y-\mathcal{M}\tilde\Y}_{\X;\alpha}\\
\le&\ C_{\alpha}\delta^{\alpha}\Big(\sum_{i=1}^{d}\|F^i\|_{C_{b}^{3}}\Big)M\Big(\delta, \sum_{a\in A}|\Y_{0}^{\bullet_a}|, \sum_{a,b\in A}|\Y_{0}^{\bullet_a\bullet_b}|, \sum_{a,b\in A}| \Y_{0}^{\tddeuxa{$a$}{$b$}}\ \,|, \fan{\Y} _{\X;\alpha }, \fan{\tilde\Y} _{\X;\alpha},\fan{\X} _{\alpha}\Big)\fan{\Y,\tilde{\Y} }_{\X,\tilde{\X};\alpha}.
\end{align*}
By  ~\meqref{eq:mycont}, the above function $M$ is bounded.
Choosing again $\delta$ small enough, we obtain
$$\fan {\mathcal{M}\Y, \mathcal{M}\tilde{\Y}}_{\X;\alpha} \le \frac{1}{2}\fan{\Y, \tilde{\Y} }_{\X;\alpha }.$$
Applying the Banach fixed point theorem, there is a unique  $\Y\in B_{\delta}(\W, R)$ such that $\mathcal{M}\Y=\Y$, as required.
\end{proof}

%We have all the ingredients in place to expose the main result in this section.
\noindent We can now state our main result in this paper.

\begin{theorem}(Planarly branched universal limit theorem)
With the setting in  ~\meqref{eq:XYF}, let $\alpha\in (\frac{1}{4},\frac{1}{3}]$ and $\X\in \brpt$. Then there is a unique $\Y\in \cbrpxt$ such that
\begin{equation}
\Y_{t}=\Y_{0}+\displaystyle\int_0^t{F(\Y_{r})\cdot dX_{r}},\quad \forall t\in [0, T],
\mlabel{eq:bY}
\end{equation}	
that is,  ~\meqref{eq:RDE} has a unique solution on $[0, T]$.
\mlabel{thm:unilthm}
\end{theorem}

\begin{proof}
{\bf (Existence).} Let $\delta$ be as in Theorem~\mref{thm:leau}.
Notice that $\delta$ does not depend on the initial condition $Y_0$.
By Theorem~\mref{thm:leau}, we get a solution on $[0, \delta]$. Taking $Y_\delta$ as a new initial condition, we obtain a solution to ~\meqref{eq:RDE} on $[\delta, 2\delta]$ by Theorem~\mref{thm:leau} again.
Continuing this process, we conclude a solution $Y$ to  ~\meqref{eq:RDE} on $[0, T]$ after finite steps.

{\bf (Uniqueness).} Let $\tilde \Y\in \cbrpxt$
be another required one. Define
$$\sigma:=\sup\Big\{t\ge 0 \mid \Y_t=\tilde\Y_t \text{ on }[0, T] \Big\}.$$
Since $\Y$ and $\tilde \Y$ are continuous, we have $\Y_{\sigma}=\tilde\Y_{\sigma}$. Suppose for a contradiction that $\sigma <T$. Take $\varepsilon>0$ to be small enough.
Then
\begin{equation}
\Y_{\sigma+\varepsilon}\ne \tilde\Y_{\sigma+\varepsilon}.
\mlabel{eq:twoy}
\end{equation}
Since  $\Y$ and $\tilde\Y$ are in $\cbrpxt$ controlled by the same $\X$, it follows from Lemma~\mref{lem:jstab2} that
\begin{align*}
&\fan {\Y|_{[\sigma, \sigma+\varepsilon]}, \tilde\Y|_{[\sigma, \sigma+\varepsilon]}}_{\X; \alpha}\\
\le&\  C_{\alpha}\varepsilon^{\alpha}\Big(\sum_{i=1}^{d}\|F^i\|_{C_{b}^{3}}\Big)M\bigg (\varepsilon, \sum_{a\in A}|\Y_{\sigma}^{\bullet_a}|, \sum_{a,b\in A}|\Y_{\sigma}^{\bullet_a\bullet_b}|, \sum_{a,b\in A}|\Y_{\sigma}^{\tddeuxa{$a$}{$b$}}\ \,|, \fan{\Y|_{[\sigma, \sigma+\varepsilon]}} _{\X;\alpha },\\
&\hspace{3.6cm}  \fan{\tilde \Y|_{[\sigma, \sigma+\varepsilon]} } _{\X;\alpha }, \fan{\X}_{\alpha}\bigg) \, \fan{\Y|_{[\sigma, \sigma+\varepsilon]}, \tilde\Y|_{[\sigma, \sigma+\varepsilon]}}_{\X; \alpha}\\
\le&\ \frac{1}{2}\fan {\Y|_{[\sigma, \sigma+\varepsilon]}, \tilde\Y|_{[\sigma, \sigma+\varepsilon]}}_{\X;\alpha},\hspace{3.3cm} (\text{by $\varepsilon$ small enough})
\end{align*}
which implies $\fan {\Y, \tilde\Y}_{\X;\alpha} = 0$ and $\Y=\tilde\Y$ on $[\sigma, \sigma+\varepsilon]$, contradicting ~\meqref{eq:twoy}. Hence $\Y=\tilde\Y$ on $[0, T]$.
\end{proof}

\vskip 0.2in

\noindent
{\bf Acknowledgments.} This work is supported by the Natural Science Foundation of Gansu Province (25JRRA644), Innovative Fundamental Research Group Project of Gansu Province (23JRRA684) and Longyuan Young Talents of Gansu Province. The first author is grateful to Laboratoire de Math\'ematiques Blaise Pascal at  Universit\'e Clermont Auvergne for warm hospitality.

\noindent
{\bf Declaration of interests. } The authors have no conflicts of interest to disclose.

\noindent
{\bf Data availability. } Data sharing is not applicable as no new data were created or analyzed.

\end{document}